\documentclass{amsart}
\usepackage{amsmath,amscd,amssymb,amsthm}
\usepackage{amsfonts}
\usepackage{mathabx}
\usepackage[all]{xy}
\usepackage{mathtools}
\usepackage{enumerate}

\usepackage{hyperref}

\usepackage{adjustbox}

\usepackage[usenames, dvipsnames]{color}

\definecolor{NTNUblue}{RGB}{0,80,158}
\definecolor{NTNUbluesupport}{RGB}{62,98,138}
\definecolor{NTNUorange}{RGB}{239,129,20}

\newcommand{\into}{\hookrightarrow}

\newcommand{\xto}{\xrightarrow}

\newcommand{\F}{\mathbb{F}}

\newcommand{\Q}{\mathbb{Q}}

\newcommand{\Z}{\mathbb{Z}}

\newcommand{\BBB}{\mathbf{B}}

\newcommand{\DDD}{\mathbf{D}}

\newcommand{\DDDw}{\widetilde{\DDD}}

\newcommand{\DDDwr}{\DDDw^{\re}}

\newcommand{\DDDwl}{\DDDw^{\li}}

\newcommand{\SSS}{\mathbf{S}}

\newcommand{\TTT}{\mathbf{T}}

\newcommand{\li}{\mathbf{l}}

\newcommand{\re}{\mathbf{r}}

\newcommand{\Ah}{\mathcal{A}}
\newcommand{\Bh}{\mathcal{B}}
\newcommand{\Ch}{\mathcal{C}}

\newcommand{\Hh}{\mathcal{H}}

\newcommand{\Vh}{\mathcal{V}}

\newcommand{\Zh}{\mathcal{Z}}

\newcommand{\Hom}{\mathrm{Hom}}

\newcommand{\uHom}{\underline{\Hom}}

\newcommand{\oddeven}{\mathrm{oe}}

\newcommand{\op}{\mathrm{op}}

\newcommand{\HH}{\mathrm{HH}}

\newcommand{\id}{\mathrm{id}}

\newcommand{\bbb}{\bullet}

\newcommand{\ba}{\bar{a}}

\newcommand{\Cb}{\Ch^\bbb}

\newcommand{\Hb}{H^\bbb}

\newcommand{\Hba}{\Hb(\Ah)}

\newcommand{\Hbb}{\Hb(\Bh)}

\newcommand{\Kbb}{K^{\bbb}_{\bbb}}

\newcommand{\dee}{\partial}

\newcommand{\dd}{\delta}

\newcommand{\ddA}{\dd_{\Ah}}

\newcommand{\ddB}{\dd_{\Bh}}

\newcommand{\fA}{f^{\Ah}}

\newcommand{\fB}{f^{\Bh}}

\newcommand{\One}{\mathbf{1}}

\newcommand{\ma}{m^{\Ah}} 

\newcommand{\mb}{m^{\Bh}}

\newcommand{\mh}{m^H}

\newcommand{\mmm}{m} 

\newcommand{\pam}{\sigma} 

\newcommand{\ga}{\gamma_{\Ah}}

\newcommand{\gag}{\gamma_G}

\newcommand{\wtheta}{\widetilde{\vartheta}}

\newcommand{\tm}{\widetilde{m}}

\newcommand{\pr}{\mathrm{pr}} 

\newcommand{\ip}{(\iota p)} 

\newcommand{\ot}{\otimes} 

\newcommand{\nn}{\mu} 

\newcommand{\mmp}{m}

\newcommand{\bU}{\overline{U}}

\newcommand{\rr}{\rho} 

\newcommand{\brr}{\overline{\rho}} 

\newcommand{\wrr}{\widetilde{\rho}}

\newcommand{\car}{\theta} 

\newcommand{\vspan} {\mathrm{span}}

\theoremstyle{plain}

\newtheorem{theorem}{Theorem}[section]
\newtheorem{question}[theorem]{Question}
\newtheorem{prop}[theorem]{Proposition}
\newtheorem{proposition}[theorem]{Proposition}
\newtheorem{lemma}[theorem]{Lemma}

\newtheorem{corollary}[theorem]{Corollary}

\theoremstyle{definition}
\newtheorem{definition}[theorem]{Definition}
\newtheorem{defn}[theorem]{Definition}
\newtheorem{example}[theorem]{Example}

\newtheorem{notn}[theorem]{Notation}
\theoremstyle{remark}
\newtheorem{rem}[theorem]{Remark}
\newtheorem{remark}[theorem]{Remark}

\begin{document}

\title{$A_3$-formality for Demushkin groups at odd primes}

\author{Ambrus P\'al}
\address{Mathematical Institute, E\"{o}tv\"{o}s Lor\'{a}nd University, H-1117 Budapest, Hungary} 
\email{ambrus.pal@ttk.elte.hu}

\author{Gereon Quick}
\address{Department of Mathematical Sciences, NTNU, NO-7491 Trondheim, Norway}
\email{gereon.quick@ntnu.no}
\thanks{Both authors were partially supported by RCN Project No.\,313472 {\it Equations in Motivic Homotopy}, 
and the project \emph{Pure Mathematics in Norway} funded by the Trond Mohn Foundation.}


\begin{abstract} 
We study a weak form of formality for differential graded algebras, called $A_3$-formality, 
for the cohomology of pro-$p$ Demushkin groups at odd primes $p$. 
We show that the differential graded $\F_p$-algebras of continuous cochains of Demushkin groups 
with $q$-invariant not equal $3$  
are $A_3$-formal, 
whereas Demushkin groups with $q$-invariant $3$ are not $A_3$-formal. 
We prove these results by an explicit computation of the Benson--Krause--Schwede canonical class 
in Hochschild cohomology. 
\end{abstract}
\subjclass{20J06, 12G05, 16E40, 18N40, 20E18, 55S30.}

\maketitle

\section{Introduction}

Let $F$ be a field and let $G_F$ denote its absolute Galois group. 
For a prime number $p$, 
let $\Cb(G_F,\F_p)$ denote the differential graded algebra of inhomogeneous continuous cochains 
of $G_F$ with coefficients in the constant discrete $G_F$-module $\F_p$. 
In \cite{HW}, Hopkins and Wickelgren showed that all triple Massey products of local and global fields at the prime $2$ 
vanish whenever they are defined. 
Since triple Massey products are the first obstruction to formality, 
Hopkins--Wickelgren therefore asked in \cite[Question 1.4]{HW} 
whether $\Cb(G_F,\F_2)$ is formal, 
i.e., whether there is a zigzag of quasi-isomorphisms of differential graded algebras 
between $\Cb(G_F,\F_2)$ and its cohomology $\Hb(G_F,\F_2)$. 
%
%
However, Positselski showed in \cite[Section 9.11]{PoMoscow} and \cite[\S 6]{Po} 
that $\Cb(G_F,\F_2)$ is not formal in general for  local fields. 
%
Then Harpaz--Wittenberg in \cite[Example A.15]{GMT} 
and more recently Merkurjev--Scavia in \cite[Theorem 1.6]{MS1} and \cite[Theorem 1.3]{MS2023} 
provided examples which 
show that the second obstruction to formality is not trivial in general, 
i.e., not all fourfold Massey products are defined when the neighbouring cup-products vanish. 

The question whether a differential graded algebra (dga) is formal as a dga 
is equivalent to whether it is formal as an $A_{\infty}$-algebra. 
%
%
We may then ask the weaker question whether a dga $\Ah$ is formal as an $A_3$-algebra, 
i.e., whether there is a there is a quasi-isomorphism of $A_3$-algebras 
between $\Ah$ and its cohomology algebra $\Hba$  
(see Section \ref{sec:A3-algebras} for a definition of $A_3$-algebras). 
The purpose of this paper is to study $A_3$-formality for the dga $\Cb(G,\F_p)$ 
of pro-$p$ Demushkin groups at odd primes. 

\begin{defn}\label{def:Demushkin}
Let $p$ be a prime number and let $G$ be a pro-$p$-group. 
Then $G$ is called a \emph{Demushkin group} if 
\begin{enumerate}
\item[(1)] $\dim_{\F_p}H^1(G,\F_p) < \infty$, 

\item[(2)] $\dim_{\F_p}H^2(G,\F_p) = 1$, 

\item[(3)] the cup product $H^1(G,\F_p) \times H^1(G,\F_p) \to H^2(G,\F_p)$ is a non-degenerate bilinear form. 
\end{enumerate}
\end{defn}

The only finite Demushkin group is $\Z/2$ and $\Cb(\Z/2,\F_2)$ is known to be intrinsically formal. 
Every Demushkin group is finitely presented as it can be topologically presented with $\dim_{\F_p} H^1(G, \F_p)$ 
number of generators and just one relation. 
In fact, by \cite{Demushkin}, \cite{Labute} and \cite{SerreDem}, a pro-$p$ Demushkin group for an odd prime number $p$ 
is completely characterised by invariants $d$ and $q = p^f$ with $f\ge 1$ as follows. 
The pro-$p$ group $G$ has an even number $d \ge 2$ of generators $x_1,\ldots,x_d$ subject to the single relation 
\begin{align*}
1 = x_1^q [ x_1,x_2] [x_3,x_4] \cdots [x_{d-1}, x_d] 
\end{align*}
where $[x,y] = x^{-1}y^{-1}xy$ denotes the commutator of elements $x,y \in G$. 
%

%
Our main result is the following theorem, 
see Theorems \ref{thm:Demushkin_groups_are_formal} and \ref{thm:Demushkin_p_equal_3}: 

\begin{theorem}\label{thm:Demushkin_groups_A3_formality_intro}
Let $p$ be an odd prime and let $G$ be a pro-$p$ Demushkin group with 
an even number of generators. 
For $q=p=3$, $\Cb(G,\F_3)$ is not $A_3$-formal. 
For $q =0$ or $q \ge 5$, $\Cb(G,\F_p)$ is $A_3$-formal.  
\end{theorem}


Demushkin groups are Poincar\'e groups of dimension two and play an important role for example in number theory since the maximal pro-$p$ quotients of absolute Galois groups of local fields that contain a primitive $p$-th root of unity are Demushkin groups 
or are trivial (see for example \cite[Th\'eor\`eme 4.2]{SerreDem}). 
Demushkin groups are fundamental building blocks of the class of elementary type pro-$p$ groups in the sense of Efrat. 
These groups are defined inductively as the class of pro-$p$ groups that includes finitely generated free pro-$p$ groups and Demushkin groups 
and is closed under taking free pro-$p$ products and semi-direct products with $\Z_p$. 
For the latter one requires that all groups are equipped with an orientation (see \cite[Section 3]{EfratDem} and e.g.~\cite[Section 4]{MPQT}). 
For positive results on Efrat's elementary type conjecture see \cite{Efrat_pro-p}, \cite{Efrat_Hasse}, \cite[Chapter 12]{FriedJarden}. 
Demushkin groups also arise as pro-$p$ completions of fundamental
groups of compact surfaces $\Sigma$ of genus $g \ge 1$ when $\Sigma$ is orientable 
and $g \ge 2$ when $\Sigma$ is not orientable (see also \cite{LLMS}). 


\begin{example}\label{example:realizable_Demushkin_group_intro}
Let $p=3$ and let $G$ be the pro-$p$ group with generators $x_1,x_2,x_3,x_4$ subject to the single relation $1 = x_1^3 [ x_1,x_2] [x_3,x_4]$. 
Following \cite{EfratDem} and \cite[page 254]{Koenigsmann}, the group $G$
is realisable as the maximal pro-$3$ Galois group $G_F(3)$ of the field $F=\Q_3(\zeta_3)$ 
where $\zeta_3$ is a root of unity of order $3$. 
Thus, Theorem \ref{thm:Demushkin_groups_A3_formality_intro} shows that 
$\Cb(G_F,\F_p)$ is not $A_3$-formal for $F=\Q_3(\zeta_3)$. 
\end{example}

\begin{example}\label{example:semidirect_product_at_p=3_is_not_A3-formal_intro} 
The pro-$3$-group with generators $x_1, x_2$ and relation $x_1^3[x_1,x_2]=1$ is isomorphic to 
the semi-direct product $\Z_3 \rtimes_{\car} \Z_3$ 
where $\car \colon \Z_3 \to 1 + 3\Z_3$ is the cyclotomic character. 
Theorem \ref{thm:Demushkin_groups_A3_formality_intro} thereby shows that 
$\Z_3 \rtimes_{\car} \Z_3$ is not $A_3$-formal even though $\Z_3$ is intrinsically $A_{\infty}$-formal. 
This example demonstrates that $A_3$-formality does not distribute over semi-direct products of pro-$p$ groups in general. 
However, we do not know of such a counterexample for $p\ge 5$. 
\end{example}


\begin{example}
For $p=3$ and $q=3^f$ with $f \ge 1$, let $G = \Z_3 \rtimes_{\car} \Z_3$ with cyclotomic character $\car \colon \Z_3 \to 1 + q\Z_3$. 
Then the cohomology algebra $\Hb(G,\F_3)$ is isomorphic to the exterior algebra over $\F_3$ in two generators in degree one. 
Theorem \ref{thm:Demushkin_groups_A3_formality_intro} shows that the differential graded algebras for $q=3$ and $q=3^f$ with $f\ge 2$ 
are not quasi-isomorphic as $A_3$-algebras. 
\end{example}



\subsection*{Relation to the Massey vanishing conjecture of Min\'a\v c--T\^an}
Before we outline the proof of  Theorem \ref{thm:Demushkin_groups_A3_formality_intro}, 
we describe the relation of our work to the Massey vanishing conjecture. 
%
Min\'a\v c and T\^an conjectured in \cite[Conjecture 1.1]{MT0} that, for every field $F$ and prime $p$, 
$G_F$ {\em satisfies $n$-Massey vanishing} with respect to $p$, i.e., 
all $n$-fold Massey products of elements in $H^1(G_F,\F_p)$ vanish whenever they are defined. 
%
By the work of Matzri \cite{Matzri}, Efrat--Matzri \cite{EM17} and Min\'a\v c--T\^an \cite{MT1}, 
all fields satisfy triple Massey vanishing with respect to all primes. 
In \cite{HarpazWittenberg}, Harpaz--Wittenberg showed that number fields satisfy $n$-Massey vanishing with respect to all primes. 
More recently, Merkurjev--Scavia proved in \cite{MS2} that all fields satisfy fourfold Massey vanishing with respect to $p=2$. 
Other cases of the conjecture have been proven in \cite{PQ}, \cite{PSz} and \cite{Quadrelli}. 
The vanishing of Massey products has concrete consequences for the structure of the Zassenhaus filtration of an absolute Galois 
and thereby led to new examples of profinite groups which are not absolute Galois groups of a field (see for example \cite{MT1, MT2}). 


In \cite[Definition 4.5]{MT0bis}, Min\'a\v c and T\^an formulate the following related property. 
Let $G$ be a profinite group and $p$ be a prime number. %
Then $G$ is said to have the 
{\em cup-defining $n$-fold Massey product property} (with respect to $\F_p$) 
if for every $\chi_1, \ldots, \chi_n \in H^1(G,\F_p)$ with 
$0 = \chi_1 \cup \chi_2 = \chi_2 \cup \chi_3 = \cdots = \chi_{n-1} \cup \chi_n$ 
the $n$-fold Massey product $\langle \chi_1, \ldots, \chi_n \rangle$ is defined.  
For $n\ge 4$, this is a non-trivial condition, 
and, in \cite[Remark 4.4]{MT0bis}, Min\'a\v c--T\^an show that not all pro-$p$ groups 
have the cup-defining $n$-fold Massey product property with respect to $\F_p$.  
In \cite[Question 4.2]{MT0bis}, Min\'a\v c--T\^an ask whether every  
Galois group of a maximal $p$-extension of a field $F$ containing a primitive $p$-th root of unity 
has the cup-defining $n$-fold Massey product property with respect to $\F_p$ 
(see also \cite[Section 8]{MTE}). 
They show that, for two pro-$p$ groups $G_1$ and $G_2$, 
the free pro-$p$ product $G_1 \ast G_2$ 
has the cup-defining $n$-fold Massey product property with respect to $\F_p$ 
if and only if both $G_1$ and $G_2$ do. 
Moreover, in \cite[Proposition 4.1]{MT0bis}, Min\'a\v c--T\^an prove that pro-$p$ Demushkin groups 
have the cup-defining $n$-fold Massey product property with respect to $\F_p$. 
Together with their work in \cite{MT2}, this implies that pro-$p$ Demushkin groups 
have the following stronger property. 

We say that a profinite group $G$ satisfies {\em strong} $n$-Massey vanishing with respect to $p$ 
if for every $\chi_1, \ldots, \chi_n \in H^1(G,\F_p)$ with 
$0 = \chi_1 \cup \chi_2  
= \cdots = \chi_{n-1} \cup \chi_n$ 
the $n$-fold Massey product $\langle \chi_1, \ldots, \chi_n \rangle$ is defined and {\em vanishes}. 
For $n\ge 4$, this is a strictly stronger condition than $n$-Massey vanishing. 
By the work of Min\'a\v c--T\^an in \cite[Proposition 4.1]{MT0bis} and \cite[Theorem 4.3]{MT1}, 
pro-$p$ Demushkin groups satisfy strong $n$-Massey vanishing with respect to $p$ and all $n \ge 3$.   
An independent proof that pro-$p$ Demushkin groups satisfy strong $n$-Massey vanishing for all $n \ge 3$ 
was given by P\'al--Szab\'o in \cite[Theorem 3.5]{PSz}. 
Moreover, by \cite[Theorem 1]{MMRT}, the absolute Galois groups of number fields which do not contain a primitive $p$th root of unity 
satisfy strong $n$-Massey vanishing with respect to $p$ for all $n\ge 3$.  
Strong vanishing of triple Massey products is a necessary condition for the $A_3$-formality of $\Cb(G,\F_p)$. 
We note, however, that $A_3$-formality is a significant strengthening of the vanishing of the Massey product obstructions  
since $A_3$-formality requires that a specific element in the triple Massey product of elements in $\Hb(G, \F_p)$ vanishes 
and that defining systems of triple Massey products can be chosen compatibly. 
%

\subsection*{New questions} 
Quadrelli showed in \cite{Quadrelli} that elementary type pro-$p$ groups satisfy strong $n$-Massey vanishing. 
Theorem \ref{thm:Demushkin_groups_A3_formality_intro} implies that elementary type pro-$3$ groups are not $A_3$-formal in general. 
For $p\ge 5$, we do not know whether there is an obstruction to $A_3$-formality and may ask the following 

\begin{question}\label{question:ETp_intro}
Let $p\ge 5$ be a prime number. Are all elementary type pro-$p$ groups $A_3$-formal? 
\end{question}

We note that free pro-$p$ groups are $A_{\infty}$-formal (see Proposition \ref{prop:free_is_formal}). 
We can also show that the free pro-$p$ product of pro-$p$ groups which are Koszul and $A_3$-formal 
is again $A_3$-formal. 
Moreover, by \cite[Theorem A]{MPQT}, the $\F_p$-cohomology algebra of elementary type pro-$p$ groups is Koszul. 
Hence, the only case missing for an answer to Question \ref{question:ETp_intro} 
is the one of semi-direct products. 
Based on the results of \cite{MMRT}, we also ask the following 

\begin{question}\label{question:number_fields_intro}
Let $F$ be a number field which does not contain a primitive $p$th root of unity. 
Is $\Cb(G_F,\F_p)$ then $A_3$-formal? 
\end{question}


\subsection*{Outline of the proof} 
We now give a brief outline of the proof of  Theorem \ref{thm:Demushkin_groups_A3_formality_intro}. 
While $n$-Massey vanishing for Demushkin groups is a direct consequence of the non-degeneracy of the cup product, 
showing $A_3$-formality is much more involved. 
%
Let $p$ be an odd prime. 
For every pro-$p$ Demushkin group $G$ we compute the {\it canonical class} $\gamma_G$ of the differential graded algebra $\Cb(G,\F_p)$ 
in the Hochschild cohomology group $\HH^{3,-1}(\Hb(G,\F_p))$. 
The canonical class of a dga was introduced by Benson--Krause--Schwede in \cite{BKS} in the context of the realizability of modules over Tate cohomology. 
It follows from the general theory of $A_3$-algebras that the class $\gamma_G$ 
vanishes if and only if $\Cb(G,\F_p)$ is $A_3$-formal. 
In fact, 
the canonical class $\gamma_G$ is the Hochschild cohomology class of the homotopy associator 
on $\Hb(G,\F_p)$ 
which is needed to construct a lift $\Hb(G,\F_p) \to \Cb(G,\F_p)$ of the identity on $\Hb(G,\F_p)$ as $A_3$-algebras. 
In Theorem \ref{thm:Demushkin_two_generators}, we first compute the canonical class for $q \ne 3$ and $d=2$ 
since the proof in the case of just two generators is significantly simpler while still demonstrating the principal ideas. 
We then show that $\gamma_G$ is non-trivial for $q=3$ and all even $d \ge 2$ in 
Theorem \ref{thm:Demushkin_p_equal_3}. 
The failure of the vanishing of $\gamma_G$ in this case relies on the fact that a certain group 
homomorphism cannot be lifted since $\binom{3}{3}=1$, 
whereas $\binom{q}{3}=0$ in $\F_p$ when $q =p^f$ and $f\ge 2$ for $p=3$. 
We give an alternative proof of Theorem \ref{thm:Demushkin_p_equal_3} in Remark \ref{rem:p3_Demushkin_triple_Massey_still_vanishes} 
and explain how the defining system of the canonical class can be modified to 
show the vanishing of the corresponding triple Massey product. 
We then provide the proof of Theorem \ref{thm:Demushkin_groups_are_formal} which is about the case $q \ne 3$ and $d\ge 4$ 
which is much more involved and relies on explicit computations both in group cohomology using results of Dwyer 
and in Hochschild cohomology. 

\begin{remark}\label{rem:Demushkin_groups_are_Koszul_intro}
The proof of Theorem \ref{thm:Demushkin_groups_A3_formality_intro} relies on the fact that the quadratic algebra $\Hb(G,\F_p)$ for a pro-$p$ Demushkin group is Koszul.  
The latter is known by the work of Min\'a\v c--Pasini--Quadrelli--T\^an in \cite[Theorem 5.2]{MPQT}. 
In fact, they prove the stronger result that $\Hb(G,\F_p)$ is a PBW-algebra which implies that $\Hb(G,\F_p)$ is Koszul. 
We provide an alternative proof that $\Hb(G,\F_p)$ is Koszul in Section \ref{sec:Demushkin_groups_are_Koszul} 
for completeness and convenience of the reader. 
We recall that Positselski and Voevodsky conjectured that all Galois cohomology algebras are Koszul \cite[\S 0.1, page 128]{PoGalois}. 
Further results on the Koszulity of Galois cohomology algebras we refer to \cite{MPPT}, \cite{MPQT}, \cite{MPQT2}, \cite{PQ}, and \cite{PoGalois}. 
\end{remark}

\begin{remark}
We note that the case $q=3$ provides examples of differential graded algebras arising from group cohomology 
which are Koszul and for which all $n$-fold Massey products which are defined vanish but which are not $A_3$-formal. 
\end{remark}

\begin{remark}\label{rem:p_being_even_or_odd_intro}
In the proof of Theorem \ref{thm:Demushkin_groups_A3_formality_intro} 
we make frequent use of the assumption that $p$ is odd. 
We do not know whether pro-$2$ Demushkin groups are $A_3$-formal or not. 
For $q \ge 5$, we do not know for which integer $n > 3$ pro-$p$ Demushkin groups are not $A_n$-formal with respect to $p$.  
That such an $n$ exists follows from Positselski's work in \cite[\S 6]{Po}. 
\end{remark}


\subsection*{Contents} 
We intended to write the paper as self-contained as possible since we found it challenging to find proofs for all the results we use in the literature. 
We hope that the reader will appreciate the additional effort. 
%
In Section \ref{sec:A3-algebras} we define $A_3$-algebras and morphisms between them. 
In Section \ref{sec:A3_formality} we show that every differential graded algebra has a minimal $A_3$-model 
which is unique up to quasi-isomorphisms of $A_3$-algebras 
and introduce the notion of $A_3$-formal algebras. 
In Section \ref{sec:A3_formal_and_Massey} we discuss the relationship between $A_3$-algebra structures and Massey products. 
We also provide a simple example of a differential graded algebra which is not $A_3$-formal 
but for which all triple Massey products which are defined vanish.  
%
In Section \ref{sec:A3_and_HH} we recall graded Hochschild cohomology groups of graded algebras 
and introduce the canonical class of Benson--Krause--Schwede in Hochschild cohomology. 
We then show that a differential graded algebra $\Ah$ over a field is $A_3$-formal if and only if 
the canonical class of $\Ah$ is trivial. 
In Section \ref{sec:A3_and_Koszul} we recall Koszul algebras 
and show in Section \ref{sec:canonical_class_Koszul} that the Koszul complex allows for a simpler construction of 
the canonical class of a differential graded algebra whose cohomology algebra is Koszul. 
In Section \ref{sec:group_coh} we recall 
Dwyer's theorem on Massey products in group cohomology 
and its consequences for profinite groups whose cohomology algebra is Koszul. 
In Section \ref{sec:Demushkin_groups_are_Koszul}, we provide a proof of the fact that the cohomology algebra of a Demushkin group is Koszul. 
%
In Section \ref{sec:A3_formality_for_Demushkin_groups} we formulate our main results on Demushkin groups 
and construct a concrete map which represents the canonical class in Section \ref{sec:Demushkin_canonical_class}. 
We then prove Theorem \ref{thm:Demushkin_p_equal_3} on the case $q=3$ in Section \ref{sec:Demushkin_q=3}. 
This yields the first part of Theorem \ref{thm:Demushkin_groups_A3_formality_intro}. 
In Section \ref{sec:proof_of_Dem_thm} we provide the proof of Theorem \ref{thm:Demushkin_groups_are_formal} for the case $q \ne 3$ 
and thereby finish the proof of Theorem \ref{thm:Demushkin_groups_A3_formality_intro}. 

\subsection*{Acknowledgements} 
We are very grateful to Jan Min\'a\v c for helpful comments. 
GQ would also like to thank Mads Hustad Sandøy for valuable conversations on $A_\infty$-algebras.


\section{$A_3$-algebras, $A_3$-formality and Massey products}

In this section we recall the theory of $A_3$-algebras needed for this paper. 

\subsection{$A_3$-algebras}
\label{sec:A3-algebras}

Let $\F$ be a field. 
For graded $\F$-vector spaces $\Ah$ and $\Bh$, we denote by $\uHom(\Ah,\Bh)$ the set of graded $\F$-linear maps $\Ah \to \Bh$. 
For $j\in \Z$ and a graded $\F$-vector space $\Ah$, we write $\Ah[j]$ for the graded vector space given in degree $i$ by $\Ah[j]^i = \Ah^{i+j}$. 
Tensor products $\otimes$ will be over $\F$, unless otherwise stated. 
We will follow the notation and sign convention of \cite{Keller}, i.e., for graded maps $f$ and $g$ and elements $x$, $y$ 
we have 
\begin{align*}
(f \otimes g) (x \otimes y) = (-1)^{|g||x|}f(x) \otimes g(y)
\end{align*}
where $|g|$ and $|x|$ denote the degrees of $g$ and $x$, respectively.

\begin{definition}\label{def:An_algebra}
Let $\Ah = \oplus_{i \ge 0} \Ah^i$ be a non-negatively graded $\F$-vector space with $\Ah^0=\F$. 
Then $\Ah$ is called an \emph{$A_3$-algebra over $\F$} if, for $i=1,2,3$, there are graded $\F$-linear maps  
\begin{align*}
m_i \colon \Ah^{\otimes i} \to \Ah[2-i]
\end{align*}
satisfying the following relations:  
We have $m_1m_1=0$, i.e., $(\Ah,m_1)$ is a cochain complex. 
%
We have 
\begin{align}\label{eq:m1m2}
m_1m_2  = m_2(m_1 \otimes \One + \One \otimes m_1)
\end{align} 
as maps $\Ah^{\otimes 2} \to \Ah$, 
where $\One$ is the identity map of $\Ah$. 
Hence $m_1$ is a graded derivation with respect to the multiplication $m_2$. 
We note that the sign rule implies that we have 
\begin{align*}
m_2(m_1 \otimes \One + \One \otimes m_1)(x \otimes y) = m_1(x)\otimes y + (-1)^{|x|}x \otimes m_1(y)
\end{align*}
since $m_1$ has degree $1$. 
%
For the map $m_3$ we require that 
\begin{equation}\label{eq:m1m2m3}
\begin{aligned}
& m_2(\One \otimes m_2 - m_2 \otimes \One) \\
= ~ & m_1m_3 + m_3(m_1 \otimes \One \otimes \One + \One \otimes m_1 \otimes \One + \One \otimes \One \otimes m_1)
\end{aligned}
\end{equation}
in $\uHom(\Ah^{\otimes 3}, \Ah)$. 
Hence, $m_2$ is associative up to homotopy. 
However, there is no further coherence condition on the homotopy of the associator.  
\end{definition}

\begin{example}
Every graded $\F$-algebra is an $A_3$-algebra with trivial $m_1$ and $m_3$.    
Every differential graded algebra over $\F$ is an $A_3$-algebra with $m_i=0$ for all $i\ge 3$. 
\end{example}

\begin{definition}\label{def:morphism_of_A3_algebra}
Let $\Ah$ and $\Bh$ be $A_3$-algebras over $\F$. 
A \emph{morphism of $A_3$-algebras} $f \colon \Ah \to \Bh$ is 
a triple $(f_1,f_2,f_3)$ of graded $\F$-linear maps $f_i \colon \Ah^{\otimes i} \to \Bh[1-i]$ 
satisfying the following relations: 
We have $f_1\ma_1 = \mb_1f_1$, i.e., $f_1$ is a morphism of cochain complexes; 
we have 
\begin{align}\label{eq:relation_f1m2}
f_1\ma_2  - \mb_2(f_1 \otimes f_1) = \mb_1f_2 + f_2(\ma_1\otimes \One_{\Ah} + \One_{\Ah} \otimes \ma_1), 
\end{align}
i.e., $f_1$ commutes with multiplication up to homotopy given by $f_2$; 
and we have 
\begin{equation}\label{eq:relation_f3} 
\begin{aligned}
& ~  \mb_1f_3 + \mb_2(f_1 \otimes f_2 - f_2 \otimes f_1) +  \mb_3(f_1 \otimes f_1 \otimes f_1) \\ 
= & ~ f_1\ma_3 + f_2 (\ma_2 \otimes \One_{\Ah} - \One_{\Ah} \otimes  \ma_2) \\
& ~ + f_3(\ma_1 \otimes \One_{\Ah}^{\otimes 2} + \One_{\Ah} \otimes \ma_1 \otimes \One_{\Ah} + \One_{\Ah}^{\otimes 2} \otimes \ma_1). 
\end{aligned} 
\end{equation} 
A morphism of $A_3$-algebras $f$ is called \emph{strict} if $f_2$ and $f_3$ are trivial. 
The identity morphism is the strict morphism with $f_1 = \id$. 
A morphism of $A_3$-algebras $f$ is called a \emph{quasi-isomorphism} if $f_1$ is a quasi-isomorphism of underlying cochain complexes. 
The composition of two $A_3$-morphisms $f \colon \Bh \to \Ch$ and $g \colon \Ah \to \Bh$ is given by
\begin{equation*}
\begin{aligned}
(f \circ g)_1 & = f_1 \circ g_1 \\
(f \circ g)_2 & = f_2 \circ (g_1 \ot g_1) + f_1 \circ g_2 \\
(f \circ g)_3 & =  f_3 \circ (g_1 \ot g_1 \ot g_1) - f_2 \circ (g_2 \ot g_1 - g_1 \ot g_2) + f_1 \circ g_3.
\end{aligned}
\end{equation*}
\end{definition}


\begin{lemma}\label{lemma:A3-isom}
Let $f \colon \Ah \to \Bh$ be a morphism of $A_3$-algebras. 
Then $f$ is an isomorphism if and only if $f_1$ is an isomorphism. 
\end{lemma}
\begin{proof}
Assume that $f_1$ is an isomorphism. 
We need to show that there is an $A_3$-inverse $g \colon \Bh \to \Ah$ of $f$. 
We set $g_1=f_1^{-1}$ to be the inverse of $f_1$. 
We then define $g_2 = - f_1^{-1} \circ (f_2(g_1\ot g_1))$ 
and 
\begin{align*}
g_3 = f_1^{-1} \circ ( - f_3 (g_1 \ot g_1 \ot g_1) + f_2 \circ (g_2 \ot g_1 - g_1 \ot g_2)). 
\end{align*}
One can then check that the required relations for the compositions are satisfied.  
\end{proof}


\subsection{$A_3$-formality}
\label{sec:A3_formality}

We adopt the following terminology from the theory of $A_{\infty}$-algebras: 

\begin{definition}
An $A_3$-algebra $\Hh$ is called \emph{minimal} if $m^{\Hh}_1=0$. 
A {\it minimal model} for an $A_3$-algebra $\Ah$ is a minimal $A_3$-algebra $\Hh$ 
together with a quasi-isomorphism of $A_3$-algebras $\Hh \to \Ah$. 
\end{definition}

\begin{lemma}\label{lemma:A3-uniqueness}
Let $f \colon \Ah \to \Bh$ be a morphism between minimal $A_3$-algebras. 
Then $f$ is an isomorphism if and only if it is a quasi-isomorphism. 
\end{lemma}
\begin{proof}
This follows from Lemma \ref{lemma:A3-isom} and the fact that $f_1$ is an isomorphism if and only if it is a quasi-isomorphism when both $\Ah$ and $\Bh$ are minimal $A_3$-algebras.  
\end{proof}

Recall that a differential graded algebra $\Ah$ is called {\it connected} 
if $\Ah^i = 0$ for $i < 0$ and $\Ah^0=\F$. 
The next theorem is a variation of a theorem due to Kadeishvili in \cite{Kadeishvili82} (see also \cite{Kadeishvili23}). 

\begin{theorem}[Kadeishvili]\label{thm:A_3_Kadeishvili}
Let $\Ah$ be a connected differential graded algebra with cohomology algebra $\Hba$. 
Then  $\Hba$ can be equipped with the structure of an $A_3$-algebra under which it is a minimal model for $\Ah$ 
and such that the multiplication $\mh_2$ on $\Hba$ is induced by $\ma_2$. 
This structure is unique up to isomorphism of $A_3$-algebras. 
\end{theorem}

This is a well-known result for $A_{\infty}$-algebras (see for example \cite[Proof of Theorem 7.2.2]{Witherspoon}). 
We provide a sketch of the proof since it provides us with constructions that will be used later. 
Moreover, we could not find a proof of the uniqueness statement for $A_3$-algebras in the literature. 

\begin{proof}[Proof of Theorem \ref{thm:A_3_Kadeishvili}.]
%
We need to define a graded $\F$-linear map $\mh_3 \colon \Hba^{\otimes 3} \to \Hba$ such that $\mh_1 = 0, \mh_2, \mh_3$ satisfy the required relations 
together with an $A_3$-algebra morphism $f \colon \Hba \to \Ah$. 
We choose $f_1 \colon \Hba \to \ker \ma_1$ to be an $\F$-linear graded map which induces the identity on $\Hba$.   
Since $f_1$ is multiplicative on cohomology, we can find a graded $\F$-linear map  
$f_2 \colon \Hba \otimes \Hba \to \Ah$ of degree $-1$ satisfying  
\begin{align}\label{eq:f2_condition_in_minimal_thm}
\ma_1f_2 = f_1\mh_2  - \ma_2(f_1 \otimes f_1). 
\end{align}
See Remark \ref{rem:Merkulov_structure} for a formula for $f_2$. 
Now we define a graded $\F$-linear map $\Phi_3 \colon \Hba^{\otimes 3} \to \Ah[-1]$ by 
\begin{align}\label{eq:def_of_Phi3_minimal}
\Phi_3 = \ma_2(f_1 \otimes f_2 - f_2 \otimes f_1) - f_2(\mh_2 \otimes \One - \One \otimes \mh_2). 
\end{align}
We check that $\Phi_3$ has image in the cocycles of $\Ah$ 
and hence induces a graded map $[\Phi_3] \colon \Hba^{\otimes 3} \to \Hba[-1]$. 
We set $\mh_3 := [\Phi_3]$. 
By construction, the difference $f_1\mh_3 - \Phi_3$ has image in the coboundaries of $\Ah$. 
Thus, we can find a graded $\F$-linear map $f_3 \colon \Hba^{\otimes 3} \to \Hba[-1]$ such that 
\begin{align*}
f_1\mh_3 - \Phi_3 = \ma_1 f_3. 
\end{align*}
See Remark \ref{rem:Merkulov_structure} for a formula for $f_3$. 
By definition of $\Phi_3$ and $\mh_3$, relation \eqref{eq:relation_f3} is satisfied where we use that $\mh_1$ and $\ma_3$ are trivial.   
The uniqueness assertion follows from Lemma \ref{lemma:A3-uniqueness} and Theorem \ref{thm:proj_extends_to_A_3} below as follows:  
If $m'_3$ and $f' \colon \Hba \to \Ah$ is another choice which turns $\Hba$ into a minimal model for $\Ah$, 
then the composition of $A_3$-morphisms $p \circ f'$ is an isomorphism of $A_3$-algebras since $p \circ f'_1$ is the identity. 
\end{proof}


It remains to show that we can construct an $A_3$-algebra structure on $\Hba$ such that we can lift the projection $\Ah \to \Hba$ 
from a map of graded vector spaces to a morphism of $A_3$-algebras.

\begin{theorem}\label{thm:proj_extends_to_A_3}
Let $\Ah$ be a connected differential graded algebra 
and let $p \colon \Ah \to \Hba$ be a graded $\F$-linear projection onto the cohomology algebra of $\Ah$. 
Then there exists an $A_3$-algebra structure on $\Hba$ with $m_1=0$ and $m_2=\mh_2$ such that 
$p$ extends to a morphism of $A_3$-algebras $\Ah \to \Hba$. 
\end{theorem}
\begin{proof}
We choose a graded $\F$-linear map $\iota \colon \Hba \to \ker \ma_1$ which induces the identity on $\Hba$. 
To simplify the notation, we write $H \coloneqq \Hba$. 
Now we choose a homotopy $h$ from the identity $\One_\Ah$ to $\iota \circ p$, 
i.e., a graded $\F$-linear map $h \colon \Ah \to \Ah[-1]$ satisfying     
\begin{align}\label{eq:h_htpy}
\One_\Ah - \iota \circ p = \ma_1 \circ h + h \circ \ma_1, 
\end{align} 
in the following way:  
We write $H^j \subset \Ah^j$ for the image of the $j$th component of the injective map $\iota$, 
and identify $\Hba$ with $\oplus_j H^j \subset \Ah$ via $\iota$.  
We then identify $p = \pr_H \colon \Ah \to \Ah$ with the projection to $H$. 
Let $Z^j$ and $B^j$ denote the cocycles and coboundaries in $\Ah^j$, respectively. 
We have $Z^j = B^j \oplus H^j$. 
We choose a subspace $L^j \subset \Ah^j$ such that $\Ah^j = B^j \oplus H^j \oplus L^j$.  
We let $h \colon \Ah \to \Ah[-1]$ be the graded $\F$-linear map such that, for every $j$, 
$h^j \colon \Ah^j \to \Ah^{j-1}$ is the map which is trivial when restricted on $L^j \oplus H^j$ 
and equals $\left(\delta^{j-1}_{|L^{j-1}} \right)^{-1}$ when restricted to $B^j$, 
where $\delta^{j-1}_{|L^{j-1}} $ denotes the restriction of $\delta_\Ah = \ma_1$ to the subspace $L^{j-1} \subset \Ah^{j-1}$. 
It follows that the image of $h^j$ is $L^{j-1}$ and that $h^{j+1} \circ \delta_\Ah^{j} = \pr_{L^j}$ and $\delta_\Ah^{j-1} \circ h^j = \pr_{B^j}$.  
We define the graded $\F$-linear map $h_2 \colon \Ah^{\ot 2} \to \Ah^{\ot 2}[-1]$ by 
\begin{align*}
h_2 := \One \otimes h + h \otimes \ip 
\end{align*}
which is a homotopy between $\One \otimes \One$ and $\ip \otimes \ip$ as maps $\Ah^{\otimes 2} \to \Ah^{\otimes 2}$. 
We define the graded $\F$-linear map $\nn_3 \colon \Ah^{\otimes 3} \to \Ah^{\ot 2}$ by 
\begin{align*}
\nn_3 \coloneqq \ma_2 \ot \One- \One \otimes \ma_2  \colon \Ah^{\otimes 3} \to \Ah^{\ot 2}. 
\end{align*} 
We note that $\ma_2\nn_3=0$ 
since $\ma_2$ is associative. 
We then define the graded $\F$-linear map $m_3 \colon \Ah^{\otimes 3} \to H[-1]$ by 
\begin{align}\label{eq:def_of_m3_pmodel}
\mmp_3 := p \circ \ma_2 \circ h_2 \circ \nn_3 (\iota ^{\ot 3}) = p \ma_2 \left(h\ma_2\ot \ip - \One \ot h \ma_2 \right) \left(\iota ^{\ot 3}\right) 
\end{align}
where we use for the right-hand equality that $h$ vanishes on the image of $\iota$. 
Since $\mh_2$ is associative and $\mh_1=0$, relation \eqref{eq:m1m2m3} is satisfied and $\mmp_3$ turns $H$ into an $A_3$-algebra.

It remains to show that the projection map $p$ can be extended to a morphism of $A_3$-algebras. 
For $n=2$, we define the graded $\F$-linear map $p_2 \colon \Ah^{\otimes 2} \to H[-1]$ by 
\begin{align*}
p_2 := p \circ \ma_2 \circ h_2 = p \ma_2 (\One \otimes h + h \otimes \ip).
\end{align*}
We then get, using the definition of $\mh_2$ as $p \circ \ma_2(\iota \otimes \iota)$, 
\begin{align*}
 & ~ p_2(\ma_1 \otimes \One + \One \otimes \ma_1)  \\
= & ~ \big( p \ma_2 (\One \otimes h + h \otimes \ip) \big) (\ma_1 \otimes \One + \One \otimes \ma_1)  \\
= & ~ p \ma_2 (\ma_1 \ot h + \One \ot h\ma_1 + h\ma_1 \ot \ip + h \ot (\iota p \ma_1)). 
\end{align*}
Now we use \eqref{eq:h_htpy} to write $h\ma_1 = \One - \ip - \ma_1h$ and $p\ma_1=0$ to continue 
\begin{align*}
 & ~ p_2(\ma_1 \otimes \One + \One \otimes \ma_1)  \\
= & ~ p \ma_2 (\ma_1 \ot h + \One \ot (\One - \ip -\ma_1h) + (\One - \ip -\ma_1h) \ot \ip )  \\  
=  & ~ p \ma_2 (\One \ot \One) - p\ma_2(\ip \otimes \ip) \\
= & ~ p \ma_2 - \mh_2(p \otimes p)
\end{align*}
where we used that $p$ vanishes on coboundaries. 
Setting $d_2 := \ma_1 \ot \One + \One \otimes \ma_1 $, this reads 
\begin{align}\label{eq:pd2_relation}
pd_2 = p\ma_2 - \mh_2(p \ot p). 
\end{align}
%
For $n=3$, we first define 
\begin{align*}
h_3 := \One^{\ot 2} \otimes h + \One \ot h \otimes \ip + h \ot \ip^{\ot 2}
\end{align*} 
and check that 
\begin{align}\label{eq:h3_htpy_relation}
\One^{\ot 3} - \ip^{\ot 3} =h_3d_3 + d_3h_3
\end{align} 
where we write $d_3 := \One^{\ot 2} \otimes \ma_1 + \One \ot \ma_1 \ot \One + \ma_1 \ot \One^{\ot 2}$. 
Furthermore, we compute 
\begin{align*}
\nn_3d_3 = & ~ (\One \ot \ma_2 - \ma_2 \ot \One)(\One^{\ot 2} \otimes \ma_1 + \One \ot \ma_1 \ot \One + \ma_1 \ot \One^{\ot 2}) \\
= & ~ \One \ot \ma_2(\One \ot \ma_1) + \One \ot \ma_2(\ma_1 \ot \One) + \ma_1 \ot \ma_2(\One \ot \One) \\
 & ~ - \big( \ma_2(\One \ot \One) \ot \ma_1 + \ma_2(\One \ot \ma_1) \ot \One + \ma_2(\ma_1 \ot \One) \ot \One  \big) \\
= & ~ \One \ot \ma_1\ma_2(\One \ot \One) + \ma_1 \ot \ma_2(\One \ot \One) \\
& ~ - \big( \ma_2(\One \ot \One) \ot \ma_1 + \ma_1\ma_2(\One \ot \One) \ot \One  \big). 
\end{align*}
Thus, we get the relation 
\begin{align}\label{eq:n3d3_d2n3}
\nn_3d_3 = d_2\nn_3.
\end{align}
Moreover, we compute 
\begin{align*}
(p \ot p)(\nn_3 \circ h_3) = & ~ p\ma_2(\One \ot \One)\ot ph + p\ma_2(\One \ot h)\ot p\ip + p\ma_2(h \ot \ip) \ot p\ip \\
 & ~ - \big( p \ot p\ma_2(\One \ot h) + p \ot p\ma_2(h \ot \ip) + ph \ot p\ma_2(\ip \ot \ip)  \big) \\
= & ~ p\ma_2(\One \ot h)\ot p + p\ma_2(h \ot \ip) \ot p \\
 & ~ - \big( p \ot p\ma_2(\One \ot h) + p \ot p\ma_2(h \ot \ip)  \big) 
\end{align*}
where we use that $ph=0$ and $p\ip=p$. 
Using the definition of $p_2$ this shows 
\begin{align}\label{eq:pp_n3h3}
(p \ot p)(\nn_3 \circ h_3) = -(p \ot p_2 - p_2 \ot p).
\end{align}
Now we define the graded $\F$-linear map $p_3 \colon \Ah^{\otimes 3} \to H[-2]$ by 
\begin{align*}
p_3 := -p_2 \circ \nn_3 \circ h_3. 
\end{align*}
We check 
\begin{align*}
& p_3d_3 + p_2 \nn_3 =  -p_2  \nn_3  h_3 d_3 + p_2 \nn_3 \\ 
= ~ & - p_2\nn_3 \left(\One^{\ot 3} - \ip^{\ot 3} - d_3 h_3 \right) + p_2 \nn_3 ~ \text{using \eqref{eq:h3_htpy_relation}}  \\ 
= ~ &  p_2\nn_3 \ip^{\ot 3} - p_2\nn_3 + p_2\nn_3 d_3h_3 + p_2 \nn_3 
=  p_2\nn_3 \ip^{\ot 3} + p_2\nn_3 d_3h_3 \\ 
= ~ &  \mmp_3 (p^{\ot 3}) + p_2d_2\nn_3h_3 ~ \text{using the definition of}~\mmp_3 ~\text{and}~ \eqref{eq:n3d3_d2n3}\\ 
= ~ &  \mmp_3 (p^{\ot 3}) + p\ma_2\nn_3h_3  - \mh_2(p\ot p) \nn_3h_3 ~ \text{using \eqref{eq:pd2_relation}} \\ 
= ~ &  \mmp_3 (p^{\ot 3}) - \mh_2(p\ot p) \nn_3h_3 ~ \text{using associativity:} ~ \ma_2 \nn_3 = 0 \\
= ~ &  \mmp_3 (p^{\ot 3}) + \mh_2(p\ot p_2 - p_2 \ot p)  ~ \text{using ~\eqref{eq:pp_n3h3}}. 
\end{align*}
Using the definition of $d_3$ and $\nn_3$, this shows that  
\begin{equation*}\label{eq:relation_p3} 
\begin{aligned}
 & ~ p_3(\ma_1 \otimes \One^{\ot 2} + \One \ot\ma_1 \ot \One + \One^{\ot 2} \ot \ma_1) + p_2(\ma_2 \ot \One - \One \ot \ma_2)\\
= & ~ \mmp_3(p \ot p \ot p) + \mh_2(p \ot p_2 - p_2 \ot p). 
\end{aligned}
\end{equation*}
Thus, \eqref{eq:relation_f3} holds and $(p,p_2,p_3)$ defines a morphism of $A_3$-algebras $\Ah \to \Hba$.  
\end{proof}

\begin{remark}\label{rem:Merkulov_structure} 
We note that the map $\mmp_3$ defined in the proof of Theorem \ref{thm:proj_extends_to_A_3} corresponds to the map of Merkulov's explicit construction of the minimal $A_{\infty}$-model in  \cite{Merkulov}. 
With the above notation 
 the graded $\F$-linear map $m_3 \colon (\Hba)^{\otimes 3} \to \Ah[-1]$ in \cite{Merkulov} is defined  by 
\begin{align*}
m_3 =  p \circ \ma_2\left((h \circ \ma_2) \otimes \One - \One \otimes ((h \circ \ma_2)\right)(\iota^{\ot 3}) 
\end{align*}
which is the map $\mmp_3$ of  the proof of Theorem \ref{thm:proj_extends_to_A_3}. 
We can then define  graded $\F$-linear maps $f_2 = - h \circ \ma_2(\iota^{\ot 2})$ and  $f_3 = - h \circ \ma_2((h \circ \ma_2) \ot \One - \One \ot (h \circ \ma_2))(\iota^{\ot 3})$. 
This defines a morphism of $A_3$-algebras $f \colon \Hba \to \Ah$ which turns $\Hba$ into a minimal model for $\Ah$. 
\end{remark}

Recall that a differential graded algebra is called ($A_{\infty}$-) {\it formal} if  its minimal $A_{\infty}$-model can be chosen such that $m_i=0$ for all $i \ge 3$. 
We will use the following weaker notion: 

\begin{defn}\label{def:Ainfty_formal}
Let $\Ah$ be a connected differential graded algebra over $\F$ with cohomology algebra $\Hba$.  
Then $\Ah$ is called \emph{$A_3$-formal} if its minimal $A_3$-model can be chosen such that $\mh_3=0$. 
\end{defn}


\begin{remark}
It follows from Theorem \ref{thm:A_3_Kadeishvili} that 
$\Ah$ is $A_3$-formal if and only if there is a morphism of $A_3$-algebras $ f \colon \Hba \to \Ah$ which lifts the identity of $\Hba$  
where we consider $\Hba$ as an $A_3$-algebra with $\mh_1=0$ and $\mh_3=0$. 
\end{remark}


\begin{remark}\label{rem:A3_formality_morphism}
Let $\Ah$ be a connected differential graded algebra over $\F$.  %
The proof of Theorem \ref{thm:A_3_Kadeishvili} and relation \eqref{eq:relation_f3} tell us that in order to 
construct a morphism of $A_3$-algebras $ f \colon \Hba \to \Ah$ which lifts the identity of $\Hba$ 
we need to show that we can choose $f_1$ and $f_2$ such that  the map 
\begin{align*}
\Phi_3 = \ma_2(f_1 \otimes f_2 - f_2 \otimes f_1) - f_2(\mh_2 \otimes \One - \One \otimes \mh_2) 
\end{align*}
has image in the coboundaries of $\Ah$. 
Since $\F$ is a field, we can then find a graded $\F$-linear map $f_3$ such that relation \eqref{eq:relation_f3} is satisfied. 
\end{remark}


\subsection{$A_3$-formality and Massey products}
\label{sec:A3_formal_and_Massey}

In this section we show that $A_3$-formality implies the vanishing of triple Massey products in all degrees 
and that fourfold Massey products are defined whenever all neighbouring cup-products vanish. 
First we recall the definition of triple and fourfold Massey products. 

\begin{defn}\label{def:Massey_product} 
Let $\F$ be a field and let $\Ah$  be a differential graded $\F$-algebra with differential $\delta$ and cohomology algebra $\Hba$.  
For an element $a\in \Ah$ of degree $d = |a|$, we write $\bar{a} \coloneqq (-1)^{1+d}a$.  
Let $a_1, a_2,a_3$  be cohomology classes of degree $d_i \coloneqq |a_i|$ such that $a_1 \cdot a_2 = 0$ and $a_2 \cdot a_3 = 0$. 
For each $i$, we choose a cocycle $a_{i,i+1}$ which represents $a_i$ and cochains $a_{13}$ and $a_{24}$ such that 
$\dd a_{13} = \bar{a}_{12} \cdot a_{23}$ and $\dd a_{24} = \bar{a}_{23} \cdot a_{34}$. 
The set $M \coloneqq \{a_{12}, a_{23}, a_{34}, a_{13}, a_{24}\}$ is called 
a \emph{defining system} for the triple Massey product of $a_1$, $a_2$, $a_3$. 
The cochain $\bar{a}_{12}\cdot a_{24} + \bar{a}_{13}\cdot a_{34} \in \Ah^{d-1}$ is a cocycle where $d=d_1+d_2+d_3$. 
We write $\langle a_1, a_2, a_3 \rangle_M \in H^{d-1}(\Ah)$ for the corresponding  cohomology class. 
The {\it triple Massey product} $\langle a_1, a_2, a_3 \rangle$ is the set of all cohomology classes $\langle a_1, a_2, a_3 \rangle_M$ for all such defining systems $M$. 

If we are given four cohomology classes $a_1, a_2,a_3,a_4$ such that $a_1 \cdot a_2 = 0$, $a_2 \cdot a_3 = 0$ and $a_3 \cdot a_4 = 0$, 
we choose again, for each $i$, a cochain $a_{i,i+1}$ which represents $a_i$ and cochains $a_{i,i+2}$ such that 
$\dd a_{i,i+2} = \bar{a}_{i,i+1} \cdot a_{i+1,i+2}$. 
If we can choose cochains $a_{14}, a_{25}$ such that 
\begin{align*}
\delta(a_{14}) = \bar{a}_{12} \cdot a_{24} + \bar{a}_{13} \cdot a_{34} 
~ \text{and} ~ 
\delta(a_{25}) = \bar{a}_{23} \cdot a_{35} + \bar{a}_{24} \cdot a_{45}
\end{align*}
then the set $M \coloneqq \{a_{i,j} \}$ is called a \emph{defining system} for the Massey product of $a_1, a_2, a_3, a_4$.  
The cochain 
\begin{align*}
\bar{a}_{12}\cdot a_{25} + \bar{a}_{13}\cdot a_{35} + \bar{a}_{14}\cdot a_{45} \in \Ah^{d-2} ~ \text{with} ~ d=d_1+d_2+d_3+d_4
\end{align*}
 is a cocycle. 
We write $\langle a_1, a_2, a_3, a_4 \rangle_M \in H^{2(d-1)}(\Ah)$ for the corresponding  cohomology class. 
The {\it fourfold Massey product} $\langle a_1, a_2, a_3, a_4 \rangle$ is the set of all cohomology classes $\langle a_1, a_2, a_3, a_4 \rangle_M$ for all defining systems $M$. 

We say a Massey product is {\it defined} if the Massey product set is not empty, i.e., at least one defining system exists, 
and we say a Massey product {\it vanishes} if the Massey product set contains $0$, 
i.e., there is a defining system such that the corresponding cohomology class is trivial.    
\end{defn}

We now show the well-known fact (see for example  \cite[Section 3]{BMM})
that the $A_3$-structure of the minimal model 
of a differential graded algebra is closely related to triple Massey products. 

\begin{proposition}
\label{prop:Massey_product_and_minimal_model}
Let $\F$ be a field. 
Let $\Ah$ be a differential graded algebra over $\F$ with cohomology $\Hba$. 
Let $(\mh_1 = 0, \mh_2, \mh_3)$ be an $A_3$-algebra structure on $\Hba$ which turns $\Hba$ into a minimal model for $\Ah$. 
Let $a_1, a_2, a_3 \in \Hba$ be cohomology classes such that $a_1 \cdot a_2 = 0$ and $a_2 \cdot a_3 = 0$. 
Then $(-1)^{1+|a_2|}\mh_3(a_1, a_2, a_3)$ is an element in the Massey product set $\langle a_1,a_2,a_3 \rangle$. 
\end{proposition}
\begin{proof}
Let $f \colon \Hba \to A$ be an $A_3$-algebra morphism such that $f_1$ lifts the identity of $\Hba$. 
%
For $\Phi_3$ as in \eqref{eq:def_of_Phi3_minimal}, the assumption $a_1 \cdot a_2 = 0$ and $a_2 \cdot a_3 = 0$ 
implies  
\begin{align*}
\Phi_3(a_1, a_2, a_3) & = 
 (-1)^{|a_1|} f_1(a_1) \cdot f_2(a_2, a_3) 
-  f_2(a_1, a_2) \cdot f_1(a_3) 
\end{align*}
where we use the sign rule and that $f_n$ is a map of degree $1-n$. 
Since the map $f_1$ picks cocycle representatives, 
we may write $a_{12} \coloneqq f_1(a_1)$, $a_{23} \coloneqq f_1(a_2)$, and $a_{34} \coloneqq f_1(a_3)$.  
By \eqref{eq:relation_f1m2}, we have $\ma_1f_2 = - \ma_2(f_1 \otimes f_1)$ 
and hence 
\[
\ma_1f_2(a_1,a_2) = - a_{12} \cdot a_{23} ~ \text{and} ~ \ma_1f_2(a_2,a_3) = - a_{23} \cdot a_{34}. 
\]
Hence we may write $a_{13} \coloneqq (-1)^{|a_1|}f_2(a_1,a_2)$ and $a_{24} \coloneqq (-1)^{|a_2|}f_2(a_2,a_3)$. 
Then we get 
\begin{align*}
(-1)^{1+|a_2|}\Phi_3(a_1, a_2, a_3) & = 
(-1)^{1+|a_2|} \left( (-1)^{|a_1|} a_{12} \cdot (-1)^{|a_2|} a_{24} 
+ (-1)^{1+|a_1|} a_{13} \cdot a_{34} \right) \\
& = \ba_{12} \cdot a_{24} 
+ \ba_{13} \cdot a_{34}. 
\end{align*}
Since $\mh_3(a_1, a_2, a_3)$ is the cohomology class represented by $\Phi_3(a_1, a_2, a_3)$, 
this shows $(-1)^{1+|a_2|}\mh_3(a_1, a_2, a_3) \in \langle a_1,a_2,a_3 \rangle$. 
\end{proof}

\begin{corollary}
\label{cor:A3_formal_triple_Massey}
Let $\F$ be a field. 
Let $\Ah$ be a differential graded algebra over $\F$. 
If $\Ah$ is $A_3$-formal, then all non-empty triple Massey product sets contain zero.  
\end{corollary}
\begin{proof}
Let $a_1, a_2, a_3 \in \Hba$ be classes for which the Massey product is defined. 
By Proposition \ref{prop:Massey_product_and_minimal_model}, $(-1)^{|a_2|}\mh_3(a_1, a_2, a_3)$ is an element in $\langle a_1,a_2,a_3 \rangle$.  
If $\Ah$ is $A_3$-formal, we have $\mh_3=0$ which proves the assertion. 
\end{proof}

We note that by \cite[Theorem 3.4]{Isaksen} it is not true that for the fourfold Massey product $\langle a_1,a_2,a_3,a_4 \rangle$ to be defined 
it is not sufficient that the triple Massey products $\langle a_1,a_2,a_3\rangle$ and $\langle a_2,a_3,a_4 \rangle$ vanish. 
There is no obstruction, however, for an $A_3$-formal differential graded algebra as the following result shows. 


\begin{proposition}\label{prop:4tuple_products_are_defined}
Let $\F$ be a field. 
Let $\Ah$ be a differential graded algebra over $\F$. 
Let $a_1, a_2, a_3, a_4 \in \Hba$ be classes such that the neighbouring cup-products vanish, i.e., $a_1 \cdot a_2 = 0$, $a_2 \cdot a_3 = 0$, and $a_3 \cdot a_4 = 0$.
If $\Ah$ is $A_3$-formal, then the fourfold Massey product $\langle a_1,a_2,a_3,a_4 \rangle$ is defined. 
\end{proposition}
\begin{proof}
As in the proof of Proposition \ref{prop:Massey_product_and_minimal_model}, 
the existence of the map $f_2$ allows us to choose elements 
$a_{13}:=(-1)^{|a_1|}f_2(a_1,a_2)$, $a_{24}:=(-1)^{|a_2|}f_2(a_2,a_3)$ and $a_{35}:=(-1)^{|a_3|}f_2(a_3,a_4)$ in $\Ah$ such that 
$(-1)^{|a_2|}\mh_3(a_1, a_2, a_3) = \ba_{12} \cdot a_{24} + \ba_{13} \cdot a_{34}$ 
and $(-1)^{|a_3|}\mh_3(a_2, a_3, a_4) = \ba_{23} \cdot a_{35} + \ba_{24} \cdot a_{45}$. 
Since $\Ah$ is $A_3$-formal, we have $\mh_3=0$. 
This implies that we can find cochains $a_{14}$ and $a_{25}$ such that 
$\delta(a_{14}) = \ba_{12} \cdot a_{24} + \ba_{13} \cdot a_{34}$ 
and 
$\delta(a_{25}) = \ba_{23} \cdot a_{35} + \ba_{24} \cdot a_{45}$. 
In fact, we can choose $a_{14} = \pm f_3(a_1,a_2,a_3)$ and $a_{25} = \pm f_3(a_2,,a_3,a_4)$. 
Hence set of elements above the diagonal in the matrix 
\begin{align*}
\begin{pmatrix}
 1 & a_{12} & a_{13} & a_{14} &  \\
  & 1& a_{23} & a_{24} & a_{25} \\
  & & 1 & a_{34} & a_{35} \\
  & & & 1 & a_{45}\\
  & & & & 1
\end{pmatrix}
\end{align*} 
is a defining system for the fourfold Massey product $\langle a_1,a_2,a_3,a_4 \rangle$. 
\end{proof}


It follows from Proposition \ref{prop:4tuple_products_are_defined} that the question whether fourfold Massey products are defined 
provides an obstruction for $A_3$-formality. 
In Example \ref{example:A3_formal_>_3Massey_vanishing} we use this obstruction to construct a simple differential graded algebra 
for which all triple Massey products of degree one classes vanish, but which is not $A_3$-formal. 
See Theorem \ref{thm:Demushkin_p_equal_3}, Remark \ref{rem:realizable_Demushkin_group} and the introduction 
for more interesting examples.

\begin{example}\label{example:A3_formal_>_3Massey_vanishing}
Similar to \cite[Section 4]{Isaksen}, we can construct a differential graded algebra for which all triple Massey products vanish, 
but which is not $A_3$-formal, as follows. 
Let $S$ be the set $\{a_{12}, a_{23}, a_{34}, a_{45}, a_{13}, a_{24}, a'_{24}, a_{35}, a_{14}, a'_{25}\}$, 
and let $V$ be the free $\F$-vector space on the set $S$. 
Let $\Ah$ be the differential graded algebra over $\F$ whose underlying algebra 
is the graded tensor algebra $T(V)$ with $V$ in degree one and differential the unique derivation defined on generators as listed in the following table: 
\begin{center}
\adjustbox{scale=0.8}{
\begin{tabular}{c|c|c|c|c|c|c|c|c|c|c} 
$x$ & $a_{12}$ & $a_{23}$ & $a_{34}$ & $a_{45}$ & $a_{13}$ & $a_{24}$ & $a'_{24}$ & $a_{35}$ & $a_{14}$ & $a'_{25}$  \\ 
\hline
$\delta(x)$ & $0$ & $0$ & $0$ & $0$ & $\ba_{12}a_{23}$ & $\ba_{23}a_{34}$ & $\ba_{23}a_{34}$ & $\ba_{34}a_{45}$ & $\ba_{12} a_{24} + \ba_{13} a_{34}$ & $\ba_{23} a_{35} + \ba'_{24} a_{45}$ 
\end{tabular}}
\end{center}
where we omit the tensor sign from the notation. 
For $1 \le i \le 4$, we let $a_i \in H^1(\Ah)$ denote the cohomology class of $a_{i,i+1}$. 
%
There are exactly two non-empty triple Massey product sets $\langle a_1,a_2,a_3 \rangle$ and $\langle a_2,a_3,a_4 \rangle$ 
of elements in degree one. 
Moreover, both Massey product sets contain zero since $\ba_{12} a_{24} + \ba_{13} a_{34}$, which represents an element in $\langle a_1,a_2,a_3 \rangle$, 
and $\ba_{23} a_{35} + \ba'_{24} a_{45}$, which represents an element in $\langle a_2,a_3,a_4 \rangle$, are coboundaries. 
%
However, the fourfold Massey product $\langle a_1,a_2,a_3,a_4 \rangle$ is not defined, 
since we cannot choose a defining system such that both triple Massey products vanish simultaneously. 
By Proposition \ref{prop:4tuple_products_are_defined}, this implies that $\Ah$ is not $A_3$-formal. 
\end{example}


\section{$A_3$-formality and Hochschild cohomology}\label{sec:A3_and_HH}

In this section, we recall graded Hochschild cohomology groups of graded algebras 
and will then introduce the canonical class of Benson--Krause--Schwede in Hochschild cohomology. 
We then show that a differential graded algebra $\Ah$ over a field is $A_3$-formal if and only if 
the canonical class of $\Ah$ is trivial. 

\subsection{Graded Hochschild cohomology}

Let $A = \oplus_{i \ge 0} A^i$ be a non-negatively graded $\F$-algebra with $A^0=\F$. 
Let $M$ be a graded $A$-bimodule. 
For $s\in \Z$, we write $M[s]$ for the graded $A$-bimodule given in degree $n$ by $M[s]_n = M_{n+s}$. 
We recall from \cite[Section 4]{BKS} that 
the graded Hochschild cochain complex $C^{\bbb,*}(A,M)$ of $A$ with coeﬃcients in $M$ 
is defined by 
\begin{align}\label{eq:def_graded_complex_HH}
C^{n,s}(A,M) = \uHom_{\F}(A^{\ot n}, M[s]). 
\end{align}
A Hochschild cochain in degree $(n,s)$ with coefficients in $M$ is thus given by a multilinear function 
from $n$-tuples of elements of $A$ to $M$ which raise the degree by $s$. 
The differential $\dee^n \colon \uHom_{\F}(A^{\otimes n}, M[s]) \to \uHom_{\F}(A^{\otimes (n+1)}, M[s])$, is defined by 
\begin{align*}
\dee^n(f)(a_1, \ldots, a_{n+1}) & = (-1)^{s|a_1|}a_1f(a_2, \ldots, a_{n+1}) \\
& + \sum_{i=1}^n (-1)^i f(a_1, \ldots, a_ia_{i+1}, \ldots, a_{n+1}) \\
& + (-1)^{n+1} f(a_1, \ldots, a_n)a_{n+1}.
\end{align*}
We note that the differential only changes the first grading denoted by $\bbb$. 
\begin{defn}
The Hochschild cohomology group $\HH^{n,s}(A,M)$ is the $n$th cohomology group of $C^{\bbb,s}(A,M)$, i.e., 
\begin{align*}
\HH^{n,s}(A,M) = H^n(C^{\bbb,s}(A,M)). 
\end{align*}
We write $\HH^{n,s}(A)$ for $\HH^{n,s}(A,A)$. 
\end{defn}


Now let $\Ah$ be a differential graded algebra over $\F$ with cohomology $\Hba$. 
We choose an $\F$-linear graded map $f_1 \colon \Hba \to \ker \ma_1$ which induces the identity on $\Hba$.   
We then choose an $\F$-linear graded map 
$f_2 \colon \Hba \otimes \Hba \to A$ of degree $-1$ satisfying  
\begin{align*}
\ma_1f_2 = f_1\mh_2  - \ma_2(f_1 \otimes f_1). 
\end{align*}
We define the graded $\F$-linear map $\Phi_3 \colon \Hba^{\otimes 3} \to A[-1]$ by 
\begin{align}\label{eq:def_of_Phi3_construction}
\Phi_3 = \ma_2(f_1 \otimes f_2 - f_2 \otimes f_1) - f_2(\mh_2 \otimes \One - \One \otimes \mh_2). 
\end{align}
For all $x, y, z \in \Hba$, $\Phi_3(x, y, z)$ is a cocycle in $\Ah$. 
We let $m_3$ denote the  graded $\F$-linear map $\Hba^{\otimes 3} \to \Hba[-1]$ induced by $\Phi_3$. 
We consider $\mmm_3$ as a cochain in the Hochschild complex $C^3(\Hba, \Hba[-1])$. 

\begin{lemma}\label{lemma:g3_is_a_cocycle}
The cochain $m_3$ in $C^3(\Hba, \Hba[-1])$ is a cocycle with respect to the Hochschild differential $\dee^3 \colon C^3(\Hba, \Hba[-1]) \to C^4(\Hba, \Hba[-1])$. 
In particular, $m_3$ defines a class in $\HH^{3,-1}(\Hba,\Hba)$. 
\end{lemma}
\begin{proof}
We compute $\dee^3\Phi_3 = \ma_1(\ma_2(f_2 \otimes f_2))$. 
Thus, the cohomology class $\dee^3\Phi_3$ is zero, and $\mmm_3$  is a cocycle in the Hochschild complex $C^3(\Hba, \Hba[-1])$. 
\end{proof}



The next result may be found in \cite[Proposition 5.4]{BKS} (see also Proposition \ref{prop:kappa3_functorial_Koszul} below):

\begin{proposition}\label{prop:m3_functorial}
Let $\varphi \colon \Ah \to \Bh$ be a morphism of connected differential graded $\F$-algebras. 
Consider $\Hbb$ as an $\Hba$-bimodule via the induced morphism $\varphi_{\bbb} \colon \Hba \to \Hbb$. 
Let $(f^{\Ah}_1, f^{\Ah}_2)$ and $(f^{\Bh}_1, f^{\Bh}_2)$ be choices of graded $\F$-linear maps as above for $A$ and $B$, respectively. 
Let $\ma_3$ and $\mb_3$ denote the induced maps for $\Hba$ and $\Hbb$, respectively, defined using formula \eqref{eq:def_of_Phi3_construction}. 
Then the Hochschild cocycles $\varphi_{\bbb} \circ \ma_3$ and $\mb_3 \circ (\varphi_{\bbb})^{\ot 3}$ are cohomologous 
in the complex $C^{\bbb}(\Hba, \Hbb[-1])$. \qed
\end{proposition}

Applying Proposition \ref{prop:m3_functorial} with $A=B$ and $\varphi$ being the identity yields the following result (see also \cite[Corollary 5.7]{BKS}):

\begin{corollary}\label{cor:independence_identity}
The class $[m_3] \in \HH^{3,-1}(\Hba, \Hba)$ only depends on the differential graded algebra $\Ah$ and not the choice of the pair $(f_1,f_2)$. \qed 
\end{corollary}

\begin{remark}
We could also deduce the independence of the class $[m_3]$ from Theorem \ref{thm:A_3_Kadeishvili} as follows. 
Let $\tm_3 \colon \Hba^{\otimes 3} \to \Hba[-1]$ 
be another map which defines an $A_3$-algebra structure which turns $\Hba$ into a minimal model of $A$. 
By Theorem \ref{thm:A_3_Kadeishvili}, there is an isomorphism of $A_3$-algebras $g_\bbb \colon \Hba \to \Hba$ such that $g_1=\id_{\Hba}$ and \eqref{eq:relation_f3} holds, 
i.e., 
\begin{equation*}
\begin{aligned}
\tm_3 (g_1 \otimes g_1 \otimes g_1) - g_1  m_3 
=  g_2 (\mh_2 \otimes \One - \One \otimes \mh_2) - \mh_2(g_1 \otimes g_2 - g_2 \otimes g_1). 
\end{aligned} 
\end{equation*}
This shows that the difference of $m_3$ and $\tm_3$ in the group $C^3(\Hba,\Hba[-1])$ is 
the coboundary $\dee^2 (g_2)$.  
\end{remark}

Following Benson--Krause--Schwede in \cite[page 3623]{BKS} we use the following terminology: 

\begin{defn}\label{def:canonical_class}
We denote the class $[m_3] \in \HH^{3,-1}(\Hba)$ by $\ga$ and call it the \emph{canonical class} of $\Ah$. 
\end{defn}

\begin{remark}\label{rem:sign_difference_to BKS_canonical_class}
We note that our construction of the canonical class differs by a sign from \cite[Construction 5.1]{BKS}. 
\end{remark}

\begin{remark}\label{rem:canonical_class_functorial}
As pointed out in \cite[Corollary 5.7]{BKS}, Proposition \ref{prop:m3_functorial} implies that the canonical class satisfies the following functoriality: 
With the notation and assumption of Proposition \ref{prop:m3_functorial}, 
the images of canonical classes are related by the formula 
\begin{align*}
\varphi_\bbb(\gamma_A) = \varphi^{\bbb}(\gamma_B) ~ \text{in} ~ \HH^{3,-1}(\Hba,\Hbb)
\end{align*}
under the induced maps 
\begin{align*}
\HH^{3,-1}(\Hba,\Hba) \xto{\varphi_\bbb} \HH^{3,-1}(\Hba,\Hbb) \xleftarrow{\varphi^\bbb} \HH^{3,-1}(\Hbb,\Hbb).
\end{align*} 
\end{remark}


\subsection{Kadeishvili's criterion for $A_3$-formality}

The following result is a modified version of Kadeishvili's theorem \cite{Kadeishvili88} (see also \cite[Theorem 4.7]{ST}). 

\begin{theorem}[Kadeishvili]\label{thm:existence_of_obstruction_class}
Let $\Ah$ be a connected differential graded algebra over $\F$ with cohomology algebra $\Hba$. 
Then $\Ah$ is $A_3$-formal if and only if the canonical class $\ga \in \HH^{3,-1}(\Hba)$ vanishes.  
\end{theorem}
\begin{proof}
If $\Ah$ is $A_3$-formal, then $\mh_3$ is trivial and the class of $\mh_3$ vanishes in $\HH^{3,-1}(\Hba)$. 
%
Now we assume that $[\mmm_3]=0$  in $\HH^{3,-1}(\Hba)$. 
We may assume that $\Phi_3$ and hence $\mmm_3$ is constructed using maps $f_1,f_2$ as in \eqref{eq:def_of_Phi3_minimal}. 
Then there exists a map $\eta \colon \Hba^{\otimes 2} \to (\ker \ma_1)[-1]$ 
such that $\dee^2 [\eta] = \mmm_3$ as maps $\Hba^{\otimes 3} \to \Hba[-1]$. 
We will now show that we can use $\eta$ to modify our initial choice of $f_2$ and thereby $\Phi_3$ 
such that the new map $\widetilde{\Phi}_3$ has values in the image of $\ma_1$. 
%
We set $\tilde{f}_2 = f_2 - \eta$. 
We note that $\tilde{f}_2$ satisfies 
\begin{align*}
\ma_1\tilde{f}_2 = \ma_1(f_2 - \eta)  
= f_1\mh_2  - \ma_2(f_1 \otimes f_1) 
\end{align*}
since $\ma_1\eta=0$ by the assumption on the image of $\eta$. 
Thus, $\tilde{f}_2$ is also a cochain homotopy between $f_1\mh_2$ and $\ma_2(f_1 \otimes f_1)$.  
%
We then define the map $\widetilde{\Phi}_3$ by replacing $f_2$ with $\tilde{f}_2$, i.e., we define 
\begin{align*}
\widetilde{\Phi}_3 \coloneqq \ma_2(f_1 \otimes \tilde{f}_2 - \tilde{f}_2 \otimes f_1) - \tilde{f}_2(\mh_2 \otimes \One - \One \otimes \mh_2). 
\end{align*}
We then have 
\begin{align*}
\Phi_3 - \widetilde{\Phi}_3 =  \ma_2(f_1 \otimes \eta - \eta \otimes f_1) - \eta(\mh_2 \otimes \One - \One \otimes \mh_2). 
\end{align*}
By definition of $\dee^2$ and the assumption on $\eta$, 
this implies $\widetilde{\Phi}_3 = \Phi_3 - \dee^2\eta = 0$ as maps $\Hba^{\otimes 3} \to \Hba[-1]$. 
This implies that the image of $\widetilde{\Phi}_3$ is contained in the image of $\ma_1$. 
As explained in Remark \ref{rem:A3_formality_morphism}, this shows that there is a graded $\F$-linear map $f_3$ 
which extends $f_1, \tilde{f}_2$ to an $A_3$-algebra morphism 
which induces the identity on $\Hba$ after taking cohomology. 
\end{proof}

\begin{remark}\label{rem:can_choose_kappa3_trivial}
The proof of Theorem \ref{thm:existence_of_obstruction_class} 
shows that, if $\Ah$ is $A_3$-formal, 
we may choose $(f_1,f_2)$ so that 
 the map $\mmm_3 \colon \Hba^{\otimes 3} \to \Hba[-1]$ is trivial. 
\end{remark}

\begin{remark}
Theorem \ref{thm:existence_of_obstruction_class} shows that we may consider the class $\ga = [m_3]$ 
in $\HH^{3,-1}(\Hba)$ as the obstruction class for $A_3$-formality. 
\end{remark}

\begin{remark}\label{rem:intrinsically_A3_formal}
Let $\Hb$ be a graded algebra with $\HH^{3,-1}(\Hb)=0$. Then Theorem \ref{thm:existence_of_obstruction_class} implies that $\Hb$ is \emph{intrinsically} $A_3$-formal, 
i.e., every connected differential graded algebra $\Ah$ with cohomology isomorphic to $\Hb$ is $A_3$-formal. 
See also \cite[Lemma 4.6]{PQ}. 
We refer to \cite{DQ} for a non-trivial example of an intrinsically $A_3$-formal $\F_2$-algebra. 
\end{remark}



\section{$A_3$-formality and Koszul algebras}\label{sec:A3_and_Koszul}

We now recall Koszul algebras and then show that they allow for a simplified construction of the canonical class. 

\subsection{Koszul algebras} 

As a general reference for quadratic and Koszul algebras we refer to \cite{ppqa}. 

\begin{defn}
A \emph{quadratic algebra} is a non-negatively graded $\F$-algebra $A=\bigoplus_{i\ge 0}A_i$ such that $A_0=\F$ and $A$ is generated over $\F$ by $A_1$ with relations of degree two. 
Explicitly, let 
\[
T(A_1) = \F \oplus A_1 \oplus (A_1\otimes A_1) \oplus \cdots = \bigoplus_{i\ge 0} A_1^{\otimes i}
\]
be the free tensor algebra of the $\F$-vector space $A_1$. 
Then $A$ is quadratic if the canonical map $\tau \colon T(A_1) \to A$ is surjective and $\ker(\tau)$ is generated by its component $R = \ker(\tau)\cap (A_1\otimes A_1)$ of degree two, so that $A=T(A_1)/(R)$, where $(R)$ denotes the ideal generated by $R$.  
A quadratic algebra $A$ is called \emph{locally finite} if each $A_i$ is a finite-dimensional vector space over $\F$. 
\end{defn} 


To every quadratic algebra one can associate the following chain complex: 

\begin{defn}\label{def:Koszulcomplex}
Let $A=T(V)/(R)$ be a quadratic algebra over $\F$. 
We denote by $K_i^i(A)$ the $\F$-linear subspace defined by $K_0^0(A)=\F,$  $K_1^1(A)=V$, $K_2^2=R$ and 
\begin{align*}
K_i^i(A) = \bigcap_{0\le j \le i-2} V^{\otimes j} \otimes R \otimes V^{\otimes i-j-2} \subset V^{\otimes i} ~ \text{for}~i\ge 3.
\end{align*}
We set $K_i(A)=A\otimes K_i^i(A)\otimes A$. 
For each $i\ge 0$, we define a homomorphism $d=d_i \colon K_i(A) \to K_{i-1}(A)$ by
\[
d_i (a\otimes x_1 \otimes \cdots \otimes x_i \otimes a') = (ax_1) \otimes x_2 \otimes \cdots \otimes x_i \otimes a' + (-1)^i a \otimes x_1 \otimes x_2 \otimes \cdots \otimes (x_ia'). 
\]
Since $R\subset V^{\otimes 2}$ generates the relations in $A$, it is clear that $d^2=0$. 
We refer to the chain complex $(K(A),d)$ as the \emph{Koszul complex of $A$}.  
A morphism $\varphi \colon A \to B$ of quadratic algebras induces a morphism of cochain complexes $\varphi_K \colon K(A) \to K(B)$. 
\end{defn}

Let $(B(A),d_B)$ denote the bar complex of $A$ (see e.g.\,\cite[(1.1.4)]{Witherspoon}). 
Since $R$ describes the relations in $A$, it follows that the natural inclusion $K(A) \into B(A)$ defines a morphism of complexes.  

\begin{defn}\label{def:Koszulity}
A quadratic algebra $A$ is called a \emph{Koszul algebra} if the inclusion $(K(A),d) \into (B(A),d_B)$ is a quasi-isomorphism. 
\end{defn}

\begin{rem}\label{rem:Koszul_equivalent_defs}
We note that there are many different ways to characterise the notion of a Koszul algebra. 
We refer to \cite[Chapter 2, Sections 1 and 3]{ppqa} and 
for example \cite[Section 2]{BGS} and \cite[Section 3]{vandenbergh} 
for proofs that the alternative definitions are equivalent to Definition \ref{def:Koszulity}.  
\end{rem}

\begin{example}\label{example:exterior}
Important examples of Koszul algebras include symmetric and exterior algebras over $\F$ (see \cite[Example on page 20]{ppqa}). 
\end{example}


\subsection{Hochschild cohomology of Koszul algebras}\label{sec:HH_and_Koszul_complex}

We now show that for a Koszul algebra, we can use the following simple complex to compute its Hochschild cohomology. 
Let $A$ be a quadratic algebra and $M$ be a graded $A$-bimodule. 
For every $n \ge 0$ and $s \in \Z$, 
let $\partial_n$ denote the $\F$-linear map 
\begin{align*}
\partial_n \colon \Hom_{\F}(K_n^n(A), M_{n+s}) \to \Hom_{\F}(K_{n+1}^{n+1}(A), M_{n+1+s})
\end{align*}
defined by setting 
\begin{align*}
\partial_n(f)(x_1, \ldots, x_{n+1}) =  (-1)^{|x_1|s} x_1f(x_2, \ldots, x_{n+1}) 
+  (-1)^{n+1}f(x_1, \ldots, x_{n})x_{n+1}. 
\end{align*}
It is easily verified that $\partial \circ \partial =0$. 
It follows immediately from the definitions that the inclusion $\iota \colon K^i_i(A) \into A^{\ot i}$ induces a morphism of complexes via restriction along $\iota$:  
\begin{align}\label{eq:contractingiso_complex_C_to_K} 
(C^{\bbb,s}(A,M), \dee) \to (\uHom_{\F}(K_{\bbb}^{\bbb}(A), M[s]), \dee). 
\end{align}
We then get the following result (see also \cite[Proposition 3.3]{vandenbergh}): 

\begin{prop}\label{prop:HHviaKoszulcomplex_simple}
Let $A$ be a Koszul algebra and $M$ be a graded $A$-bimodule. 
For every pair $(n, s)$, the map of complexes  \eqref{eq:contractingiso_complex_C_to_K} 
induces an isomorphism 
\begin{align}\label{eq:diag_iso_of_HH}
\iota^* \colon \HH^{n,s}(A,M) \xto{\cong} H^n(\uHom_{\F}(K_{\bbb}^{\bbb}(A), M[s]),\partial). 
\end{align} 
\end{prop}
\begin{proof}
Let $A^e = A\otimes A^{\op}$ denote the enveloping algebra of $A$ where $A^{\op}$ denotes the opposite algebra of $A$ (see e.g., \cite[Section 1.1]{Witherspoon}). 
The ring $A^e$ inherits a structure of a graded $\F$-algebra from $A$. 
We can consider the graded $A$-bimodule $M$ as a graded $A^e$-module. 
%
For every $n$, there is a natural isomorphism 
\begin{align}\label{eq:contractingiso_Hom_BtoT}
\Hom_{A^e}(B^n(A), M) = \Hom_{A^e}(A \otimes A^{\ot n} \otimes A, M) \xto{\cong} \Hom_{\F}(A^{\ot n}, M) 
\end{align}   
defined by sending $f$ to the $\F$-linear map which sends $a_1 \otimes \cdots \otimes a_n$ to $f(1\otimes a_1 \otimes \cdots \otimes a_n \otimes 1)$. 
%
This isomorphism descends to the category of graded $A^e$-modules and $\F$-vector spaces respectively, 
i.e., for every integer $s$, we have an isomorphism 
\begin{align}\label{eq:contractingiso_BtoT}
\uHom_{A^e}(B^n(A), M[s])  \xto{\cong} \uHom_{\F}(A^{\ot n}, M[s]).  
\end{align}  
It is straight-forward to check that the isomorphism \eqref{eq:contractingiso_BtoT} 
yields an isomorphism of cochain complexes 
\begin{align*}
(\uHom_{A^e}(B(A), M[s]),d_B^*)  \xto{\cong}  (C^{\bbb,s}(A,M), \dee).  
\end{align*} 
The restriction of \eqref{eq:contractingiso_Hom_BtoT} to Koszul complexes 
induces an isomorphism of cochain complexes 
\begin{align*}
(\uHom_{A^e}(K_{\bbb}(A), M[s]),d^*)  \xto{\cong} (\uHom_{\F}(K_{\bbb}^{\bbb}(A), M[s]),\partial). 
\end{align*}
%
We then get the following commutative diagram of cochain complexes 
\begin{align*}
\xymatrix{
(\uHom_{A^e}(B(A), M[s]),d_B^*) \ar[d]_-{\simeq} \ar[r]^-{\cong} & (C^{\bbb,s}(A,M), \dee) \ar[d] \\
(\uHom_{A^e}(K_{\bbb}(A), M[s]),d^*) \ar[r]^-{\cong} & (\uHom_{\F}(K_{\bbb}^{\bbb}(A), M[s]),\partial)
}
\end{align*} 
in which the horizontal maps are isomorphisms and the left-hand morphism is a quasi-isomorphism since $A$ is a Koszul algebra. 
This implies that the right-hand morphism is a quasi-isomorphism as well. 
\end{proof}


The Koszul complex simplifies the  computation of Hochschild cohomology. 
Here is a first example which will also play a role later. 

\begin{example}\label{example:exterior_algebra_on_2_gen}
Let $A$ be the exterior algebra on two generators $x$ and $y$ in degree one. 
Note that $A=T(V)/(R)$ with $V = A^1 = \F \langle x,y \rangle$ and 
\[
R = \F \langle x\ot x, y \ot y, x \ot y + y \ot x \rangle.
\]
It is well-known that $A$ is a Koszul algebra (see e.g.~\cite[Example on page 20]{ppqa}). 
%
We have $K_2^2(A) = R$, and 
\begin{align*}
K_3^3(A) & = (V\ot R) \cap (R\ot V) \\
& = \F \langle x\ot x \ot x, y \ot y \ot y, ~ y \ot x\ot x + x \ot y \ot x + x \ot x \ot y, \\
& ~ x \ot y \ot y + y \ot x \ot y + y \ot y \ot x \rangle.
\end{align*}
We note that every cochain $\psi \colon K^3_3(A) \to A^2$ is a cocycle since $A^3=0$, i.e., there is only the trivial map $K^4_4(A) \to A^3$. 
Hence the vector space of cocycles for $\partial_3$ equals the space $\Hom_{\F}(K^3_3,H^2)$ which has dimension four over $\F$. 
Now we can check that the differential 
$\partial_2 \colon \Hom_{\F}(K_2^2(A),A^1)  \to \Hom_{\F}(K^3_3(A),A^2)$ 
is trivial. 
Hence we get that $\HH^{3,-1}(A)$ is a four-dimensional vector space. 
In particular, it is non-trivial. 
\end{example} 
%

\begin{remark}\label{rem:dgas_with_exterior_algebra_as_cohomology}
Let $A$ again be the exterior algebra on two generators $x$ and $y$ in degree one. 
By \cite[Proposition 3.2]{Luetal0}, in order to define an $A_{\infty}$-algebra structure on the graded algebra $A$, 
it suffices to specify $\F$-linear maps 
$m_n \colon (A^1)^{\ot n} \to A^2$ for all $n \ne 2$. 
We set $m_1=0$, $m_n=0$ for $n\ge 4$ and can choose $m_3$ to be any $\F$-linear map $(A^1)^{\ot 3} \to A^2$.  
This defines a minimal $A_{\infty}$-algebra structure on $A$. 
Now we recall that, for every $A_{\infty}$-algebra $\Ah$, the canonical morphism of $A_{\infty}$-algebras 
$\Ah \to \Omega B\Ah$, 
where $\Omega$ denotes the cobar and $B$ the bar construction, is a quasi-isomorphism (see for example \cite[Proposition 4.5]{Luetal0}). 
Since $\Omega B\Ah$ is a differential graded algebra, every $A_{\infty}$-algebra is quasi-isomorphic as an $A_{\infty}$-algebra to a differential graded algebra. 
Since quasi-isomorphisms between $A_{\infty}$-algebras are homotopy equivalences by \cite[Theorem 3.7 on page 13]{Keller}, 
we conclude that $A$ is the minimal model of some differential graded algebra $\Ah$ with canonical class  $[m_3] \in \HH^{3,-1}(\Ah)$.  
By Theorem \ref{thm:A_3_Kadeishvili} and Example \ref{example:exterior_algebra_on_2_gen}, 
this shows that there are many differential graded algebras whose cohomology algebra is 
isomorphic to $A$ 
but which are not quasi-isomorphic as $A_3$-algebras. 
In particular, $A$ is not intrinsically $A_3$-formal (see Remark \ref{rem:intrinsically_A3_formal}). 
\end{remark}


\subsection{Canonical class for Koszul cohomology algebras}\label{sec:canonical_class_Koszul}

Let $(\Ah, \ddA)$ be a connected differential graded $\F$-algebra. 
We assume that its cohomology algebra $\Hb$ is a Koszul algebra. 
Let $R \subset H^1 \ot H^1$ denote the relations such that $\Hb = T(H^1)/(R)$.  
Let $\Zh^1 = \ker \ddA \subset \Ah^1$ denote the cocycles in degree one. 
Let $f_1 \colon K_1^1(\Hb) = H^1 \to  \Zh^1$ be an $\F$-linear map which induces the identity on cohomology. 
Let $f_2 \colon K_2^2(\Hb) = R \to \Ah^1$ be an $\F$-linear map such that $-\ma_2(f_1 \ot f_1) = \ddA f_2$. 
We define the $\F$-linear map $\Psi_3$ by 
\begin{align}\label{eq:def_of_Psi3_construction}
\Psi_3 = \ma_2(f_1 \otimes f_2 - f_2 \otimes f_1) \colon (H^1 \ot R) \cap (R \ot H^1) \to \Ah^2.
\end{align}
We check that $\Psi_3$ has image in the cocycles of $\Ah^2$: 
\begin{align*}
\ddA \Psi_3 & = \ddA \ma_2(f_1 \otimes f_2 - f_2 \otimes f_1) \\
 & = \ma_2(f_1 \ot \ddA f_2 - \ma_1f_2 \ot f_1) ~ \text{by \eqref{eq:m1m2} and} ~ \ddA f_1=0 \\
 & = \ma_2(-f_1 \ot \ma_2(f_1\ot f_1) + \ma_2(f_1 \ot f_1) \ot f_1) ~ \text{using \eqref{eq:relation_f1m2},} ~ \mh_1=0, \mh_2=0 \\
 & = 0 ~ \text{by associativity of} ~ \ma_2. 
\end{align*}
Hence $\Psi_3$ induces an $\F$-linear map 
\begin{align}\label{eq:def_of_kappa3_construction}
\kappa_3 = \ma_2(f_1 \otimes f_2 - f_2 \otimes f_1) \colon K_3^3(\Hb) = (H^1 \ot R) \cap (R \ot H^1) \to H^2.
\end{align}
The map $\kappa_3$ is compatible with maps of differential graded algebras in the following way: 

\begin{proposition}\label{prop:kappa3_functorial_Koszul}
Let $\varphi \colon \Ah \to \Bh$ be a morphism of connected differential graded $\F$-algebras. 
Assume that both $\Hba$ and $\Hbb$ are Kozul algebras. 
Consider $\Hbb$ as an $\Hba$-bimodule via the induced morphism $\varphi_{\bbb} \colon \Hba \to \Hbb$. 
Let $(\fA_1, \fA_2)$ and $(\fB_1, \fB_2)$ be choices of graded $\F$-linear maps as above for $\Ah$ and $\Bh$, respectively. 
Let $\kappa^{\Ah}_3$ and $\kappa^{\Bh}_3$ denote the induced maps for $\Hba$ and $\Hbb$, respectively, defined using formula \eqref{eq:def_of_kappa3_construction}. 
Then the cocycles $\varphi_{\bbb} \circ \kappa^{\Ah}_3$ and $\kappa^{\Bh}_3 \circ (\varphi_{\bbb})^{\ot 3}$ represent the same class in $\HH^{3,-1}(\Hba,\Hbb)$. 
\end{proposition}
\begin{proof}
The graded $\F$-linear maps $\fB_1 \circ \varphi_\bbb$ and $\varphi|_{\ker \ddA} \circ \fA_1$ are maps 
$H^1(\Ah) \to \ker \ddB \subset \Bh^1$ of
cochain complexes which induce the same map in cohomology. 
This implies that there is a graded $\F$-linear map $g \colon H^1(\Ah) \to \Bh^0 = \F$ such that 
\begin{align}\label{eq:ddB_fB_varphi_varphi_fA}
\ddB g = \fB_1 \circ \varphi_\bbb - \varphi|_{\ker \ddA} \circ \fA_1.
\end{align}
%
We define the $\F$-linear map $t \colon K^2_2(\Hba) \to \Bh^1$ by 
\begin{align*}
t(x,y) := \varphi(\fA_2(x,y)) - \fB_2(\varphi_\bbb(x), \varphi_\bbb(y)) 
- \varphi(\fA_1(x))\cdot g(y) + g(x) \cdot \fB_1(\varphi_\bbb(y))
\end{align*}
for $(x,y) \in  K^2_2(\Hba)$. 
We can then compute that $t(x,y)$ lies in $\ker \ddB$ for all pairs $(x,y)$. 
%
Hence we may define the $\F$-linear map $\tau \colon  K^2_2(\Hba) \to H^1(\Bh)$ by sending $(x,y)$ to the cohomology class of $t(x,y)$. 
Using \eqref{eq:ddB_fB_varphi_varphi_fA}, we can then compute 
\begin{align*}
(\varphi \circ \kappa^{\Ah}_3 - \partial^2 \tau)(x,y,z) = \kappa^{\Bh}_3 \circ (\varphi_{\bbb})^{\ot 3}(x,y,z)
\end{align*}
in $H^2(\Bh)$. 
This implies the assertion. 
\end{proof}

\begin{remark}\label{rem:cochain_on_Koszul_is_the_restriction}
We note that the maps $t$ and $\tau$ in the proof of Proposition \ref{prop:kappa3_functorial_Koszul} 
are the restriction of the corresponding cochains 
which are used  in \cite[Proof of Proposition 5.4]{BKS} 
for the proof of Proposition \ref{prop:m3_functorial}. 
\end{remark}

\begin{corollary}\label{cor:independence_identity_Koszul} 
Let $\Ah$ be a connected differential graded $\F$-algebra such that its cohomology algebra $\Hba$ is a Koszul algebra. 
Let $\kappa_3$ be as defined in \eqref{eq:def_of_kappa3_construction}. 
Then the class $[\kappa_3] \in H^3(\uHom_{\F}(\Kbb(\Hba), \Hba[-1]),\partial)$ 
only depends on the differential graded algebra $\Ah$ and not the choice of the pair $(\fA_1,\fA_2)$. 
\end{corollary}
\begin{proof}
The assertion follows from Proposition \ref{prop:kappa3_functorial_Koszul} with $\Ah = \Bh$ and $\varphi$ being the identity. 
\end{proof}

We are now ready to show that we can compute the canonical class of a differential graded algebra 
whose cohomology is Koszul as follows: 

\begin{proposition}\label{prop:canonical_class_via_Koszul_complex} 
Let $(\Ah, \ddA)$ be a connected differential graded $\F$-algebra such that its cohomology algebra $\Hba$ is a Koszul algebra. 
Let $(\fA_1,\fA_2)$ be a choice of maps $\fA_1 \colon H^1(\Ah) \to \ker \ddA \subset \Ah^1$ 
and $\fA_2 \colon K_2^2(\Hba) \to \Ah^1$ such that 
$\ddA \fA_2 = -\ma_2(\fA_1 \ot \fA_1)$.
Let $\kappa_3$ be defined as in \eqref{eq:def_of_kappa3_construction}. 
Then $[\kappa_3]$ 
is the image of the canonical class $\ga \in \HH^{3,-1}(\Hba)$ 
under the isomorphism $\iota^*$ in \eqref{eq:diag_iso_of_HH}. 
In particular, the canonical class $\ga$ is determined by any pair $(\fA_1,\fA_2)$ satisfying the above assumptions.  
Moreover, $\ga = 0$ if and only if $[\kappa_3] = 0$. 
\end{proposition}
\begin{proof}
To simplify the notation we write $\Hb$ for $\Hba$. 
By definition, $\ga$ is determined by the following data. 
We choose an $\F$-linear graded map $f_1 \colon \Hb \to \ker \ddA$ which induces the identity on $\Hb$.   
We then choose an $\F$-linear graded map 
$f_2 \colon \Hb \otimes \Hb \to \Ah[-1]$ satisfying  
$\ddA f_2 = f_1\mh_2  - \ma_2(f_1 \otimes f_1)$. 
We define the graded $\F$-linear map $\Phi_3 \colon (\Hb)^{\otimes 3} \to \Ah[-1]$ by 
\begin{align*}
\Phi_3 = \ma_2(f_1 \otimes f_2 - f_2 \otimes f_1) - f_2(\mh_2 \otimes \One - \One \otimes \mh_2),
\end{align*}
and let $m_3 \colon (\Hb)^{\otimes 3} \to \Hb[-1]$ be the induced map to cohomology. 
Then $\ga = [m_3] \in \HH^{3,-1}(\Hb)$. 
The isomorphism $\iota^*$ of \eqref{eq:diag_iso_of_HH} sends $[m_3]$ to the class represented 
by the restriction ${m_3}|_{K^3_3(\Hb)} \colon K_3^3(\Hb) \to H^2$ 
of $m_3$ to the subspace $K^3_3(\Hb) = (H^1\otimes R) \cap (R \otimes H^1) \subset (\Hb)^{\ot 3}$. 
Since the restrictions of $\mh_2 \ot \One$ and $\One \ot \mh_2$ to $(H^1\otimes R) \cap (R \otimes H^1)$ vanish in $H^2 = (H^1 \ot H^1)/R$, 
the map ${m_3}|_{K^3_3(\Hb)} \colon K_3^3(\Hb) \to H^2$ is given by class of the $\F$-linear map 
\begin{align*}
\ma_2((f_1)|_{H^1} \otimes  (f_2)|_{K_2^2(\Hb)} - (f_2)|_{K_2^2(\Hb)} \otimes (f_1)|_{H^1})  \colon (H^1\otimes R) \cap (R \otimes H^1) \to \Ah^2. 
\end{align*}
Hence, ${m_3}|_{K^3_3(\Hb)}$ is the map $\kappa_3$ defined as in \eqref{eq:def_of_kappa3_construction} using the pair $((f_1)|_{H^1}, (f_2)|_{K_2^2(\Hb)})$. 
We thus have 
\begin{align*}
\iota^*(\ga) = [{m_3}|_{K^3_3(\Hb)}] = [\kappa_3] ~ \text{in} ~ H^3(\uHom_{\F}(\Kbb(\Hb), \Hb[-1]),\partial). 
\end{align*}
Moreover, by Corollary \ref{cor:independence_identity_Koszul}, any other choice of pair $(\fA_1, \fA_2)$ yields the same class $[\kappa_3]$.  
This finishes the proof. 
\end{proof}

\begin{remark}\label{rem:Koszul_restricted_f3_suffices}
Let $(\Ah, \ddA)$ be a connected differential graded $\F$-algebra such that its cohomology algebra $\Hba$ is a Koszul algebra. 
Let $(\fA_1,\fA_2)$ be a choice of maps $\fA_1 \colon H^1(\Ah) \to \ker \ddA$ and $\fA_2 \colon K_2^2(\Hba) \to \Ah^1$ such that 
$\ddA \fA_2 = -\ma_2(\fA_1 \ot \fA_1)$. 
It is a consequence of Proposition \ref{prop:canonical_class_via_Koszul_complex} that, in order to prove that $\ga$ vanishes, 
it suffices to show that the image of the map 
$\ma_2(f_1 \otimes f_2 - f_2 \otimes f_1)$ 
is contained in the coboundaries in  $\Ah^2$.  
\end{remark}


As a direct application of Proposition \ref{prop:canonical_class_via_Koszul_complex} we observe the following: 

\begin{proposition}\label{prop:Koszul_with_H2_equal_zero_is_A3formal}
Let $\Ah$ be a connected differential graded algebra such that its cohomology algebra $\Hba$ is Koszul. 
Assume that $H^2(\Ah) = 0$. 
Then $\Ah$ is $A_3$-formal. 
\end{proposition}
\begin{proof}
By Proposition \ref{prop:canonical_class_via_Koszul_complex}, 
we can compute the canonical class of $\gamma_{\Ah}$ 
via the map $\kappa_3 \colon K^3_3(\Hba) \to H^2(\Ah)$ of \eqref{eq:def_of_kappa3_construction} since $\Hba$ is Koszul. 
Since $H^2(\Ah) = 0$ by assumption, $\kappa_3$ is trivial and $\gamma_{\Ah}=0$. 
Hence $\Ah$ is $A_3$-formal by Theorem \ref{thm:existence_of_obstruction_class}. 
\end{proof}



\section{$A_3$-formality and group cohomology}\label{sec:group_coh}

We now specialise to differential graded algebras which arise from continuous group cohomology of profinite groups. 
In Section \ref{sec:Demushkin_groups_are_Koszul}, we show that the $\F_p$-cohomology algebra of a pro-$p$ Demushkin group is Koszul. 

\subsection{Continuous group cohomology and $A_3$-formality}

Let $G$ be a profinite group, and let $G^n$ denote the $n$-fold direct product of $G$ with itself. 
Let $p$ be a prime number. 
Let $\Ch^n(G,\F_p)$ denote the $\F_p$-vector space of continuous functions $G^n \to \F_p$ with respect to the discrete topology on $\F_p$ and the profinite topology on $G$. 
Following \cite[\S 2.2]{Se} the differential $\delta \colon \Ch^n(G,\F_p) \to \Ch^{n+1}(G,\F_p)$, which is defined by 
\begin{align*}
(\delta \varphi)(g_1, \ldots, g_{n+1}) & = \varphi(g_2, \ldots,g_{n+1}) \\
& + \sum_{i=1}^n (-1)^i \varphi(g_1, \ldots, g_ig_{i+1}, \ldots, g_{n+1}) \\
&  + (-1)^{n+1} \varphi(g_1,\ldots,g_n), 
\end{align*}
turns $\Cb(G,\F_p)$ into a cochain complex whose cohomology $\Hb(G,\F_p)$ is the continuous cohomology of $G$ with coefficients in the trivial $G$-module $\F_p$.  
In particular, $H^1(G,\F_p)$ is the group of continuous group homomorphisms $G \to \F_p$. 
The cohomology $\Hb(G,\F_p)$ is equipped with a cup-product defined as follows. 
For every $\varphi \in \Ch^i(G,\F_p)$ and $\psi \in \Ch^j(G,\F_p)$, 
we define their product $\varphi \cup \psi \in \Ch^{i+j}(G,\F_p)$ by the formula: 
\begin{align*}
(\varphi \cup \psi)(g_1,\ldots,g_{i+j}) = 
\varphi(g_1,\ldots,g_i) \cdot \psi(g_{i+1},\ldots,g_{i+j}).  
\end{align*}
This induces the cup-product on cohomology which turns $\Hb(G,\F_p)$ into a graded 
$\F_p$-algebra. 

\begin{defn}
Let $G$ be a profinite group and $p$ be a prime number. 
We say that $G$ is \emph{$A_3$-formal with respect to $p$}, or just \emph{$A_3$-formal} if the choice of $p$ is clear, 
if the differential graded $\F_p$-algebra $\Cb(G,\F_p)$  
is $A_3$-formal with respect to $p$.  
We write $\gag$ for the canonical class $\gamma_{\Cb(G,\F_p)}$ in $\HH^{3,-1}(\Hb(G,\F_p))$ and call it the \emph{canonical class of $G$}.  
We say that $G$ is \emph{$A_{\infty}$-formal} if the differential graded $\F_p$-algebra $\Cb(G,\F_p)$ is $A_{\infty}$-formal. 
\end{defn}

As a first example, we observe the following: 

\begin{example}
Let $G$ be a profinite group such that $\Hb(G,\F_p)$ is Koszul,  
and assume $H^2(G,\F_p)=0$. 
It then follows from Proposition \ref{prop:Koszul_with_H2_equal_zero_is_A3formal} that 
$G$ is $A_3$-formal with respect to $p$. 
\end{example}

Free pro-$p$ groups satisfy the following much stronger property (see also Remark \ref{rem:intrinsically_A3_formal}):  
A graded $\F_p$-algebra $A$ is called {\em intrinsically} $A_{\infty}$-formal 
if every differential graded algebra $\Ah$ with $\Hb(\Ah)\cong A$ is $A_{\infty}$-formal.

\begin{proposition}\label{prop:free_is_formal}
Let $G$ be a free pro-$p$ group. Then $G$ is intrinsically $A_{\infty}$-formal. 
In particular, $G$ is $A_3$-formal. 
\end{proposition}
\begin{proof}
Since the inclusion map of the Koszul complex into the bar complex is the identity, $\Hb(G,\F_p)$ is Koszul. 
Hence we may compute the Hochschild cohomology of $\Hb(G,\F_p)$ using the Koszul complex. 
Since $H^i(G,\F_p)$ is trivial for $i\ge 2$, the groups $\HH^{n,s}(\Hb(G,\F_p))$ for $n+s\ge 2$ are trivial. 
In particular, we have $\HH^{n,2-n}(\Hb(G,\F_p))$ for all $n \ge 3$. 
By Kadeishvili's theorem, proven also in \cite[Theorem 4.7, page 85]{ST}, this implies that $\Cb(G,\F_p)$ is intrinsically $A_{\infty}$-formal. 
\end{proof} 

For an example of a finite group we note the following: 

\begin{example}\label{example:finite_Demushkin_is_A3_formal}
Recall that the only finite Demushkin group is $G=\Z/2\Z$. 
In this case, $\Hb(G,\F_2)$ is isomorphic to the polynomial algebra $\F_2[x]$ in one generator. 
The latter is a Koszul algebra with no relations and hence $K_n^n(\Hb(G,\F_2))$ is trivial for all $n\ge 2$. 
As in the proof of Proposition \ref{prop:free_is_formal},  
this implies that $\Hb(G,\F_2)$ is intrinsically formal. 
See also \cite[Lemma 6.2]{PQ}.
\end{example}


\subsection{Dwyer's criterion}\label{sec:Dwyer}

We now recall from \cite[Theorem 2.4]{Dwyer} that the vanishing of triple Massey products in group cohomology 
can be characterised as follows. 
Let $U_n(\F_p)$ denote the group of all upper triangular unipotent $(n\times n)$-matrices with coefficients in $\F_p$. 
Let $Z_n(\F_p)$ denote the center of $U_n(\F_p)$, i.e., the subgroup of all matrices in $U_n(\F_p)$ with all off-diagonal entries being $0$ except at position $(1,n)$. 
Write $\bU_n(\F_p) = U_n(\F_p)/Z_n(\F_p)$. 

\begin{notn} 
We let $e_{ij} \colon U_n(\F_p) \to \F_p$ denote the projection to the $(i,j)$-coordinate. 
\end{notn}
The following result is a special case of \cite[Theorem 2.4]{Dwyer}. 

\begin{theorem}[Dwyer]\label{thm:Dwyer_alt1}
Let $G$ be a profinite group. 
Let $a_1, a_2, a_3 \in H^1(G,\F_p)$. 
There is a one-one correspondence $M \leftrightarrow \brr_M$ between defining systems $M$ for $\langle a_1, a_2, a_3 \rangle$ 
and continuous group homomorphisms $\brr_M \colon G \to \bU_4(\F_p)$ such that 
$e_{i,i+1} \circ (\brr_M) = - a_i$ for $i=1,2,3$. 
The correspondence is given by sending a defining system $M=\{a_{i,j}\}$ to the 
continuous group homomorphism $\brr \colon G \to \bU_4(\F_p)$ given by $e_{i,j} \circ \brr = -a_{i,j}$ for $1 \le i < j \le 4$. 
Moreover, the element $\langle a_1, a_2, a_3 \rangle_M \in H^2(G,\F_p)$ vanishes if and only if the dotted arrow in the diagram 
\begin{align*}
\xymatrix{
 & & & \ar@{.>}[dl]_-{\rr_M} G \ar[d]^{\brr_M} & \\
0 \ar[r] & Z_4(\F_p) \ar[r] & U_4(\F_p) \ar[r] & \bU_4(\F_p) \ar[r] & 0 
}
\end{align*}
exists and makes the diagram commutative. 
\end{theorem}

For later references, we formulate a particular consequence of Dwyer's theorem. 

\begin{corollary}[Dwyer]\label{cor:Dwyer_consequence} 
Let $G$ be a profinite group. 
Let $a_1, a_2, a_3 \in H^1(G,\F_p)$ such that $a_1 \cup a_2 = a_2 \cup a_3 = 0$. 
Let $\{a_{i,j}\}$ be a defining system for the triple Massey product $\langle a_1, a_2, a_3 \rangle$, 
and let $\brr \colon G \to \bU_4(\F_p)$ 
be the corresponding continuous group homomorphism. 
Then the cocycle $a_{1,2} \cup a_{2,4} + a_{1,3} \cup a_{3,4}$ is a coboundary 
if and only if $\brr$ extends to a continuous group homomorphism $\rr \colon G \to U_4(\F_p)$. 
\end{corollary}


\begin{remark}\label{rem:unipotent3_and_cup_product}
A special case of Dwyer's result is the vanishing of the cup-product itself 
which we now explain for later reference. 
Let $G$ be a profinite group. 
Let $\chi_1, \chi_2 \in H^1(G,\F_p)$. 
%
Since both $\Zh^1$ and $H^1$ are given by the vector space of group homomorphisms $G \to \F_p$,  
we identify $\chi_1$ and $\chi_2$ with its cocycle representatives and will just write $\chi_{1}$ and $\chi_{2}$ for the corresponding cocycles. 
Then we have $\chi_1 \cup \chi_2 = 0$ if and only if there is a continuous group homomorphism 
$\varphi \colon G \to U_3(\F_p)$ such that 
$e_{1,2} \circ \varphi = - \chi_1$ and $e_{2,3} \circ \varphi = - \chi_2$.  
In particular, the continuous map $\eta \coloneqq e_{1,3} \circ \varphi \colon G \to \F_p$ is a cochain in $\Ch^1(G,\F_p)$ 
such that $\delta \eta = - \chi_1 \cup \chi_2$. 
\end{remark}


\begin{remark}\label{rem:A3_formal_Koszul_group_coh}
Let $G$ be a profinite group such that $\Hb(G,\F_p)=:\Hb$ is a Koszul algebra.  
Let $R \subset H^1 \ot H^1$ denote the relations in $\Hb$ 
such that $\Hb = T(H^1)/(R)$. 
We write $\Cb = \Cb(G,\F_p)$ and let $\Zh^{\bbb}$ denotes the cocycles in $\Cb$. 
%
%
We now describe how we can use Dwyer's Theorem to analyse the canonical class of $G$. 
By Proposition \ref{prop:canonical_class_via_Koszul_complex}, 
we need to find an $\F_p$-linear map $f_1 \colon H^1 \to \Zh^1$ which induces the identity on cohomology. 
Again, since both $\Zh^1$ and $H^1$ are given by the vector space of group homomorphisms $G \to \F_p$ 
we consider $f_1$ as an identification of $H^1$ with $\Zh^1$ and will omit it from the notation. 
Since $H^2 = (H^1 \ot H^1)/R$, we can find an $\F_p$-linear map $f_2 \colon R \to \Ch^1$ such that 
\[
\dd f_2(\chi_1 \ot \chi_2) = - \chi_1 \cup \chi_2
\] 
for $\chi_1 \ot \chi_2 \in R$. 
Let $\chi_1 \ot \chi_2 \ot \chi_3 \in K_3^3(\Hb) = (H^1 \otimes R) \cap (R \otimes H^1)$ be a decomposable tensor. 
Since $H^2 = (H^1 \ot H^1)/R$, 
the triple Massey product $\langle \chi_1, \chi_2, \chi_3 \rangle$ is defined, 
and 
\begin{align}\label{eq:defining_system_Dwyer_Koszul_group_cohom}
\{\chi_1, \chi_2, \chi_3, -f_2(\chi_1 \ot \chi_2), -f_2(\chi_2 \ot \chi_3)\}
\end{align}
is a defining system. 
Moreover, $\kappa_3 (\chi_1 \ot \chi_2 \ot \chi_3) \in H^2$ is given by the class of the cocycle  
\begin{align*}
\Psi_3(\chi_1 \ot \chi_2 \ot \chi_3) = - \chi_1  \cup f_2(\chi_2 \ot \chi_3) - f_2(\chi_1 \ot \chi_2) \cup \chi_3, 
\end{align*}
and we have 
$\kappa_3 (\chi_1 \ot \chi_2 \ot \chi_3) \in \langle \chi_1, \chi_2, \chi_3 \rangle$. 
By Dwyer's Theorem \ref{thm:Dwyer_alt1}, 
the defining system \eqref{eq:defining_system_Dwyer_Koszul_group_cohom} 
corresponds to the continuous group homomorphism 
\begin{align*}
\brr(\chi_1 \ot \chi_2 \ot \chi_3) \colon G \longrightarrow  \bU_4(\F_p) 
\end{align*}
given by 
\begin{align*}
\brr(\chi_1 \ot \chi_2 \ot \chi_3) =
\begin{pmatrix}
 1 & - \chi_1 & f_2(\chi_1 \ot \chi_2) &  * \\
  & 1& - \chi_2 & f_2(\chi_2 \ot \chi_3)  \\
  & & 1 & - \chi_3 \\
  & & & 1 
\end{pmatrix}. 
\end{align*} 
Thus, $\kappa_3(\chi_1 \ot \chi_2 \ot \chi_3)$ vanishes 
if and only if $\brr(\chi_1 \ot \chi_2 \ot \chi_3)$ can be extended to a continuous group homomorphism 
\begin{align*}
\rr(\chi_1 \ot \chi_2 \ot \chi_3) \colon G \to U_4(\F_p).
\end{align*} 
In particular, the continuous map $\vartheta \coloneqq - e_{1,4} \circ \varphi \colon G \to \F_p$ is a cochain in $\Ch^1(G,\F_p)$ 
such that $\delta \vartheta = \Psi_3(\chi_1 \ot \chi_2 \ot \chi_3)$. 
We will make frequent use of this observation in Sections \ref{sec:A3_formality_for_Demushkin_groups} and 
\ref{sec:proof_of_Dem_thm} by expressing elements of $K_3^3(\Hb)$ as linear combinations of suitable decomposable tensors 
to compute $\kappa_3$ on $K_3^3(\Hb)$. 
%
\end{remark}

\subsection{Demushkin groups are Koszul}\label{sec:Demushkin_groups_are_Koszul}

We now show that the cohomology algebra of a Demushkin group is $A_3$-formal. 
We will deduce this fact from a more general result on quadratic algebras which is well known 
and proven for example in \cite[Proposition 2.3 in Chapter 2 on page 24, and Proposition 4.2 of Chapter 6 on page 124]{ppqa}. 
We provide a proof for completeness and convenience of the reader.

Let $A = T(V)/(R)$ be a quadratic algebra over a field $\F$ with $\dim_{\F}A^1 = \dim_{\F} V = d$ and $\dim_{\F}A^2 = 1$. 
We will show that $A$ is a Koszul algebra. 
To do so, we recall that the Hilbert series of a graded vector space $\Vh$ is the formal power series given by 
\[
h_{\Vh}(z) = \sum_{n \in \Z} (\dim_{\F}\Vh^n) \cdot z^n.
\]
In particular, the Hilbert series of $A$ is $h_A(z) = 1+dz+z^2$. 

\begin{lemma}\label{lemma:dim_of_Knn}
Let $A = T(V)/(R)$ be a quadratic algebra with $\dim_{\F}A^1 = d$. 
The dimension $b_n \coloneqq \dim_{\F} K_n^n(A)$ is given by the recursive formula $b_0=1$, $b_1=d$, and $b_{n+1} = d \cdot b_n - b_{n-1}$ for $n\ge 1$.  
Therefore, the Hilbert series of $K_{\bbb}^{\bbb}(A)$ equals $(1 - dz + z^2)^{-1}$, 
i.e., $h_A(-z) \cdot h_{K_{\bbb}^{\bbb}(A)}(z) = 1$. 
\end{lemma}
\begin{proof}
Let $\sum_{n \in \Z} b_n \cdot z^n$ be the formal powers series such that 
\[
\left( \sum_{n \in \Z} b_n \cdot z^n \right) \cdot (1 - dz + z^2)^{-1} = 1.
\]  
Then we have $b_n = 0$ for $n<0$, $b_0=1$, $b_1=d$ and, for all $n \ge 1$, 
\[
b_{n+1} z^{n+1} - (d z) \cdot (b_n \cdot z^{n}) + z^2 \cdot (b_{n-1} \cdot z^{n-1}) = 0. 
\] 
Thus, the coefficients are determined by the equation $b_{n+1} = d \cdot b_n - b_{n-1}$ for all $n \ge 1$. 
Since $\dim K_0^0(A) = 1$, $\dim K_1^1(A) = d$, and $\dim K_2^2(A) = \dim R = d^2-1$, 
it remains to prove the recursive formula for $\dim_{\F_p} K_n^n(A)$ to prove the lemma. 
%
For $n \ge 2$, we 
have $K_{n+1}^{n+1}(A) = (K_{n}^{n}(A) \ot V) \cap (V^{\ot n-1} \ot R)$.   
We then get 
\begin{align*}
\dim K_{n+1}^{n+1}(A) = & \dim (K_{n}^{n}(A) \ot V) + \dim (V^{\ot n-1} \ot R) \\
 & ~ ~ - \dim ((K_{n}^{n}(A) \ot V) \oplus (V^{\ot n-1} \ot R)).  
\end{align*}
Since $\dim (K_{n}^{n}(A) \ot V) = d \cdot \dim K_{n}^{n}(A)$, it suffices to determine the difference 
\[
\Delta \coloneqq \dim (V^{\ot n-1} \ot R) - \dim ((K_{n}^{n}(A) \ot V) \oplus (V^{\ot n-1} \ot R)).
\] 
The integer $\Delta$ is determined by how much the space $(K_{n}^{n}(A) \ot V) \oplus (V^{\ot n-1} \ot R)$ exceeds 
the space $V^{\ot n-1} \ot R$, i.e., 
it is given by the dimension of the quotient space 
$Q \coloneqq ((K_{n}^{n}(A) \ot V) \oplus (V^{\ot n-1} \ot R))/(V^{\ot n-1} \ot R)$. 
Now we use the assumption that $\dim A^2=1$ which means that we can choose a single element $\omega \in V \ot V$ 
whose image generates the one-dimensional quotient space $A^2 = (V \ot V)/R$. 
The space $Q$ is then isomorphic to 
\[
K_{n-1}^{n-1}(A) \ot \mathrm{span}_{\F} (\omega) = ((R \ot V^{\ot n-3}) \cap \cdots \cap (V^{\ot n-3} \ot R)) \ot \mathrm{span}_{\F}(\omega). 
\] 
%
In other words, we get $\dim Q = \dim K_{n-1}^{n-1}(A)$. 
This proves the recursive formula and the assertion of the lemma. 
\end{proof}


\begin{proposition}\label{prop:A_is_Koszul}
Let $A = T(V)/(R)$ be a quadratic algebra with 
$\dim_{\F}A^2 = 1$. 
Then $A$ is a Koszul algebra. 
\end{proposition}
\begin{proof}
By Lemma \ref{lemma:dim_of_Knn}, we have $h_A(-z) \cdot h_{K_{\bbb}^{\bbb}(A)}(z) = 1$. 
This implies that the sequence 
\begin{align*}
\cdots \to A \ot K_{n+1}^{n+1}(A) \to A \ot K_{n}^{n}(A) \to A \ot K_{n-1}^{n-1}(A) \to \cdots \to A \to \F \to 0
\end{align*}
is exact. 
Thus, $A \ot K_{\bbb}^{\bbb}(A)$ provides a free resolution of $\F$. 
This implies that $K(A)$ is a minimal free resolution of $A$ as an $A$-bimodule 
by \cite[Proposition 3.1]{vandenbergh}, 
where we note that, by \cite[Section 2.8]{BGS}, the complex $K'(A)$ in \cite{vandenbergh} is isomorphic to the complex we denote by $K(A)$. 
The two-out-of-three property of quasi-isomorphisms then implies that $A$ is Koszul as defined in Definition \ref{def:Koszulity} 
\end{proof}

We already know that the only finite Demushkin group $\Z/2\Z$ is Koszul by Example \ref{example:finite_Demushkin_is_A3_formal}. 
The infinite case now follows from Proposition \ref{prop:A_is_Koszul} and Definition \ref{def:Demushkin}. 

\begin{corollary}\label{cor:Demushkin_groups_are_Koszul}
Let $G$ be a pro-$p$ Demushkin group. 
Then the cohomology algebra $\Hb(G,\F_p)$ is a Koszul algebra. \qed
\end{corollary}


\begin{remark}\label{rem:alternative_proof_that_Demushkin_are_Koszul}
Corollary \ref{cor:Demushkin_groups_are_Koszul} also follows from \cite[Theorem 5.2]{MPQT} where a stronger result is proven. 
The proof of \cite[Theorem 5.2]{MPQT} uses the fact that the dual algebra $(\Hb)^! = T((H^1)^*)/(R^\perp)$ 
of $\Hb = \Hb(G,\F_p)$ 
is a quadratic algebra which satisfies the assumption of \cite[Lemma 2.15]{MPQT}, 
i.e., $\chi_1^* \ot \chi_1^* \notin R^{\perp}$ while $\chi^*_1 \ot \chi_2^* \in R^{\perp}$, 
where $\chi_i^*$ denotes basis vectors of $(H^1)^*$ which are dual to the  $\chi_i$. 
This implies that $(\Hb)^!$ has a Poincar\'e--Birkhoff--Witt basis which is known to imply that an algebra is Koszul. 
Then it remains to use the general fact that a locally finite quadratic algebra is Koszul if and only if its quadratic dual is Koszul. 
\end{remark}


\section{$A_3$-formality for Demushkin groups at odd primes}\label{sec:A3_formality_for_Demushkin_groups}

In this section, we discuss $A_3$-formality for pro-$p$ Demushkin groups. 
%


\subsection{Demushkin groups with $q$-invariant $q \ne 2,3$ are $A_3$-formal}\label{sec:Demushkin_q_ge_5}

First we prove $A_3$-formality for Demushkin groups with two generators. 
Even though we generalise the following result to any even number of generators, 
we prove a special case first since it demonstrates the main idea of the argument. 

\begin{theorem}\label{thm:Demushkin_two_generators}
Let $p$ be an odd prime number and let $q = p^f$ with $f \ge1$ and $f\ge 2$ if $p=3$, or $q = 0$. 
Let $G$ be the pro-$p$ group generated by elements $x_1$ and $x_2$ and the single relation 
$x_1^q [x_1,x_2]=1$. 
Then $G$ is $A_3$-formal. 
\end{theorem}


\begin{remark}\label{rem:Positselski_example_as_semidirect_product}
A group with the presentation as in Theorem \ref{thm:Demushkin_two_generators} can  be realised as follows. 
Let $\car \colon \Z_p \to 1 + q\Z_p$ be a cyclotomic character on $\Z_p$ with 
$q = p^f$. 
Then $G \coloneqq \Z_p \rtimes_{\car} \Z_p$ is a Demushkin group generated by $x_1$ and $x_2$ subject to the single relation 
$x_1^q [x_1,x_2]=1$. 
\end{remark}

\begin{remark}
Let $G$ be as in Theorem \ref{thm:Demushkin_two_generators} and Remark \ref{rem:Positselski_example_as_semidirect_product}. 
We note that the cohomology algebra $\Hb(G,\F_p)$ is an exterior $\F_p$-algebra with two generators. 
It therefore follows from Example \ref{example:exterior_algebra_on_2_gen} 
and Remark \ref{rem:dgas_with_exterior_algebra_as_cohomology} 
that $\HH^{3,-1}(\Hb(G,\F_p))$ is non-trivial. 
In particular, pro-$p$ Demushkin groups are not intrinsically $A_3$-formal in general. 
\end{remark}


\begin{notn}
For matrices $M$ and $N$ in $U_n(\F_p)$, we write $[M,N] \coloneqq M^{-1}N^{-1}MN$. 
\end{notn}


In the following proofs and constructions we will frequently use the following observation, often without explicitly mentioning it: 

\begin{lemma}\label{lemma:matrix_relations}
Let $M_1, \ldots, M_d \in U_n(\F_p)$ be a sequence of matrices and let $I_n \in U_n(\F_p)$ denote the identity matrix. 
If 
\begin{align*}
M_1^q[M_1,M_2] [M_3,M_4] \cdots [M_{d-1},M_d] = I_n ~ \text{in} ~ U_n(\F_p), 
\end{align*} 
then the assignment $\rr \colon x_i \mapsto M_i$ for $x_1,\ldots,x_d$ 
defines a continuous group homomorphism $\rr \colon G \to U_n(\F_p)$. 
\end{lemma}
\begin{proof}
This follows directly from the defining relations for $G$ and the fact that $U_n(\F_p)$ is a finite $p$-group. 
\end{proof}


\begin{proof}[Proof of Theorem \ref{thm:Demushkin_two_generators}]
%
Let $\Cb = \Cb(G,\F_p)$ denote the complex of continuous inhomogeneous cochains, let $\Zh^\bbb$ denote the cocycles, 
and let $\Hb = \Hb(G,\F_p)$ denote the corresponding cohomology algebra. 
We consider $\Ch^n$ and $\Zh^n$ as $\F_p$-vector spaces with addition and scalar multiplication defined pointwise. 
%
%
We know that $\Hb$ is a Koszul algebra by Corollary \ref{cor:Demushkin_groups_are_Koszul}. 
Hence we can use Proposition \ref{prop:canonical_class_via_Koszul_complex} to construct the canonical class of $G$. 

Let $\chi_1, \chi_2 \colon G \to \F_p$ be a basis of $H^1$ such that $\chi_i(x_j) = - \delta_{ij}$, where $\delta_{ij}$ denotes the Kronecker symbol. 
The minus sign is chosen so that we minimise the number of signs on the forthcoming formulas. 
We have $\Hb = T(H^1)/(R)$ 
where $R \subset H^1 \ot H^1$ is the $\F_p$-vector subspace 
\begin{align*}
R = \F_p \langle \chi_1 \ot \chi_1, \chi_2 \ot \chi_2, \chi_1 \ot \chi_2 + \chi_2 \ot \chi_1 \rangle 
\end{align*}
and $H^2 = \F_p \langle \chi_1 \cup \chi_2 \rangle$. 
The vector space $K^3_3(\Hb) = (H^1\ot R) \cap (R\ot H^1)$ is then given by 
$\F_p$-linear span 
\begin{align*}
K^3_3(\Hb) = \F_p \langle & \chi_1 \ot \chi_1 \ot \chi_1, \chi_2 \ot \chi_2 \ot \chi_2, \\
 & ~ \chi_2 \ot \chi_1 \ot \chi_1 + \chi_1 \ot \chi_2 \ot \chi_1 + \chi_1 \ot \chi_1 \ot \chi_2, \\
& ~ \chi_1 \ot \chi_2 \ot \chi_2 + \chi_2 \ot \chi_1 \ot \chi_2 + \chi_2 \ot \chi_2 \ot \chi_1 \rangle.
\end{align*}
Now we choose $\F_p$-linear maps $f_1 \colon H^1 \to \Zh^1$ and $f_2 \colon R \to \Ch^1$ to construct the canonical class of $G$.  
Since both $\Zh^1$ and $H^1$ are given by the vector space of group homomorphisms $G \to \F_p$,  
we consider $f_1$ as an identification of $H^1$ with $\Zh^1$ and will omit it from the notation. 
We define the $\F_p$-linear map $f_2 \colon R \to \Ch^1$ as follows. 
Let $A$ denote the matrix 
\begin{align*}
A = \begin{pmatrix} 
1 & 1 & 0  \\
0 & 1 & 1  \\
0 &  0 & 1  \\
\end{pmatrix} 
~ \text{with} ~ 
A^n = \begin{pmatrix} 
1 & n & \binom{n}{2} \\
0 & 1 & n  \\
0 &  0 & 1  
\end{pmatrix} 
~ \text{for} ~ n \ge 0,
\end{align*}
where $\binom{n}{k}$ denotes the binomial coefficient and we set $\binom{n}{k}=0$ when $n<k$. 
In particular, we have $A^q = I_3 \in U_3(\F_p)$ for $q= p^f$ as $p$ is odd,   
and $[A^n,A^m]=I_3$ for all $n,m$.   
%
By Lemma \ref{lemma:matrix_relations}, 
we can define a continuous group homomorphism $\varphi \colon G \to U_3(\F_p)$  by setting $\varphi(x_j) = A$ for $j=1,2$. 
Recall that, for a matrix $M$, we let $e_{ij}(M)$ denote the entry in position $(i,j)$ in $M$. 
%
By Remark \ref{rem:unipotent3_and_cup_product}, 
the continuous map $\eta \colon G \to \F_p$ defined by $\eta(g) \coloneqq e_{13}(\varphi(g))$ is then a cochain in $\Ch^1$ such that 
\begin{align*}
\delta \eta = - (\chi_1 + \chi_2) \cup (\chi_1 + \chi_2) 
= - \chi_1 \cup \chi_1 - \chi_2 \cup \chi_2 - (\chi_1 \cup \chi_2 + \chi_2 \cup \chi_1).
\end{align*}
For $i=1,2$, by Lemma \ref{lemma:matrix_relations}, 
we define a continuous group homomorphism $\varphi_i \colon G \to U_3(\F_p)$ 
by setting $\varphi(x_j) = A^{\delta_{ij}}$. 
Again by Remark \ref{rem:unipotent3_and_cup_product}, 
the continuous map $\eta_i \colon G \to \F_p$, for $i=1,2$, 
defined by $g \mapsto e_{13}(\varphi_i(g))$ is then a cochain in $\Ch^1$ such that 
\begin{align*}
\delta \eta_i = - \chi_i \cup \chi_i. 
\end{align*}
We define the $\F_p$-linear map $f_2 \colon R \to \Ch^1$ on the basis element $\chi_i \cup \chi_i$ to be the continuous map $G \to \F_p$ 
given by 
\begin{align*}
f_2(\chi_i \ot \chi_i) \coloneqq \eta_i 
\end{align*}
for $i=1,2$. 
Using the $\F_p$-vector space structure on $\Ch^1$, we then set  
\begin{align*}
f_2(\chi_1 \ot \chi_2 + \chi_2 \ot \chi_1) \coloneqq \eta - \eta_1 -\eta_2. 
\end{align*}
The map \eqref{eq:def_of_Psi3_construction} then becomes the map $\Psi_3 \colon K_3^3(\Hb) \to \Zh^2$ given by  
\begin{align*}
\Psi_3(\chi_a, \chi_b, \chi_c) = - \chi_a \cup f_2(\chi_b \ot \chi_c) - f_2(\chi_a \ot \chi_b) \cup \chi_c. 
\end{align*}
Taking the cohomology class induces the map $\kappa_3 \colon K^3_3(\Hb) \to H^2$ 
which represents the canonical class of $G$.

Now we use Dwyer's Theorem \ref{thm:Dwyer_alt1} 
and Remark \ref{rem:A3_formal_Koszul_group_coh} to show that the map $\kappa_3$ 
is trivial.  
%
To do so, it suffices to show that $\kappa_3$ vanishes on each basis element of $K^3_3(\Hb)$. 
We note that 
\begin{align*}
(\chi_1 + \chi_2)^{\ot 3} & = \chi_1^{\ot 3} + \chi_2^{\ot 3} +  (\chi_2 \ot \chi_1 \ot \chi_1 + \chi_1 \ot \chi_2 \ot \chi_1 + \chi_1 \ot \chi_1 \ot \chi_2) \\
& + (\chi_1 \ot \chi_2 \ot \chi_2 + \chi_2 \ot \chi_1 \ot \chi_2 + \chi_2 \ot \chi_2 \ot \chi_1),
\end{align*}
and 
\begin{align*}
(\chi_1 - \chi_2)^{\ot 3} & = \chi_1^{\ot 3} - \chi_2^{\ot 3} - (\chi_2 \ot \chi_1 \ot \chi_1 + \chi_1 \ot \chi_2 \ot \chi_1 + \chi_1 \ot \chi_1 \ot \chi_2) \\
& + (\chi_1 \ot \chi_2 \ot \chi_2 + \chi_2 \ot \chi_1 \ot \chi_2 + \chi_2 \ot \chi_2 \ot \chi_1). 
\end{align*}
Since $p$ is odd, it therefore suffices to show that $\kappa_3$ vanishes 
on the elements  
$\chi_1^{\ot 3}$, $\chi_2^{\ot 3}$, $(\chi_1 + \chi_2)^{\ot 3}$, and $(\chi_1 - \chi_2)^{\ot 3}$.

Let $B_+$ denote the matrix 
\begin{align}\label{eq:def_of_matrix_B_+}
B_+ = \begin{pmatrix} 
1 & 1 & 0 & 0 \\
0 & 1 & 1 & 0 \\
0 &  0 & 1 & 1 \\
0 & 0 & 0 & 1 
\end{pmatrix} 
~ \text{with} ~ 
B_+^n = \begin{pmatrix} 
1 & n & \binom{n}{2} & \binom{n}{3} \\
0 & 1 & n & \binom{n}{2} \\
0 &  0 & 1 & n \\
0 & 0 & 0 & 1 
\end{pmatrix} 
~ \text{for} ~ n \ge 0
\end{align}
where we set $\binom{n}{k}=0$ whenever $n<k$. 
In particular, we have $B_+^q = I_4 \in U_4(\F_p)$ for $q= p^f$ 
since $p$ is odd and $f \ge 2$ if $p=3$, 
and $[B_+,B_+]=I_4$. 
%
By Lemma \ref{lemma:matrix_relations}, 
we can therefore, for $i=1,2$, 
define a continuous group homomorphism $\rr_i \colon G \to U_4(\F_p)$ 
by setting $\rr_i(x_j) = B_+^{\delta_{ij}}$ for $j=1,2$. 
By Dwyer's Theorem \ref{thm:Dwyer_alt1}, 
the homomorphism $\rr_i$ corresponds to a defining system of the triple Massey product $\langle \chi_i, \chi_i, \chi_i \rangle$. 
Thus, by Remark \ref{rem:A3_formal_Koszul_group_coh} and the construction of $f_2$, 
the continuous map $\vartheta_i \colon G \to \F_p$ defined by $\vartheta_i(g) \coloneqq - e_{14}(\rr_i(g))$ is a cochain in $\Ch^1$ 
which witnesses 
the vanishing of the triple Massey product $\langle \chi_i, \chi_i, \chi_i \rangle$, i.e., 
such that 
\begin{align*}
\delta \vartheta_i = \psi_3(\chi_i^{\ot 3}) 
\end{align*}
for $i=1,2$. 
This shows that $\kappa_3(\chi_i^{\ot 3}) =0$ for $i=1,2$.

Now we define a continuous group homomorphism $\rr_+ \colon G \to U_4(\F_p)$ 
by setting $\rr_+(x_j) = B_+$ for $j=1,2$. 
By Dwyer's Theorem \ref{thm:Dwyer_alt1}, 
the homomorphism $\rr_+$ corresponds to a defining system of the triple Massey product 
\[
\langle \chi_1 + \chi_2, \chi_1 + \chi_2, \chi_1 + \chi_2 \rangle.
\]
Moreover, by Remark \ref{rem:A3_formal_Koszul_group_coh}, 
the continuous map $\vartheta_+ \colon G \to \F_p$ defined by $\vartheta_+(g) \coloneqq - e_{14}(\rr_+(g))$ is a cochain in $\Ch^1$ 
such that 
\begin{align*}
\delta \vartheta_+ = \Psi_3 ((\chi_1 + \chi_2)^{\ot 3}). 
\end{align*}
This shows $\kappa_3((\chi_1 + \chi_2)^{\ot 3} = 0$.  

Now let $B_-$ denote the matrix 
\begin{align*}
B_- = \begin{pmatrix} 
1 & -1 & 0 & 0 \\
0 & 1 & -1 & 0 \\
0 &  0 & 1 & -1 \\
0 & 0 & 0 & 1 
\end{pmatrix}. 
\end{align*}
We have  
$[B_+, B_-] =I_4$, 
and hence the relation $B_+^q[B_+,B_-]=I_4$ holds in $U_4(\F_p)$. 
%
%
We can therefore define a continuous group homomorphism $\rr_- \colon G \to U_4(\F_p)$ 
by setting $\rr_-(x_1) = B_+$ and $\rr_-(x_2) = B_-$. 
By Dwyer's Theorem \ref{thm:Dwyer_alt1}, 
the homomorphism $\rr_-$ corresponds to a defining system of the triple Massey product $\langle \chi_1 - \chi_2, \chi_1 - \chi_2, \chi_1 - \chi_2 \rangle$. 
Moreover, by Remark \ref{rem:A3_formal_Koszul_group_coh}, 
the continuous map $\vartheta_- \colon G \to \F_p$ defined by $\vartheta_-(g) \coloneqq - e_{14}(\rr_-(g))$ is a cochain in $\Ch^1$ 
such that 
\begin{align*}
\delta \vartheta_- = \Psi_3 ((\chi_1 - \chi_2)^{\ot 3} ). 
\end{align*}
This shows $\kappa_3 ((\chi_1 - \chi_2)^{\ot 3})=0$.  
This proves that the canonical class of $G$ vanishes and proves the theorem. 
\end{proof}


In fact, we can generalise Theorem \ref{thm:Demushkin_two_generators} and show that all other Demushkin pro-$p$ groups at odd primes are $A_3$-formal.

\begin{theorem}\label{thm:Demushkin_groups_are_formal}
Let $p$ be an odd prime number, let $d\ge 2$ be an even number, 
and let $q = p^f$ with $f \ge 1$ and $f \ge 2$ if $p = 3$, or $q = 0$.  
Let $G$ be a pro-$p$ group with minimal set of generators $\{x_1,\ldots,x_d\}$ 
satisfying the single relation 
\begin{align*}
1 = x_1^q [ x_1,x_2] [x_3,x_4] \cdots [x_{d-1}, x_d].
\end{align*}
Then the canonical class of $G$ vanishes and $G$ is $A_3$-formal. 
\end{theorem}

The case $q=3$ will be discussed in Section \ref{sec:Demushkin_q=3}. 
Our proof of Theorem \ref{thm:Demushkin_groups_are_formal} is based on a direct but rather long computation of the canonical class of $G$.  
We construct the canonical class for all $q \ne 2$ in Section \ref{sec:Demushkin_canonical_class} below. 
We provide the proof of Theorem \ref{thm:Demushkin_groups_are_formal} 
in Section \ref{sec:proof_of_Dem_thm}.


\subsection{The canonical class of a Demushkin group}\label{sec:Demushkin_canonical_class}

We will now construct the canonical class of pro-$p$ Demushkin groups with an arbitrary even number of generators. 
Let $d \ge 2$ be an even number. 
We assume that $p$ is an odd prime number and $q = p^f$ with $f\ge 1$, or $q = 0$. 
Let $G$ be a pro-$p$ group with minimal set of generators $\{x_1,\ldots,x_d\}$ 
satisfying the single relation 
\begin{align*}
1 = x_1^q [ x_1,x_2] [x_3,x_4] \cdots [x_{d-1}, x_d].
\end{align*}
Such a group $G$ is a Demushkin group which is completely characterised by the invariants $d$ and $q$, see \cite{Demushkin}, \cite{Labute} and \cite{SerreDem}. 

Let $\Cb = \Cb(G,\F_p)$ denote the complex of continuous inhomogeneous cochains, let $\Zh^\bbb = \Zh^\bbb(G,\F_p)$ denote the cocycles, 
and let $\Hb = \Hb(G,\F_p)$ denote the corresponding cohomology algebra. 
%
Let $\{\chi_1,\ldots,\chi_d\}$ be an $\F_p$-basis of $H^1$  
such that $\chi_i(x_j) = - \delta_{ij}$, where $\delta_{ij}$ denotes the Kronecker symbol. 
The $\F_p$-vector space $H^2$ is then generated by the single element $\chi_1 \cup \chi_2$. 
We have $\Hb = T(H^1)/(R)$ 
where, by \cite[Proposition 3.9.13]{NSW} (see also \cite[page 27]{MPQT}), $R \subset H^1 \ot H^1$ is the $\F_p$-vector subspace 
$R = \mathrm{span}(\BBB)$ where $\BBB$ is the set 
\begin{align*}
\BBB = & \{ 
\chi_i \ot \chi_j ~ \text{for all} ~ 1 \le i, j \le d  ~ 
 ~ \text{with}~ \begin{cases}
i \ne j+1 &  \text{if} ~ j ~ \text{is odd} \\
i \ne j-1 &  \text{if} ~ j ~ \text{is even}
\end{cases} \\
& ~ \chi_{2i-1} \ot \chi_{2i} + \chi_{2i} \ot \chi_{2i-1}, ~ \text{for} ~ 1\le i \le d/2, \\
& ~ \chi_{1} \ot \chi_{2} + \chi_{2k} \ot \chi_{2k-1} ~ \text{for} ~ 2 \le k \le d/2
\}.
\end{align*}


\begin{example}
For $d=4$, for example, this gives 
\begin{align*}
R = \F_p \langle & \chi_1 \ot \chi_1, \chi_2 \ot \chi_2, \chi_3 \ot \chi_3, \chi_4 \ot \chi_4, \\
& \chi_1 \ot \chi_3, \chi_1 \ot \chi_4, \chi_2 \ot \chi_3, \chi_2 \ot \chi_4, 
 \chi_3 \ot \chi_1, \chi_3 \ot \chi_2, \chi_4 \ot \chi_1, \chi_4 \ot \chi_2 \\
& \chi_{1} \ot \chi_{2} + \chi_{2} \ot \chi_{1}, \chi_{3} \ot \chi_{4} + \chi_{4} \ot \chi_{3}, 
\chi_{1} \ot \chi_{2} + \chi_{4} \ot \chi_{3}
\rangle.  
\end{align*}
\end{example}


\begin{remark}
The set $\BBB$ is clearly linearly independent and generates $R$ 
since it consists of $d^2-1 = \dim (H^1 \ot H^1) -1$ elements as required. 
Note that we could have chosen other basis elements to present the relations in $H^2$.  
For example, we have the relation 
\begin{align*}
\chi_{2} \ot \chi_{1} + \chi_{2i-1} \ot \chi_{2i} = & (\chi_{2} \ot \chi_{1} + \chi_{1} \ot \chi_{2}) + (\chi_{2i-1} \ot \chi_{2i} + \chi_{2i} \ot \chi_{2i-1}) \\
& ~~ - (\chi_{1} \ot \chi_{2} + \chi_{2i} \ot \chi_{2i-1}) \in R
\end{align*}
which we will use later. 
\end{remark}


We will now construct the canonical class of $G$.  
We choose $f_1$ to be the identity $H^1 \to  \Zh^1 = \Hom(G,\F_p)$ and omit it in the notation, 
where $\Hom(G,\F_p)$ denotes the $\F_p$-vector space of continuous group homomorphisms. 
We construct the $\F_p$-linear map $f_2 \colon R \to \Ch^1$  by defining it on each element of $\BBB$ and then extend it $\F_p$-linearly.

Let $A_{1,0}$, $A_{0,1}$, and $A_{1,1}$ denote the following matrices: 
\begin{align*}
A_{1,0} = \begin{pmatrix} 
1 & 1 & 0  \\
0 & 1 & 0  \\
0 &  0 & 1  \\
\end{pmatrix}, ~
A_{0,1} = \begin{pmatrix} 
1 & 0 & 0  \\
0 & 1 & 1  \\
0 &  0 & 1  \\
\end{pmatrix} 
~ \text{and} ~ 
A_{1,1} = \begin{pmatrix} 
1 & 1 & 0  \\
0 & 1 & 1  \\
0 &  0 & 1  \\
\end{pmatrix}. 
\end{align*}
%
For $n\ge 0$, we have 
\begin{align*}
A^{n_{01}}_{0,1} \cdot A^{n_{10}}_{1,0} &= \begin{pmatrix} 
1 & n_{10} & 0 \\
0 & 1 & n_{01}  \\
0 &  0 & 1  \\
\end{pmatrix}, 
A_{1,1}^n = \begin{pmatrix} 
1 & n & \binom{n}{2} \\
0 & 1 & n  \\
0 &  0 & 1  
\end{pmatrix}, \\
A^{n_{10}}_{1,0} \cdot A^{n_{01}}_{0,1} & = \begin{pmatrix} 
1 & n_{10} & n_{10} n_{01} \\
0 & 1 & n_{01}  \\
0 &  0 & 1  \\
\end{pmatrix},  
\end{align*}
where $\binom{n}{2}$ denotes the binomial coefficient with $\binom{n}{2}=0$ for $n=0,1$.  
In particular, we have $A_{1,0}^p = A_{0,1}^p = I_3 \in U_3(\F_p)$ since $p$ is odd. 
For $i \ne j+1$ if $j$ is odd and $i\ne j-1$ if $j$ is even, 
we can therefore define a continuous group homomorphism $\varphi_{ij} \colon G \to U_3(\F_p)$  by setting 
$\varphi_{ij}(x_i) = A_{1,0}$, $\varphi_{ij}(x_j) = A_{01}$, and $\varphi_{ij}(x_k) = I_3$ for $k\neq i,j$. 
By Remark \ref{rem:unipotent3_and_cup_product},   
the continuous map $\eta_{ij} \colon G \to \F_p$ defined by $g \mapsto \eta_{ij}(g) \coloneqq e_{13}(\varphi_{ij}(g))$ is then a cochain in $\Ch^1$ such that 
\begin{align*}
\delta \eta_{ij} = - \chi_i \cup \chi_j.
\end{align*}
We define the $\F_p$-linear map $f_2 \colon R \to \Ch^1$ on the basis element $\chi_i \cup \chi_j$ to be the continuous map $G \to \F_p$ 
given by 
\begin{align*}
f_2(\chi_i \ot \chi_j) \coloneqq \eta_{ij}. 
\end{align*}

\begin{rem}
Note that the above formula does not work for the cup-product $\chi_1 \cup \chi_2$ or any $\chi_{2i-1} \cup \chi_{2i}$. 
The difference is that the commutator relation is a non-trivial condition. 
However, since $A_{1,0}^q=I_3$ and $A_{10}$ and $A_{01}$ do not commute, we have 
\[
A_{1,0}^q[A_{1,0},A_{0,1}] \neq I_3.
\]  
Thus, we do not get a group homomorphism $G \to U_3(\F_p)$  by setting 
$\varphi(x_{2i-1}) = A_{1,0}$, $\varphi(x_{2i}) = A_{0,1}$, $\varphi(x_j) = I_3$ for $j \ne 2i-1,2i$.  
\end{rem}

We have $A_{1,1}^p = I_3 \in U_3(\F_p)$ since $p$ is odd,  
and $[A_{1,1}^n,A_{1,1}^m]=I_3$ for all $n,m$.   
We define a continuous group homomorphism $\varphi_{ii} \colon G \to U_3(\F_p)$ 
by setting $\varphi_{ii}(x_j) = A_{1,1}^{\delta_{ij}}$. 
The continuous map $\eta_{ii} \colon G \to \F_p$ defined by $g \mapsto e_{13}(\varphi_{ii}(g))$ is then a cochain in $\Ch^1$ such that 
\begin{align*}
\delta \eta_{ii} = - \chi_i \cup \chi_i.
\end{align*}
%
We define the $\F_p$-linear map $f_2 \colon R \to \Ch^1$ on the basis element $\chi_i \cup \chi_i$ to be the continuous map $G \to \F_p$ 
given by 
\begin{align*}
f_2(\chi_i \ot \chi_i) \coloneqq \eta_{ii}. 
\end{align*}

Now we define a continuous group homomorphism $\varphi^{\oddeven}_{i} \colon G \to U_3(\F_p)$, where the superscript $\mathrm{oe}$ stands for odd/even, 
by setting 
$\varphi^{\oddeven}_{i}(x_{2i-1}) =  \varphi^{\oddeven}_{i}(x_{2i}) = A_{1,1}$ and $\varphi^{\oddeven}_{i}(x_j) = I_3$ for $j \neq 2i-1, 2i$. 
By Remark \ref{rem:unipotent3_and_cup_product}, 
the continuous map $\eta^{\oddeven}_{i} \colon G \to \F_p$ defined by $\eta^{\oddeven}_{i}(g) \coloneqq e_{13}(\varphi^{\oddeven}_{i}(g))$ is then a cochain in $\Ch^1$ such that 
\begin{align*}
\delta \eta^{\oddeven}_{i} = - (\chi_{2i-1} + \chi_{2i}) \cup (\chi_{2i-1} + \chi_{2i}) 
= & ~ - \chi_{2i-1} \cup \chi_{2i-1} - \chi_{2i} \cup \chi_{2i} \\
& ~ - (\chi_{2i-1} \cup \chi_{2i} + \chi_{2i} \cup \chi_{2i-1}).
\end{align*}
We then define $f_2(\chi_{2i-1} \ot \chi_{2i} + \chi_{2i} \ot \chi_{2i-1})$ such that 
\begin{align*}
& ~ f_2((\chi_{2i-1} + \chi_{2i}) \ot (\chi_{2i-1} \ot \chi_{2i})) \\
= & ~ f_2(\chi_{2i-1} \ot \chi_{2i-1}) + f_2(\chi_{2i} \ot \chi_{2i}) + 
f_2(\chi_{2i-1} \ot \chi_{2i} + \chi_{2i} \ot \chi_{2i-1}).  
\end{align*}
That is, using the $\F_p$-vector space structure on $\Ch^1$, we set  
\begin{align*}
f_2(\chi_{2i-1} \ot \chi_{2i} + \chi_{2i} \ot \chi_{2i-1}) \coloneqq \eta^{\oddeven}_{i} - \eta_{2i-1,2i-1} -\eta_{2i,2i}. 
\end{align*}
%

%
Finally, we define a homomorphism $\varphi^{\oddeven}_{1,k} \colon G \to U_4(\F_p)$ by setting 
\begin{align*}
\varphi^{\oddeven}_{1,k}(x_1) = A_{1,0}, ~ 
\varphi^{\oddeven}_{1,k}(x_2) = A_{0,1}, ~  
\varphi^{\oddeven}_{1,k}(x_{2k-1}) = A_{0,1}, ~ 
\varphi^{\oddeven}_{1,k}(x_{2k}) = A_{1,0},
\end{align*} 
and $\varphi^{\oddeven}_{1,k}(x_j) = I_3$ for $j \ne 1,2,2k-1,2k$.  
By Remark \ref{rem:unipotent3_and_cup_product}, 
the continuous map $\eta^{\oddeven}_{1,k} \colon G \to \F_p$ defined by $\eta^{\oddeven}_{1,k}(g) \coloneqq e_{13}(\varphi(g))$ is a cochain in $\Ch^1$ such that 
\begin{align*}
\delta \eta^{\oddeven}_{1,k} = - (\chi_1 + \chi_{2k}) \cup (\chi_2 + \chi_{2k-1}). 
\end{align*}
We then define $f_2(\chi_{1} \ot \chi_{2} + \chi_{2k} \ot \chi_{2k-1})$ such that 
\begin{align*}
& ~ f_2((\chi_1 + \chi_{2k}) \ot (\chi_2 + \chi_{2k-1})) \\
= & ~ f_2(\chi_{1} \ot \chi_{2k-1}) + f_2(\chi_{2k} \ot \chi_{2}) + 
f_2(\chi_{1} \ot \chi_{2} + \chi_{2k} \ot \chi_{2k-1}).  
\end{align*}
That is, using the $\F_p$-vector space structure on $\Ch^1$, we set  
\begin{align*}
f_2(\chi_1 \ot \chi_2 +  \chi_{2k} \ot \chi_{2k-1}) \coloneqq \eta^{\oddeven}_{1,k} - \eta_{1,2k-1} - \eta_{2k,2}. 
\end{align*}
We now define the map $f_2 \colon R \to \Ch^1$ on all of $R$ by extending it $\F_p$-linearly from $\BBB$ to $R$. 
From Proposition \ref{prop:canonical_class_via_Koszul_complex} and Corollary \ref{cor:Demushkin_groups_are_Koszul} 
we then deduce: 

\begin{prop}\label{prop:canonical_class_of_Demushkin_group_construction} 
With the above notation, we define the map $\Psi_3 \colon K_3^3(\Hb) \to \Zh^2$ by  
\begin{align}\label{eq:def_of_Psi3_Demushkin}
\Psi_3(\chi_a, \chi_b, \chi_c) = - \chi_a \cup f_2(\chi_b \ot \chi_c) - f_2(\chi_a \ot \chi_b) \cup \chi_c. 
\end{align}
Taking the cohomology class of $\Psi_3$ defines an $\F_p$-linear map $\kappa_3 \colon K^3_3(\Hb) \to H^2$. 
Then $\kappa_3$ is a cocycle in the complex $(\uHom_{\F}(K_{\bbb}^{\bbb}(\Hb), \Hb[-1]), \dee)$ which computes the Hochschild cohomology of $\Hb$, 
and the class of $\kappa_3$ in $\HH^{3,-1}(\Hb)$ is the canonical class of $G$. \qed
\end{prop}


\subsection{Demushkin groups with invariant $q=3$ are not $A_3$-formal}\label{sec:Demushkin_q=3}

Our goal now is to compute the canonical class for all pro-$p$ Demushkin groups for odd primes $p$. 
We begin by showing that pro-$3$ Demushkin groups with invariant $q=3$ 
are not $A_3$-formal. 
Let $\chi_1, \ldots, \chi_d \in H^1$ be as above. 
We note that $\chi_1^{\ot 3}$ is an element in $K^3_3(\Hb)$ 
and refer to Lemma \ref{lemma:Demushkin_K_3} for a complete basis of $K^3_3(\Hb)$.  

\begin{lemma}\label{lemma:chito3_is_not_hit}
Let $\pam \colon R \to H^1$ be an $\F_p$-linear map which we consider as a cochain in the complex $(\uHom_{\F}(K_{\bbb}^{\bbb}(\Hb), \Hb[-1]), \dee)$. 
Then $\partial \pam (\chi_1^{\ot 3}) = 0$. 
\end{lemma}
\begin{proof}
Let $c_j \in \F_p$ be coefficients such that $\pam(\chi_1 \ot \chi_1) = \sum_{j=1}^d c_j \chi_j$ in $H^1$. 
Since $\chi_1 \cup \chi_j$ and $\chi_j \cup \chi_1$ are nonzero in $H^2$ if and only if $j=2$, we get 
\begin{align*}
\partial \pam (\chi_1^{\ot 3} ) & = - \chi_1 \cup \pam (\chi_1 \ot \chi_1) - \pam (\chi_1 \ot \chi_1) \cup \chi_1\\
& = - c_2 (\chi_1 \cup \chi_2) - c_2 (\chi_2 \cup \chi_1) \\
& = - c_2 (\chi_1 \cup \chi_2 + \chi_2 \cup \chi_1) \\
& = 0
\end{align*}
where we use the relations in $H^2$ for the final equality. 
This proves the assertion. 
\end{proof}


\begin{theorem}\label{thm:Demushkin_p_equal_3}
Let $d\ge 2$ be an even number and let $G$ be the pro-$3$ group generated by elements $x_1, \ldots, x_d$ with the single relation 
$x_1^3 [x_1,x_2]\cdots[x_{d-1},x_d]=1$. 
The canonical class of $G$ is non-trivial, and $G$ is not $A_3$-formal. 
\end{theorem}
\begin{proof}
%
Let $\kappa_3 \colon K^3_3(\Hb) \to H^2$ denote the map defined in Proposition \ref{prop:canonical_class_of_Demushkin_group_construction}. 
By Lemma \ref{lemma:chito3_is_not_hit}, 
if $\kappa_3(\chi_1^{\ot 3}) \ne 0$ in $H^2$, 
then the class of $\kappa_3$ in $\HH^{3,-1}(\Hb)$ is non-trivial. 
Thus, to prove the theorem it suffices to show $\kappa_3(\chi_1^{\ot 3}) \ne 0$.    
To show the latter we show that the cocycle $\Psi_3(\chi_1^{\ot 3})$ is not a coboundary in $\Cb$. 
%
Let $B_+$ denote the matrix defined in \eqref{eq:def_of_matrix_B_+}. 
Since $\binom{3}{2}=0$ and $\binom{3}{3}=1$ in $\F_3$, we get 
\begin{align*}
B_+^3 = \begin{pmatrix} 
1 & 0 & 0 & 1 \\
0 & 1 & 0 & 0 \\
0 & 0 & 1 & 0 \\
0 & 0 & 0 & 1 
\end{pmatrix} 
~ \text{in} ~ U_4(\F_3).
\end{align*}
%
Thus, $B_+^3 = I_4$ in the quotient $\bU_4(\F_3) = U_4(\F_3)/Z_4(\F_3)$. 
Since $[B_+,I_4]=I_4$ in $U_4(\F_3)$ and thereby also $[B_+,I_4]=I_4$ in $\bU_4(\F_3)$,  
the assignment $\brr_1(x_1) = B_+$ and  $\brr_1(x_j) = I_4$ for $j=2,\ldots,d$ 
defines a continuous group homomorphism $\brr_1 \colon G \to \bU_4(\F_3)$. 
However, since $B_+^3 \ne I_4$ in $U_4(\F_3)$, $\brr_1$ does not extend to a continuous group homomorphism $G \to U_4(\F_3)$. 
By Corollary \ref{cor:Dwyer_consequence}, 
this implies that the cocycle $\Psi_3(\chi_1^{\ot 3})$ is not a coboundary.  
This 
proves the assertion. 
\end{proof}


\begin{remark}
We note that in the proof of Theorem \ref{thm:Demushkin_p_equal_3} we cannot replace $\chi_1$ by any other $\chi_i$. 
That is, if $i \ne 1$, we can define a continuous group homomorphism $\rr \colon G \to U_4(\F_3)$ by setting $\rr(x_i) = B_+$ and $\rr(x_j)=I_4$ 
for $j \ne i$, 
which yields the vanishing of $\kappa_3(\chi_i^{\ot 3})$ for $i=2,\ldots,d$. 
%
The difference is that the non-triviality of $B_+^3$ only matters for $x_1$ 
which occurs outside the commutators in the relation $x_1^3 [x_1,x_2]\cdots[x_{d-1},x_d]=1$. 
\end{remark}


\begin{remark}\label{rem:p3_Demushkin_triple_Massey_still_vanishes} 
%
We now provide an alternative proof of Theorem \ref{thm:Demushkin_p_equal_3}. 
Let $G$ be the pro-$3$ group as in Theorem \ref{thm:Demushkin_p_equal_3}. 
It is well-known that triple Massey products vanish for Demushkin groups 
(see for example \cite[Theorem 4.3]{MT2} and \cite[Theorem 3.5]{PSz}).  
We therefore point out that the argument in the proof of Theorem \ref{thm:Demushkin_p_equal_3} 
does not imply that the triple Massey product 
$\langle \chi_1, \chi_1, \chi_1 \rangle$ does not vanish. 
In fact, the argument shows that for the particular defining system 
$(\chi_1,\chi_1,\chi_1, - f_2(\chi_1 \ot \chi_1), - f_2(\chi_1 \ot \chi_1) )$,  
the cocycle 
\[
\alpha \coloneqq - \chi_1 \cup f_2(\chi_1 \ot \chi_1) - f_2(\chi_1 \ot \chi_1) \cup \chi_1 
\] 
is not a coboundary 
and therefore provides a non-trivial element in the set $\langle \chi_1, \chi_1, \chi_1 \rangle$. 
However, we can modify our choice of defining system as follows. 
Let $C$ be the matrix given by 
\begin{align*}
C = \begin{pmatrix} 
1 & 0 & 1 & 0 \\
0 & 1 & 0 & 0 \\
0 & 0 & 1 & 0 \\
0 & 0 & 0 & 1 
\end{pmatrix}~ \text{in} ~ U_4(\F_3).
\end{align*}
%
Since $B^3_+ = \begin{pmatrix} 
1 & 0 & 0 & 1 \\
0 & 1 & 0 & 0 \\
0 & 0 & 1 & 0 \\
0 & 0 & 0 & 1 
\end{pmatrix}$
and 
$[B_+,C] = \begin{pmatrix} 
1 & 0 & 0 & -1 \\
0 & 1 & 0 & 0 \\
0 & 0 & 1 & 0 \\
0 & 0 & 0 & 1 
\end{pmatrix}$, 
we get 
\begin{align*}
B_+^3 \cdot [B_+,C] = I_4 ~ \text{in} ~ U_4(\F_3).
\end{align*}
Thus, we can define a continuous homomorphism $\wrr \colon G \to U_4(\F_3)$ 
by setting 
\begin{align*}
\wrr(x_1) = B_+, ~ \wrr(x_2) = C, ~ \text{and} ~ \wrr(x_j) = I_4 ~ \text{for} ~ j =3,\ldots,d. 
\end{align*}
Then $\wrr$ corresponds to the defining system 
\begin{align*}
(\chi_1,\chi_1,\chi_1, - f_2(\chi_1 \ot \chi_1) + \chi_2, - f_2(\chi_1 \ot \chi_1) )
\end{align*}
and yields that the corresponding cocycle is a coboundary. 
In fact, the continuous map $\wtheta \colon G \to \F_3$ given by $\wtheta(g) \coloneqq - e_{14}(\wrr(g))$ 
provides a cochain such that 
\[
\delta \wtheta = \alpha + \chi_2 \cup \chi_1.
\] 
In particular, we get that the class of $\alpha$ in $H^2$ equals $\chi_1 \cup \chi_2$, 
i.e., $\kappa_3(\chi_1^{\ot 3}) = \chi_1 \cup \chi_2$. 
\end{remark}

\begin{remark}\label{rem:realizable_Demushkin_group}
By \cite[page 254]{Koenigsmann}, the group $G$ of Theorem \ref{thm:Demushkin_p_equal_3} for $d=4$ 
is realisable as the maximal pro-$3$ Galois group $G_F(3)$ of the field $F=\Q_3(\zeta_3)$ 
where $\zeta_3$ is a root of unity of order $3$. 
According to \cite{EfratDem} and \cite{Koenigsmann} (see also \cite[Remark 3.3]{Quadrelli}), 
this is the only Demushkin group of rank $4$ which is known to be realisable as the maximal pro-$p$ Galois group $G_F(p)$ of a field. 
\end{remark}


\begin{remark}\label{rem:semidirect_product_at_p=3_is_not_A3-formal}
As pointed out in Remark \ref{rem:Positselski_example_as_semidirect_product}, 
the pro-$3$-group with generators $x_1$ and $x_2$ and relation $x_1^3[x_1,x_2]=1$ is isomorphic to 
the semi-direct product $\Z_3 \rtimes_{\car} \Z_3$ 
where $\car \colon \Z_3 \to 1 + 3\Z_3$ is the cyclotomic character. 
Theorem \ref{thm:Demushkin_p_equal_3} 
shows that $\Z_3 \rtimes_{\car} \Z_3$ is not $A_3$-formal even though $\Z_3$ is intrinsically $A_{\infty}$-formal. 
\end{remark}


\section{Proof of Theorem \ref{thm:Demushkin_groups_are_formal}}\label{sec:proof_of_Dem_thm}

In this section we prove Theorem \ref{thm:Demushkin_groups_are_formal}. 
For $d=2$, the assertion is proven in Theorem \ref{thm:Demushkin_two_generators}. 
We therefore assume from now on $d\ge 4$. 
%
%
We will first determine a basis of $K^3_3(\Hb)$. 
We then compute the values of the map $\kappa_3$ induced by the map $\Psi_3$ of Proposition \ref{prop:canonical_class_of_Demushkin_group_construction}. 
In fact, we will show that $\kappa_3$ vanishes on most basis elements, 
while $\kappa_3$ is non-trivial on a certain subset of the basis. 
However, we then show that $\kappa_3$ is a Hochschild-coboundary. 
Using our previous results, we can then deduce Theorem \ref{thm:Demushkin_groups_are_formal}. 

\subsection{A basis for $K^3_3(\Hb)$}\label{sec:basis_for_K3}

We continue to use the notation of Section \ref{sec:A3_formality_for_Demushkin_groups}, 
in particular, the notation introduced in Section \ref{sec:Demushkin_canonical_class} leading to Proposition \ref{prop:canonical_class_of_Demushkin_group_construction}. 
In addition, we will use the following abbreviated notation. 

\begin{notn}
We will often write $\chi_{i,j,k}$ for $\chi_i \ot \chi_j \ot \chi_k \in H^1 \ot H^1 \ot H^1$ 
when it makes formulas easier to read or easier to fit into the text. 
\end{notn}

We can now determine a basis of $K_3^3(\Hb)$. 


\begin{lemma}\label{lemma:Demushkin_K_3}
The vector space $K^3_3(\Hb) = (H^1\ot R) \cap (R\ot H^1)$ is given by 
\begin{align*}
K^3_3(\Hb) = \mathrm{span}_{\F_p}(\SSS \cup \DDD \cup \DDDw \cup \TTT)
\end{align*}
where $\SSS \cup \DDD \cup \DDDw \cup \TTT$ is a basis, 
and the sets $\SSS$ (single terms), $\DDD$, $\DDDw$ (double sums), and $\TTT$ (triple sums) are 
\begin{align*}
\SSS & = \left\{ \chi_i \ot \chi_j \ot \chi_k 
~ \text{with}~ \begin{cases}
i, k \ne j+1 &  \text{if} ~ j ~ \text{is odd} \\
i, k \ne j-1 &  \text{if} ~ j ~ \text{is even}; 
\end{cases} 
\right\}, 
\end{align*}
\begin{align*}
\DDD = & ~ \{ \chi_k \ot (\chi_{2i-1} \ot \chi_{2i} + \chi_{2i} \ot \chi_{2i-1}), \\
& ~ ~(\chi_{2i-1} \ot \chi_{2i} + \chi_{2i} \ot \chi_{2i-1}) \ot \chi_k ~ \text{for} ~ 1\le i \le d/2, ~ k \ne 2i-1, 2i, \\
& ~ ~ \chi_k \ot (\chi_{1} \ot \chi_{2} + \chi_{2i} \ot \chi_{2i-1}), \\
& ~ ~(\chi_{1} \ot \chi_{2} + \chi_{2i} \ot \chi_{2i-1}) \ot \chi_k ~ \text{for} ~ 2 \le i \le d/2, ~ k \ne 1, 2, 2i-1, 2i
\} 
\end{align*}
\begin{align*}
\DDDw = & ~ \{ \chi_{1} \ot (\chi_{1} \ot \chi_{2} + \chi_{2i} \ot \chi_{2i-1}), 
\chi_{2i} \ot (\chi_{1} \ot \chi_{2} + \chi_{2i} \ot \chi_{2i-1}), \\
& ~ \chi_{2} \ot (\chi_{2} \ot \chi_{1} + \chi_{2i-1} \ot \chi_{2i}), 
\chi_{2i-1} \ot (\chi_{2} \ot \chi_{1} + \chi_{2i-1} \ot \chi_{2i}), \\
& ~(\chi_{2} \ot \chi_{1} + \chi_{2i-1} \ot \chi_{2i}) \ot \chi_1, 
(\chi_{2} \ot \chi_{1} + \chi_{2i-1} \ot \chi_{2i}) \ot \chi_{2i},  \\ 
& ~ (\chi_{1} \ot \chi_{2} + \chi_{2i} \ot \chi_{2i-1}) \ot \chi_2,  
(\chi_{1} \ot \chi_{2} + \chi_{2i} \ot \chi_{2i-1}) \ot \chi_{2i-1},  ~ \text{for} ~ 2 \le i \le d/2
\}.
\end{align*}
\begin{align*}
\TTT = & ~ \{ \chi_{2i} \ot \chi_{2i-1} \ot \chi_{2i-1} + \chi_{2i-1} \ot \chi_{2i} \ot \chi_{2i-1} + \chi_{2i-1} \ot \chi_{2i-1} \ot \chi_{2i}, \\
& ~ ~\chi_{2i-1} \ot \chi_{2i} \ot \chi_{2i} + \chi_{2i} \ot \chi_{2i-1} \ot \chi_{2i} + \chi_{2i} \ot \chi_{2i} \ot \chi_{2i-1} ~ \text{for} ~ 1\le i \le d/2 
\}, 
\end{align*}
\end{lemma}
\begin{proof}
By Lemma \ref{lemma:dim_of_Knn}, 
we know that $\dim_{\F_p} K^3_3(\Hb) = d^3-2d$. 
We have 
\begin{align*}
\# \SSS & = (d-1)d(d-1)= d(d-1)^2, \\
\# \DDD & = 2(d/2)(d-2) + 2(d/2-1)(d-4) = 2(d-2)^2, \\
~ \# \DDDw & = 8(d/2-1) = 4d - 8, 
~ \text{and} ~ \# \TTT  = 2 d/2= d. 
\end{align*}
Thus, 
\begin{align*}
\# \SSS + \# \DDD + \#\DDDw + \#\TTT 
= & ~ d(d-1)^2 + 2(d-2)^2 + 4d - 8 + d \\
= & ~ d^3 -2d.
\end{align*} 
Hence, to prove the assertion, it suffices to check that the union 
of sets $\SSS \cup \DDD \cup \DDDw \cup \TTT$ is linearly independent.  
To show the latter claim, we note that the vectors $\chi_i \ot \chi_j \ot \chi_k$ for $i,j,k \in \{1,\ldots, d\}$ 
form a basis of $H^1\ot H^1 \ot H^1$. 
Since elements in $\SSS$ are single vectors of the form$\chi_i \ot \chi_j \ot \chi_k$ 
and since $\vspan (\SSS) \cap \vspan (\DDD \cup \DDDw \cup \TTT) = \{0\}$, 
it suffices show that the set $\DDD \cup \DDDw \cup \TTT$ is linearly independent. 

Now we assume that the zero vector in $H^1\ot H^1 \ot H^1$ is written as a linear combination of elements of $\DDD \cup \DDDw \cup \TTT$. 
Vectors of the form $\chi_{2i-1} \ot \chi_{2i} \ot \chi_{2i-1}$ 
only occur in sums in $\TTT$. 
Thus, the coefficient of the term containing $\chi_{2i-1} \ot \chi_{2i} \ot \chi_{2i-1}$ is zero. 
For $i=1$, this implies that the coefficients of 
$\chi_{1} \ot (\chi_{1} \ot \chi_{2} + \chi_{2j} \ot \chi_{2j-1})$ and 
$(\chi_{2} \ot \chi_{1} + \chi_{2j-1} \ot \chi_{2j}) \ot \chi_1$ must be zero.  
For $i\ge 2$, this implies that the coefficients of $\chi_{2i-1} \ot (\chi_{2} \ot \chi_{1} + \chi_{2i-1} \ot \chi_{2i})$ and 
$(\chi_{1} \ot \chi_{2} + \chi_{2i} \ot \chi_{2i-1}) \ot \chi_{2i-1}$ 
for $2 \le j \le d/2$ in $\DDDw$ must be zero as well.  
Similarly, vectors of the form $\chi_{2i} \ot \chi_{2i-1} \ot \chi_{2i}$ 
only occur in sums in $\TTT$. 
Thus, the coefficient of the term containing $\chi_{2i} \ot \chi_{2i-1} \ot \chi_{2i}$ is zero. 
For $i=1$, this implies that the coefficients of 
$\chi_{2} \ot (\chi_{2} \ot \chi_{1} + \chi_{2j-1} \ot \chi_{2j})$ and 
$(\chi_{1} \ot \chi_{2} + \chi_{2j} \ot \chi_{2j-1}) \ot \chi_2$ 
for $2 \le j \le d/2$ in $\DDDw$ must be zero.  
For $i\ge 2$, this implies that the coefficients of 
$\chi_{2i} \ot (\chi_{1} \ot \chi_{2} + \chi_{2i} \ot \chi_{2i-1})$ and 
$(\chi_{2} \ot \chi_{1} + \chi_{2i-1} \ot \chi_{2i}) \ot \chi_{2i}$ 
in $\DDDw$ must be zero as well.  
This shows that the coefficients of all vectors in $\DDDw$ and $\TTT$ are zero. 
It remains to consider the coefficients of vectors in $\DDD$. 
For $k=1,2$ and $i \ge 2$, the vectors 
$\chi_k \ot (\chi_{2i-1} \ot \chi_{2i} + \chi_{2i} \ot \chi_{2i-1})$ 
and $(\chi_{2i-1} \ot \chi_{2i} + \chi_{2i} \ot \chi_{2i-1}) \ot \chi_k$ 
share a summand each with exactly one vector in $\DDDw$. 
Similarly, for $k=2j-1,2j$ and $j\ne 1, i$, the vectors 
$\chi_k \ot (\chi_{1} \ot \chi_{2} + \chi_{2i} \ot \chi_{2i-1})$ 
and $(\chi_{1} \ot \chi_{2} + \chi_{2i} \ot \chi_{2i-1}) \ot \chi_k$ 
share a summand each with exactly one vector in $\DDDw$. 
However, since the coefficients of vectors in $\DDDw$ are zero, 
the coefficients of the above vectors in $\DDD$ must be zero as well. 
The remaining vectors in $\DDD$ consist of sums for which the summands only occur in $\DDD$, 
and each summand occurs in exactly one vector in $\DDD$.  
This implies that the coefficients of all vectors in $\DDD$ vanish.  
This proves that all coefficients must be zero 
and hence the claim. 
\end{proof}

\begin{example}
For $d=2$, the sets $\DDD$ and $\DDDw$ are empty, 
and we recover the basis used in the proof of Theorem \ref{thm:Demushkin_two_generators}. 
For $d=4$, we have 
\begin{align*}
\DDD =  \{ & \chi_{1,2,3}+ \chi_{2,1,3}, 
\chi_{1,2,4} + \chi_{2,1,4}, 
\chi_{3,1,2} + \chi_{3,2,1}, 
\chi_{4,1,2} + \chi_{4,2,1}, \\
&  \chi_{3,4,1} + \chi_{4,3,1}, 
  \chi_{3,4,2} + \chi_{4,3,2}, 
  \chi_{1,3,4} + \chi_{1,4,3}, 
  \chi_{2,3,4} + \chi_{2,4,3}\}, 
 \end{align*}
\begin{align*}
\DDDw = ~ \{ & \chi_{1,1,2} + \chi_{1,4,3}, 
 \chi_{4,1,2}  + \chi_{4,4,3}, 
 \chi_{1,2,2} + \chi_{4,3,2}, 
 \chi_{1,2,3} + \chi_{4,3,3}, \\
 &  \chi_{2,2,1} + \chi_{2,3,4},  
 \chi_{3,2,1} + \chi_{3,3,4}, 
 \chi_{2,1,1} + \chi_{3,4,1}, 
 \chi_{2,1,4} + \chi_{3,4,4} \}, 
\end{align*}  
\begin{align*}
\TTT =  \{ & \chi_{1,1,2} + \chi_{1,2,1} + \chi_{2,1,1}, 
\chi_{2,2,1} + \chi_{2,1,2} + \chi_{1,2,2}, \\
&  \chi_{3,3,4} + \chi_{3,4,3} + \chi_{4,3,3}, 
\chi_{4,4,3} + \chi_{4,3,4} + \chi_{3,4,4}\}.  
\end{align*}
\end{example}


\subsection{Computation of $\kappa_3$}\label{sec:computing_kappa3}

%

We will now compute the values of $\kappa_3$. 
To do so, we will make frequent use of the following simple observation: 

\begin{lemma}\label{lemma:B^q=I_4}
For all choices of $\varepsilon_i \in \{-1,0,1\}$, 
the $q$th power of the matrix $B_{\varepsilon_1, \varepsilon_2, \varepsilon_3}$ given by 
\begin{align*}
B_{\varepsilon_1, \varepsilon_2, \varepsilon_3} = \begin{pmatrix} 
1 & \varepsilon_1 & 0 & 0 \\
0 & 1 & \varepsilon_2 & 0 \\
0 &  0 & 1 & \varepsilon_3\\
0 & 0 & 0 & 1 
\end{pmatrix}
\end{align*}
is the identity matrix in $U_4(\F_p)$, i.e., 
\begin{align*}
B^q_{\varepsilon_1, \varepsilon_2, \varepsilon_3} = I_4 ~ \text{in} ~ U_4(\F_p) 
\end{align*}
\end{lemma}
\begin{proof}
This follows from a direct computation for each case where we use that $q=p^f$ with $p$ odd and $f \ge 2$ when $p=3$. 
\end{proof}



\begin{lemma}\label{lemma:kappa_3_vanishes_on_many_subsets}
Let $\kappa_3 \colon K_3^3(\Hb) \to H^2$ be the map defined in Proposition \ref{prop:canonical_class_of_Demushkin_group_construction}. 
The map $\kappa_3$ vanishes on all elements in the sets $\SSS$, $\DDD$ and $\TTT$.  
Moreover, $\kappa_3$ vanishes on the elements in the subset $\DDDwr \subset \DDDw$ ($r$ for right-multiplication) defined by 
\begin{align*}
\DDDwr = & ~ \{ (\chi_{2} \ot \chi_{1} + \chi_{2i-1} \ot \chi_{2i}) \ot \chi_1, 
(\chi_{2} \ot \chi_{1} + \chi_{2i-1} \ot \chi_{2i}) \ot \chi_{2i},  \\ 
& ~ (\chi_{1} \ot \chi_{2} + \chi_{2i} \ot \chi_{2i-1}) \ot \chi_2,  
(\chi_{1} \ot \chi_{2} + \chi_{2i} \ot \chi_{2i-1}) \ot \chi_{2i-1} ~ \text{for} ~ 2 \le i \le d/2
\}. 
\end{align*}
\end{lemma}
\begin{proof}
Let $\Psi_3 \colon K_3^3(\Hb) \to \Zh^2$ be the cocycle 
\begin{align*}
\Psi_3(\chi_a, \chi_b, \chi_c) = - \chi_a \cup f_2(\chi_b \ot \chi_c) - f_2(\chi_a \ot \chi_b) \cup \chi_c 
\end{align*}
defined in \eqref{eq:def_of_Psi3_Demushkin} 
where, by slight abuse of notation, we allow $(\chi_a, \chi_b, \chi_c)$ to denote a sum of tensors in $K_3^3(\Hb)$. 
We prove the assertion by constructing cochains $\vartheta \in \Ch^1$ such that $\delta \vartheta = \Psi_3$ for the elements in each of the given sets. 
To do so, we consider each set separately and, in addition, group elements within the sets into different classes.  
We construct $\vartheta$ for each class by constructing a suitable continuous group homomorphism $\rr \colon G \to U_4(\F_p)$ 
which extends the continuous group homomorphism $\brr \colon G \to \bU_4(\F_p)$ 
which corresponds to $\Psi_3$ by Dwyer's Theorem \ref{thm:Dwyer_alt1}. 
By Corollary \ref{cor:Dwyer_consequence} and Remark \ref{rem:A3_formal_Koszul_group_coh}, 
this implies that $\Psi_3$ on the given basis element vanishes. 
To construct $\rr$, 
we use Lemma \ref{lemma:matrix_relations} often without explicitly stating that the corresponding matrices 
satisfy the required relation whenever it is trivial to check it. 

We begin with elements in $\SSS$. 
\begin{itemize}
\item Elements of the form $\chi_i \ot \chi_i \ot \chi_i$: 
We define a continuous group homomorphism $\rr^{\SSS}_i \colon G \to U_4(\F_p)$ 
by setting $\rr^{\SSS}_i(x_i) = B_{1,1,1}$ and $\rr^{\SSS}_i(x_j) = I_4$ for $j \ne i$.  
We define the continuous map $\vartheta^{\SSS}_{i} \colon G \to \F_p$ by $g \mapsto \vartheta^{\SSS}_{i}(g) \coloneqq - e_{14}(\rr^{\SSS}_{i}(g))$.  
By construction of $\Psi_3$ and the map $f_2$, we have  
\begin{align*}
\delta \vartheta^{\SSS}_{i} = \Psi_3 (\chi_i \ot \chi_i \ot \chi_i). 
\end{align*} 



\item Elements of the form $\chi_i \ot \chi_i \ot \chi_j$ with $i\ne j$ and $i \ne j-1$ if $j$ is even: 
We define a continuous group homomorphism $\rr^{\SSS}_{iij} \colon G \to U_4(\F_p)$ 
by setting 
\begin{align*}
\rr^{\SSS}_{iij}(x_i) = B_{1,1,0}, ~ \rr^{\SSS}_{iij}(x_j) = B_{0,0,1}, 
\end{align*}
and $\rr^{\SSS}_{iij}(x_k) = I_4$ for $k \ne i,j$.  
We define the continuous map $\vartheta^{\SSS}_{iij} \colon G \to \F_p$ by $g \mapsto \vartheta^{\SSS}_{iij}(g) \coloneqq - e_{14}(\rr^{\SSS}_{iij}(g))$. 
We then have  
\begin{align*}
\delta \vartheta^{\SSS}_{iij} = \Psi_3 (\chi_i \ot \chi_i \ot \chi_j).
\end{align*} 



\item Elements of the form $\chi_i \ot \chi_j \ot \chi_j$ with $i\ne j$ and $i \ne j-1$ if $j$ is even: 
We define a continuous group homomorphism $\rr^{\SSS}_{ijj} \colon G \to U_4(\F_p)$ 
by setting 
\begin{align*}
\rr^{\SSS}_{ijj}(x_i) = B_{1,0,0}, ~ \rr^{\SSS}_{ijj}(x_j) = B_{0,1,1}, 
\end{align*}
and $\rr^{\SSS}_{ijj}(x_k) = I_4$ for $k \ne i,j$.  
We define a continuous map $\vartheta^{\SSS}_{ijj} \colon G \to \F_p$ by $g \mapsto \vartheta^{\SSS}_{ijj}(g) \coloneqq - e_{14}(\rr^{\SSS}_{ijj}(g))$. 
We then have  
\begin{align*}
\delta \vartheta^{\SSS}_{ijj} = \Psi_3 (\chi_i \ot \chi_j \ot \chi_j).
\end{align*} 


\item Elements of the form $\chi_i \ot \chi_j \ot \chi_k$ with $i,k\ne j$, $i \ne j-1$ if $j$ is even and $k \ne j+1$ if $j$ is odd: 
We define a continuous group homomorphism $\rr^{\SSS}_{ijk} \colon G \to U_4(\F_p)$ 
by setting 
\begin{align*}
\rr^{\SSS}_{ijk}(x_i) = B_{1,0,0}, ~ 
\rr^{\SSS}_{ijk}(x_j) = B_{0,1,0}, ~ 
\rr^{\SSS}_{ijk}(x_k) = B_{0,0,1}, 
\end{align*}
and $\rr^{\SSS}_{ijk}(x_l) = I_4$ for $l \ne i,j,k$.  
We define the continuous map $\vartheta^{\SSS}_{ijk} \colon G \to \F_p$ by $g \mapsto \vartheta^{\SSS}_{ijk}(g) \coloneqq - e_{14}(\rr^{\SSS}_{ijk}(g))$, 
and we have 
\begin{align*}
\delta \vartheta^{\SSS}_{ijk} = \Psi_3 (\chi_i \ot \chi_j \ot \chi_k).
\end{align*} 

\end{itemize}


This finishes the proof for the set $\SSS$. 
Next we consider elements in $\DDD$. 
First, we consider $1\le i \le d/2$ and $k \ne 2i-1, 2i$. 
\begin{itemize}

\item Elements of the form $\chi_k \ot (\chi_{2i-1} \ot \chi_{2i} + \chi_{2i} \ot \chi_{2i-1})$: 
We have 
\begin{align*}
& ~ \chi_k \ot (\chi_{2i-1} + \chi_{2i})^{\ot 2} \\
= & ~ \chi_k \ot (\chi_{2i-1} \ot \chi_{2i} + \chi_{2i} \ot \chi_{2i-1}) + \chi_{k,2i-1,2i-1} + \chi_{k,2i,2i}.
\end{align*}
Since we already know that $\kappa_3$ vanishes on $\chi_{k,2i-1,2i-1}$ and $\chi_{k,2i,2i}$, 
and $\kappa_3$ is a linear map, 
it suffices to determine a cochain whose boundary is $\Psi_3 (\chi_k \ot (\chi_{2i-1} + \chi_{2i})^{\ot 2})$. 
By Lemma \ref{lemma:matrix_relations}, 
we can define a continuous group homomorphism $\rr^{\DDD}_{i,k,\li} \colon G \to U_4(\F_p)$ 
by setting 
\[
\rr^{\DDD}_{i,k,\li}(x_{k}) = B_{1,0,0}, \rr^{\DDD}_{i,k,\li}(x_{2i-1}) = \rr^{\DDD}_{i,k,\li}(x_{2i}) = B_{0,1,1}, 
\] 
and $\rr^{\DDD}_{i,k,\li}(x_j) = I_4$ for $j \ne 2i-1, 2i, k$. 
We define the continuous map $\vartheta^{\DDD}_{i,k,\li} \colon G \to \F_p$ by $g \mapsto - e_{14}(\rr^{\DDD}_{i,k,\li}(g))$ 
and we get 
\begin{align*}
\delta \vartheta^{\DDD}_{i,k,\li} = \Psi_3 (\chi_k \ot (\chi_{2i-1} + \chi_{2i})^{\ot 2}). 
\end{align*}


\item Elements of the form $(\chi_{2i-1} \ot \chi_{2i} + \chi_{2i} \ot \chi_{2i-1}) \ot \chi_k$: 
We have 
\begin{align*}
& ~ (\chi_{2i-1} + \chi_{2i})^{\ot 2} \ot \chi_k \\
= & ~ (\chi_{2i-1} \ot \chi_{2i} + \chi_{2i} \ot \chi_{2i-1}) \ot \chi_k + \chi_{2i-1,2i-1,k} + \chi_{2i,2i,k}.
\end{align*}
We already know that $\kappa_3$ vanishes on $\chi_{2i-1,2i-1,k}$ and $\chi_{2i,2i,k}$, 
it suffices to determine a cochain whose boundary is $\Psi_3 ((\chi_{2i-1} + \chi_{2i})^{\ot 2} \ot \chi_k)$. 
Since 
\begin{align*}
[B_{1,1,0},B_{1,1,0}] = I_4 ~ \text{and} ~ [B_{0,0,1},I_4] = I_4 ~ \text{in} ~ U_4(\F_p),
\end{align*} 
we can define a continuous group homomorphism $\rr^{\DDD}_{i,k,\re} \colon G \to U_4(\F_p)$ 
by setting 
\[
\rr^{\DDD}_{i,k,\re}(x_{k}) = B_{0,0,1}, \rr^{\DDD}_{i,k,\re}(x_{2i-1}) = \rr^{\DDD}_{i,k,\re}(x_{2i}) = B_{1,1,0}, 
\] 
and $\rr^{\DDD}_{i,k,\re}(x_j) = I_4$ for $j \ne 2i-1, 2i, k$. 
We define the continuous map $\vartheta^{\DDD}_{i,k,\re} \colon G \to \F_p$ defined by $g \mapsto - e_{14}(\rr^{\DDD}_{i,k,\re}(g))$ 
and we set 
\begin{align*}
\delta \vartheta^{\DDD}_{i,k,\re} = \Psi_3 ((\chi_{2i-1} + \chi_{2i})^{\ot 2} \ot \chi_k).   
\end{align*}
\end{itemize}

%
%

Second, we consider $2\le i \le d/2$ and $k \ne 1,2,2i-1,2i$.  
\begin{itemize}

\item Elements of the form $\chi_k \ot (\chi_{1} \ot \chi_{2} + \chi_{2i} \ot \chi_{2i-1})$: 
We have 
\begin{align*}
& \chi_k \ot (\chi_{1} \ot \chi_{2i}) \ot  (\chi_{2} \ot \chi_{2i-1}) \\
= & ~ \chi_k \ot (\chi_{2i-1} \ot \chi_{2i} + \chi_{2i} \ot \chi_{2i-1}) + \chi_{k,1,2i-1} + \chi_{k,2i,2}.
\end{align*}
We know that $\kappa_3$ vanishes on $\chi_{k,1,2i-1}$ and $\chi_{k,2i,2}$.  
Hence, it suffices to determine a cochain whose boundary is $\Psi_3 (\chi_k \ot (\chi_{1} \ot \chi_{2i}) \ot  (\chi_{2} \ot \chi_{2i-1}))$. 
Since 
\begin{align*}
[B_{0,1,0},B_{0,0,1}]\cdot [B_{0,0,1},B_{0,1,0}] \cdot [B_{1,0,0},I_4] = I_4 ~ \text{in} ~ U_4(\F_p),
\end{align*} 
we can define a continuous group homomorphism $\rr^{\DDD}_{1,i,k,\li} \colon G \to U_4(\F_p)$ 
by setting 
\begin{align*}
\rr^{\DDD}_{1,i,k,\li}(x_{k}) = B_{1,0,0}, \rr^{\DDD}_{1,i,k,\li}(x_{1}) = \rr^{\DDD}_{1,i,k,\li}(x_{2i}) = B_{0,1,1}, \\
\rr^{\DDD}_{1,i,k,\li}(x_{2}) = \rr^{\DDD}_{1,i,k,\li}(x_{2i-1}) = B_{0,0,1}
\end{align*}
and $\rr^{\DDD}_{1,i,k,\li}(x_j) = I_4$ for $j \ne 2i-1, 2i, k$. 
We let $\vartheta^{\DDD}_{1,i,k,\li} \colon G \to \F_p$ be the continuous map defined by $g \mapsto - e_{14}(\rr^{\DDD}_{1,i,k,\li}(g))$ 
and we get 
\begin{align*}
\delta \vartheta^{\DDD}_{1,i,k,\li} = \Psi_3 (\chi_k \ot (\chi_{2i-1} + \chi_{2i})^{\ot 2}). 
\end{align*}

\item Elements of the form $(\chi_{1} \ot \chi_{2} + \chi_{2i} \ot \chi_{2i-1}) \ot \chi_k$: 
We have 
\begin{align*}
& (\chi_{1} \ot \chi_{2i}) \ot  (\chi_{2} \ot \chi_{2i-1}) \ot \chi_k \\
= & ~ (\chi_{2i-1} \ot \chi_{2i} + \chi_{2i} \ot \chi_{2i-1}) \ot \chi_k + \chi_{1,2i-1,k} + \chi_{2i,2,k}.
\end{align*}
We have defined cochains whose boundaries are $\Psi_3(\chi_{1,2i-1,k})$ and $\Psi_3(\chi_{2i,2,k})$, respectively. 
Hence, it suffices to determine a cochain whose boundary is $\Psi_3 ((\chi_{1} \ot \chi_{2i}) \ot  (\chi_{2} \ot \chi_{2i-1}) \ot \chi_k)$. 
Since 
\begin{align*}
[B_{1,0,0},B_{0,1,0}]\cdot [B_{0,1,0},B_{1,0,0}] \cdot [B_{0,0,1},I_4] = I_4 ~ \text{in} ~ U_4(\F_p),
\end{align*} 
we can define a continuous group homomorphism $\rr^{\DDD}_{1,i,k,\re} \colon G \to U_4(\F_p)$ 
by setting 
\begin{align*}
\rr^{\DDD}_{1,i,k,\re}(x_{k}) = B_{001}, \rr^{\DDD}_{1,i,k,\re}(x_{1}) = \rr^{\DDD}_{1,i,k,\re}(x_{2i}) = B_{1,1,0}, \\
\rr^{\DDD}_{1,i,k,\re}(x_{2}) = \rr^{\DDD}_{1,i,k,\re}(x_{2i-1}) = B_{0,1,0}
\end{align*}
and $\rr^{\DDD}_{1,i,k,\re}(x_j) = I_4$ for $j \ne 1, 2, 2i-1, 2i, k$. 
We let $\vartheta^{\DDD}_{1,i,k,\re} \colon G \to \F_p$ be the continuous map defined by $g \mapsto - e_{14}(\rr^{\DDD}_{1,i,k,\re}(g))$ 
and we get 
\begin{align*}
\delta \vartheta^{\DDD}_{1,i,k,\re} = \Psi_3 ((\chi_{1} \ot \chi_{2i}) \ot  (\chi_{2} \ot \chi_{2i-1}) \ot \chi_k). 
\end{align*}

\end{itemize}


This proves the assertion for the set $\DDD$. 
Next, we consider elements in $\TTT$:  
%
We have  
\begin{align*}
(\chi_{2i-1} + \chi_{2i})^{\ot 3} = & ~ \chi_{2i-1}^{\ot 3} + \chi_{2i}^{\ot 3} \\
& + (\chi_{2i} \ot \chi_{2i-1} \ot \chi_{2i-1} + \chi_{2i-1} \ot \chi_{2i} \ot \chi_{2i-1} 
+ \chi_{2i-1} \ot \chi_{2i-1} \ot \chi_{2i}) \\
& + (\chi_{2i-1} \ot \chi_{2i} \ot \chi_{2i} + \chi_{2i} \ot \chi_{2i-1} \ot \chi_{2i} + \chi_{2i} \ot \chi_{2i} \ot \chi_{2i-1}),
\end{align*}
and
\begin{align*}
(\chi_{2i-1} - \chi_{2i})^{\ot 3} = & ~ \chi_{2i-1}^{\ot 3} - \chi_{2i}^{\ot 3} \\
& -  (\chi_{2i} \ot \chi_{2i-1} \ot \chi_{2i-1} + \chi_{2i-1} \ot \chi_{2i} \ot \chi_{2i-1} 
+ \chi_{2i-1} \ot \chi_{2i-1} \ot \chi_{2i}) \\
& + (\chi_{2i-1} \ot \chi_{2i} \ot \chi_{2i} + \chi_{2i} \ot \chi_{2i-1} \ot \chi_{2i} + \chi_{2i} \ot \chi_{2i} \ot \chi_{2i-1}). 
\end{align*}
Since we know that $\kappa_3$ vanishes on $\chi_{2i-1}^{\ot 3}$ and $\chi_{2i}^{\ot 3}$ 
and since $p$ is odd, 
it thus suffices to determine cochains whose boundaries are $\Psi_3((\chi_{2i-1} + \chi_{2i})^{\ot 3})$ and $\Psi_3((\chi_{2i-1} - \chi_{2i})^{\ot 3})$, 
respectively.  

\begin{itemize}
\item First, we define a continuous group homomorphism $\rr^{\TTT}_{i,+} \colon G \to U_4(\F_p)$ 
by setting 
\[
\rr^{\TTT}_{i,+}(x_{2i-1}) = \rr^{\TTT}_{i,+}(x_{2i}) = B_{1,1,1}, 
\]
and $\rr^{\TTT}_{i,+}(x_j) = I_4$ for $j \ne 2i-1, 2i$.  
We define the continuous map $\vartheta^{\TTT}_{i,+} \colon G \to \F_p$ by $g \mapsto - e_{14}(\rr^{\TTT}_{i,+}(g))$ 
and we get 
\begin{align*}
\delta \vartheta^{\TTT}_{i,+} = \Psi_3 ((\chi_{2i-1} + \chi_{2i})^{\ot 3}). 
\end{align*} 


\item Second, since $B_{1,1,1}^q=I_4$ and $[B_{1,1,1}, B_{-1,-1,-1}]=I_4$, 
we can define a continuous group homomorphism $\rr^{\TTT}_{i,-} \colon G \to U_4(\F_p)$ 
by setting 
\[
\rr^{\TTT}_{i,-}(x_{2i-1}) = B_{1,1,1}, ~ \rr^{\TTT}_{i,-}(x_{2i}) = B_{-1,-1,-1}, 
\]
and $\rr^{\TTT}_{i,-}(x_j) = I_4$ for $j \ne 2i-1, 2i$.  
We define the continuous map $\vartheta^{\TTT}_{i,-} \colon G \to \F_p$ by $g \mapsto - e_{14}(\rr^{\TTT}_{i,-}(g))$ 
and we get 
\begin{align*}
\delta\vartheta^{\TTT}_{i,-} = \Psi_3 ((\chi_{2i-1} - \chi_{2i})^{\ot 3}).
\end{align*} 

\end{itemize}


This proves the assertion for the set $\TTT$.  
It remains to consider the elements in $\DDDwr$.  
We let $2\le i \le d/2$. 
First we consider elements of the form $(\chi_{2} \ot \chi_{1} + \chi_{2i-1} \ot \chi_{2i}) \ot \chi_1$ 
and $(\chi_{2} \ot \chi_{1} + \chi_{2i-1} \ot \chi_{2i}) \ot \chi_{2i}$. 
Note that we have 
\begin{align*}
& (\chi_2 + \chi_{2i-1}) \ot (\chi_1 + \chi_{2i}) \ot (\chi_1 + \chi_{2i}) \\ 
= & ~ (\chi_{2,1,1} +  \chi_{2i-1,2i,1} ) + (\chi_{2,1,2i} + \chi_{2i-1,2i,2i}) \\
& ~ + \chi_{2,2i,1} + \chi_{2,2i,2i}  + \chi_{2i-1,1,2i}  + \chi_{2i-1,1,1} 
\end{align*}
and 
\begin{align*}
 & (\chi_2 - \chi_{2i-1}) \ot (-\chi_1 + \chi_{2i}) \ot (\chi_1 - \chi_{2i}) \\
= & ~ - (\chi_{2,1,1} +  \chi_{2i-1,2i,1} ) + (\chi_{2,1,2i} + \chi_{2i-1,2i,2i}) \\
& ~ + \chi_{2,2i,1} - \chi_{2,2i,2i}  - \chi_{2i-1,1,2i}  + \chi_{2i-1,1,1}. 
\end{align*}
We have defined cochains whose boundaries are $\Psi_3$ evaluated on each summand not in parentheses in the above sums. 
Thus, in order to find cochains whose boundaries are $\Psi_3(\chi_{2,1,1} +  \chi_{2i-1,2i,1} )$ and $\Psi_3(\chi_{2,1,2i} + \chi_{2i-1,2i,2i})$, respectively, 
it suffices to determine cochains whose boundaries are, respectively,  
\begin{align*}
\Psi_3((\chi_2 + \chi_{2i-1}) \ot (\chi_1 + \chi_{2i}) \ot (\chi_1 + \chi_{2i}))
\end{align*}
and 
\begin{align*}
\Psi_3((\chi_2 - \chi_{2i-1}) \ot (-\chi_1 + \chi_{2i}) \ot (\chi_1 - \chi_{2i})).
 \end{align*} 
\begin{itemize}

\item Since we have 
\begin{align*}
[B_{0,1,1},B_{1,0,0}]\cdot [B_{1,0,0},B_{0,1,1}] = I_4 ~ \text{in} ~ U_4(\F_p),
\end{align*}
we can define a continuous group homomorphism $\rr^{\DDDwr}_{1,2i,+} \colon G \to U_4(\F_p)$ 
by setting 
\begin{align*}
\rr^{\DDDwr}_{1,2i,+}(x_{1}) = \rr^{\DDDwr}_{1,2i,+}(x_{2i}) = B_{0,1,1}, 
\rr^{\DDDwr}_{1,2i,+}(x_{2}) = \rr^{\DDDwr}_{1,2i,+}(x_{2i-1}) = B_{1,0,0}, 
\end{align*}
and $\rr^{\DDDwr}_{1,2i,+}(x_j) = I_4$ for $j \ne 1,2,2i-1,2i$. 
We let $\vartheta^{\DDDwr}_{1,2i,+} \colon G \to \F_p$ be the continuous map defined by $g \mapsto - e_{14}(\rr^{\DDDwr}_{1,2i,+}(g))$ 
and we get 
\begin{align*}
\delta \vartheta^{\DDDwr}_{1,2i,+} = \Psi_3 ((\chi_2 + \chi_{2i-1}) \ot (\chi_1 + \chi_{2i}) \ot (\chi_1 + \chi_{2i})).
\end{align*} 

\item Next we observe that 
\begin{align*}
[B_{0,-1,1},B_{1,0,0}]\cdot [B_{-1,0,0},B_{0,1,-1}] = I_4 ~ \text{in} ~ U_4(\F_p). 
\end{align*}
Thus, we can define a continuous group homomorphism $\rr^{\DDDwr}_{1,2i,-} \colon G \to U_4(\F_p)$ 
by setting 
\begin{align*}
& \rr^{\DDDwr}_{1,2i,-} (x_{1}) = B_{0,-1,1},  \rr^{\DDDwr}_{1,2i,-}(x_{2}) = B_{1,0,0}, \\
\text{and} ~ & ~ 
\rr^{\DDDwr}_{1,2i,-}(x_{2i-1}) = B_{-1,0,0},  \rr^{\DDDwr}_{1,2i,-}(x_{2i}) = B_{0,1,-1}, 
\end{align*}
and $\rr^{\DDDw}_{i,\re,-}(x_j) = I_4$ for $j \ne 1,2,2i-1,2i$. 
We let $\vartheta^{\DDDwr}_{1,2i,-} \colon G \to \F_p$ be the continuous map defined by $g \mapsto - e_{14}(\rr^{\DDDwr}_{1,2i,-}(g))$ 
and get  
\begin{align*}
\delta \vartheta^{\DDDwr}_{1,2i,-} = \Psi_3 ((\chi_2 - \chi_{2i-1}) \ot (-\chi_1 + \chi_{2i}) \ot (\chi_1 - \chi_{2i})). 
\end{align*}

\end{itemize}


Next, we consider elements of the form $(\chi_{1} \ot \chi_{2} + \chi_{2i} \ot \chi_{2i-1}) \ot \chi_k$ for $k = 2$ and $k=2i-1$. 
We have 
\begin{align*}
& (\chi_1 + \chi_{2i}) \ot (\chi_2 + \chi_{2i-1}) \ot (\chi_2 + \chi_{2i-1}) \\ 
= & ~ (\chi_{1,2,2} + \chi_{2i,2i-1,2}) + (\chi_{1,2,2i-1} + \chi_{2i,2i-1,2i-1}) \\
& ~ + \chi_{1,2,2i-1} + \chi_{1,2i-1,2i-1} + \chi_{2i,2,2} + \chi_{2i,2,2i-1}
\end{align*}
and 
\begin{align*}
& (\chi_1 - \chi_{2i}) \ot (\chi_2 - \chi_{2i-1}) \ot (-\chi_2 + \chi_{2i-1}) \\ 
= & ~ - (\chi_{1,2,2} + \chi_{2i,2i-1,2}) + (\chi_{1,2,2i-1} + \chi_{2i,2i-1,2i-1}) \\
& ~ + \chi_{1,2,2i-1} - \chi_{1,2i-1,2i-1} + \chi_{2i,2,2} - \chi_{2i,2,2i-1}.
\end{align*}
Since we have already shown that $\kappa_3$ vanishes on all terms in the above sums 
except from 
$(\chi_{1,2,2} + \chi_{2i,2i-1,2})$ and $(\chi_{1,2,2i-1} + \chi_{2i,2i-1,2i-1})$, 
it remains to show that $\kappa_3$ vanishes on the two tensor products 
\begin{align*}
(\chi_1 + \chi_{2i}) \ot (\chi_2 + \chi_{2i-1}) \ot (\chi_2 + \chi_{2i-1}) \\
\text{and} ~ (\chi_1 - \chi_{2i}) \ot (\chi_2 - \chi_{2i-1}) \ot (-\chi_2 + \chi_{2i-1}). 
 \end{align*} 
\begin{itemize}

\item Since 
\begin{align*}
[B_{1,0,0},B_{0,1,1}]\cdot [B_{0,1,1},B_{1,0,0}] = I_4 ~ \text{in} ~ U_4(\F_p),
\end{align*}
we can define a continuous group homomorphism $\rr^{\DDDwr}_{2,2i-1,+} \colon G \to U_4(\F_p)$ 
by setting 
\begin{align*}
\rr^{\DDDwr}_{2,2i-1,+}(x_{1}) & = \rr^{\DDDwr}_{2,2i-1,+}(x_{2i}) = B_{1,0,0}, \\
\rr^{\DDDwr}_{2,2i-1,+}(x_{2}) & = \rr^{\DDDwr}_{2,2i-1,+}(x_{2i-1}) = B_{0,1,1}, 
\end{align*}
and $\rr^{\DDDwr}_{2,2i-1,+}(x_j) = I_4$ for $j \ne 1,2,2i-1,2i$. 
We let $\vartheta^{\DDDwr}_{2,2i-1i,+} \colon G \to \F_p$ be the continuous map defined by $g \mapsto - e_{14}(\rr^{\DDDwr}_{2,2i-1,+}(g))$ 
and we get 
\begin{align*}
\delta \vartheta^{\DDDwr}_{2,2i-1,+} = \Psi_3 ((\chi_1 + \chi_{2i}) \ot (\chi_2 + \chi_{2i-1}) \ot (\chi_2 + \chi_{2i-1})). 
\end{align*}

\item We observe that 
\begin{align*}
[B_{1,0,0},B_{0,1,-1}]\cdot [B_{0,-1,1},B_{-1,0,0}] = I_4 ~ \text{in} ~ U_4(\F_p). 
\end{align*}
Thus, we can define a continuous group homomorphism $\rr^{\DDDwr}_{2,2i-1,-} \colon G \to U_4(\F_p)$ 
by setting 
\begin{align*}
& \rr^{\DDDwr}_{2,2i-1,-}(x_{1}) = B_{1,0,0}, 
\rr^{\DDDwr}_{2,2i-1,-}(x_{2}) = B_{0,1,-1}, \\
\text{and} ~ & ~ 
\rr^{\DDDwr}_{2,2i-1,-}(x_{2i-1}) = B_{0,-1,1}, 
\rr^{\DDDwr}_{2,2i-1,-}(x_{2i}) = B_{-1,0,0}, 
\end{align*}
and $\rr^{\DDDwr}_{2,2i-1,-}(x_j) = I_4$ for $j \ne 1,2,2i-1,2i$.  
We let $\vartheta^{\DDDwr}_{2,2i-1,-} \colon G \to \F_p$ be the continuous map defined by $g \mapsto - e_{14}(\rr^{\DDDwr}_{2,2i-1,-}(g))$ 
and get 
\begin{align*}
\delta \vartheta^{\DDDwr}_{2,2i-1,-} = \Psi_3 ((\chi_1 - \chi_{2i}) \ot (\chi_2 - \chi_{2i-1}) \ot (-\chi_2 + \chi_{2i-1})). 
\end{align*}

\end{itemize}

This proves the assertion for the set $\DDDwr$ and finishes the proof. 
\end{proof}


Next we show that $\kappa_3$ is not trivial for $d\ge 4$. 
However, we also show how all the nonzero values of $\kappa_3$ are related. 

\begin{lemma}\label{lemma:kappa_3_does_not_vanish}
The map $\kappa_3$ does not vanish on the basis elements in the subset $\DDDwl \subset \DDDw$  ($l$ for left-multiplication) 
defined by 
\begin{align*}
\DDDwl = & ~ \{ \chi_{1} \ot (\chi_{1} \ot \chi_{2} + \chi_{2i} \ot \chi_{2i-1}), 
\chi_{2i} \ot (\chi_{1} \ot \chi_{2} + \chi_{2i} \ot \chi_{2i-1}), \\
& ~ \chi_{2} \ot (\chi_{2} \ot \chi_{1} + \chi_{2i-1} \ot \chi_{2i}), 
\chi_{2i-1} \ot (\chi_{2} \ot \chi_{1} + \chi_{2i-1} \ot \chi_{2i}) ~ \text{for} ~ 2 \le i \le d/2
\}. 
\end{align*}
We have the relations 
\begin{align*}
\kappa_3(\chi_{1} \ot (\chi_{1} \ot \chi_{2} + \chi_{2i} \ot \chi_{2i-1})) + \kappa_3(\chi_{2i} \ot (\chi_{1} \ot \chi_{2} + \chi_{2i} \ot \chi_{2i-1})) & = 0, \\
\kappa_3(\chi_{2} \ot (\chi_{2} \ot \chi_{1} + \chi_{2i-1} \ot \chi_{2i})) + \kappa_3(\chi_{2i-1} \ot (\chi_{2} \ot \chi_{1} + \chi_{2i-1} \ot \chi_{2i})) & = 0, \\
\text{and} ~ 
\kappa_3(\chi_{1} \ot (\chi_{1} \ot \chi_{2} + \chi_{2i} \ot \chi_{2i-1})) + \kappa_3(\chi_{2} \ot (\chi_{2} \ot \chi_{1} + \chi_{2i-1} \ot \chi_{2i})) & = 0 
\end{align*}
for all $2 \le i \le d/2$, and 
\begin{align}\label{eq:kappa3_relation_i_and_j}
\kappa_3(\chi_{1} \ot (\chi_{1} \ot \chi_{2} + \chi_{2i} \ot \chi_{2i-1})) = \kappa_3(\chi_{1} \ot (\chi_{1} \ot \chi_{2} + \chi_{2j} \ot \chi_{2j-1})) 
\end{align}
for $2 \le i,j \le d/2$.  
\end{lemma}
\begin{proof}
First we show that $\kappa_3$ is non-trivial. 
We compute that 
\begin{align*}
[B_{1,-1,0},B_{0,0,1}]\cdot [B_{0,0,-1},B_{-1,1,0}] = 
\begin{pmatrix} 
1 & 0 & 0 & 2 \\
0 & 1 & 0 & 0 \\
0 &  0 & 1 & 0 \\
0 & 0 & 0 & 1 
\end{pmatrix} 
~ \text{in} ~ U_4(\F_p).
\end{align*}
Thus, the assignment 
\begin{align*}
\brr(x_{1}) = B_{1,-1,0}, 
\brr(x_{2}) = B_{0,0,1}, 
\brr(x_{2i-1}) = B_{0,0,-1}, 
\brr(x_{2i}) = B_{-1,1,0}, 
\end{align*}
and $\brr(x_j) = I_4$ for $j \ne 1,2,2i-1,2i$ 
defines a continuous group homomorphism 
\begin{align*}
\brr \colon G \to \bU_4(\F_p) = U_4(\F_p)/Z
\end{align*}
which does not lift to a continuous group homomorphism 
$\rr \colon G \to U_4(\F_p)$ since $p$ is odd. 
This shows that 
$\Psi_3 ((\chi_1 - \chi_{2i}) \ot (-\chi_1 + \chi_{2i}) \ot (\chi_2 - \chi_{2i-1}))$ 
is not the boundary of a cochain in $\Ch^1$ by Corollary \ref{cor:Dwyer_consequence}. 
This shows 
\begin{align*}
\kappa_3 ((\chi_1 - \chi_{2i}) \ot (-\chi_1 + \chi_{2i}) \ot (\chi_2 - \chi_{2i-1})) \ne 0, 
\end{align*} 
i.e., $\kappa_3$ is a non-trivial map. 
We note that 
\begin{align*}
& (\chi_1 - \chi_{2i}) \ot (-\chi_1 + \chi_{2i}) \ot (\chi_2 - \chi_{2i-1}) \\
= & - (\chi_{1,1,2} + \chi_{1,2i,2i-1}) + (\chi_{2i,1,2} + \chi_{2i,2i,2i-1}) \\
& ~ + \chi_{1,1,2i-1} + \chi_{1,2i,2} - \chi_{2i,1,2i-1} - \chi_{2i,2i,2}.
\end{align*}
%
Since we have shown that $\kappa_3$ vanishes on $\chi_{1,1,2i-1}$, $\chi_{1,2i,2}$, $\chi_{2i,1,2i-1}$, and $\chi_{2i,2i,2}$, 
we get that $\kappa_3$ is non-trivial on $(\chi_{1,1,2} + \chi_{1,2i,2i-1}) - (\chi_{2i,1,2} + \chi_{2i,2i,2i-1})$. 
To prove the first claim, it now suffices to show the asserted relations. 

First we show relation \eqref{eq:kappa3_relation_i_and_j}. 
If $d \le 4$ or $i=j$, the assertion is trivial. 
So we assume $d \ge 6$ and $2 \le i < j \le d/2$. 
Since $\kappa_3$ is multilinear, it suffices to show 
\begin{align*}
\kappa_3(\chi_{1} \ot \chi_{2i} \ot \chi_{2i-1} - \chi_{1} \ot \chi_{2j} \ot \chi_{2j-1})) = 0. 
\end{align*} 
Moreover, since $\chi_1\ot (\chi_{2j-1} \ot \chi_{2j} + \chi_{2j} \ot \chi_{2j-1})\in \DDD$, 
it suffices by Lemma \ref{lemma:kappa_3_vanishes_on_many_subsets} 
to show 
\begin{align*}
\kappa_3(\chi_{1} \ot \chi_{2i} \ot \chi_{2i-1} + \chi_{1} \ot \chi_{2j-1} \ot \chi_{2j}) = 0. 
\end{align*} 
Note that  
\begin{align*}
& \chi_1 \ot (\chi_{2i} + \chi_{2j-1}) \ot (\chi_{2i-1} + \chi_{2j}) \\
= & \chi_{1} \ot \chi_{2i} \ot \chi_{2i-1} + \chi_{1} \ot \chi_{2j-1} \ot \chi_{2j} + \chi_{1} \ot \chi_{2i} \ot \chi_{2j} + \chi_{1} \ot \chi_{2j-1} \ot \chi_{2i-1}. 
\end{align*}
Since $2 \le i < j$, both $\chi_{1} \ot \chi_{2i} \ot \chi_{2j}$ and $\chi_{1} \ot \chi_{2j-1} \ot \chi_{2i-1}$ are in $\SSS$.  
It therefore suffices to show that $\Psi_3(\chi_1 \ot (\chi_{2i} + \chi_{2j-1}) \ot (\chi_{2i-1} + \chi_{2j}))$ is a coboundary. 
Since 
\begin{align*}
[B_{1,0,0},I_4] \cdot [B_{0,0,1},B_{0,1,0}] \cdot [B_{0,1,0},B_{0,0,1}] = I_4, 
\end{align*} 
we can define a continuous group homomorphism $\rr^{\DDDwl}_{i,j} \colon G \to U_4(\F_p)$ 
by setting 
\begin{align*}
 & \rr^{\DDDwl}_{i,j}(x_{1}) = B_{1,0,0}, \rr^{\DDDwl}_{i,j}(x_{2}) = I_4, \\
 & \rr^{\DDDwl}_{i,j}(x_{2i-1}) = B_{0,0,1}, \rr^{\DDDwl}_{i,j}(x_{2i}) = B_{0,1,0}, \\
\text{and} ~ & \rr^{\DDDwl}_{i,j}(x_{2j-1}) = B_{0,1,0},  \rr^{\DDDwl}_{i,j}(x_{2j}) = B_{0,0,1},
\end{align*}
and $\rr^{\DDDwl}_{i,j}(x_{k})=I_4$ for all $k \ne 1,2,2i-1,2i,2j-1,2j$. 
We define the continuous map $\vartheta^{\DDDwl}_{i,j} \colon G \to \F_p$ by $g \mapsto - e_{14}(\rr^{\DDDwl}_{i,j}(g))$, 
and we get 
\begin{align*}
\delta \vartheta^{\DDDwl}_{i,j}= \Psi_3 (\chi_1 \ot (\chi_{2i} + \chi_{2j-1}) \ot (\chi_{2i-1} + \chi_{2j})). 
\end{align*}
This proves relation \eqref{eq:kappa3_relation_i_and_j}.  


Now we prove the other relations. 

\begin{itemize}
\item We begin with elements of the form $\chi_1 \ot (\chi_{1} \ot \chi_{2} + \chi_{2i} \ot \chi_{2i-1})$ 
and $\chi_{2i} \ot (\chi_{1} \ot \chi_{2} + \chi_{2i} \ot \chi_{2i-1})$. 
We have 
\begin{align*}
& (\chi_1 + \chi_{2i}) \ot (\chi_1 + \chi_{2i}) \ot (\chi_2 + \chi_{2i-1}) \\ 
= & (\chi_{1,1,2} + \chi_{1,2i,2i-1})  + (\chi_{2i,1,2} + \chi_{2i,2i,2i-1}) \\
& ~ + \chi_{1,1,2i-1} + \chi_{1,2i,2} + \chi_{2i,1,2i-1} + \chi_{2i,2i,2}. 
\end{align*}
We have already shown that $\Psi_3$ applied to the last four summands is a coboundary. 
Thus, it suffices to show that $\Psi_3(\chi_1 + \chi_{2i})^{\ot 2} \ot (\chi_2 + \chi_{2i-1})$ is a coboundary. 
Since 
\begin{align*}
[B_{1,1,0},B_{0,0,1}]\cdot [B_{0,0,1},B_{1,1,0}] = I_4 ~ \text{in} ~ U_4(\F_p),
\end{align*}
we can define a continuous group homomorphism $\rr^{\DDDwl}_{1,2i} \colon G \to U_4(\F_p)$ 
by setting 
\begin{align*}
\rr^{\DDDwl}_{1,2i}(x_{1}) = B_{1,1,0}, \rr^{\DDDwl}_{1,2i}(x_{2}) = B_{0,1,1}, \\
\text{and} ~ \rr^{\DDDwl}_{1,2i}(x_{2i-1}) = B_{0,1,1},  \rr^{\DDDwl}_{1,2i}(x_{2i}) = B_{1,1,0}, 
\end{align*}
and $\rr^{\DDDwl}_{1,2i}(x_j) = I_4$ for $j \ne 1,2,2i-1,2i$. 
We define the continuous map $\vartheta^{\DDDwl}_{1,2i} \colon G \to \F_p$ by $g \mapsto - e_{14}(\rr^{\DDDwl}_{1,2i}(g))$, 
and we get 
\begin{align*}
\delta \vartheta^{\DDDwl}_{1,2i} = \Psi_3 ((\chi_1 + \chi_{2i}) \ot (\chi_1 + \chi_{2i}) \ot (\chi_2 + \chi_{2i-1})). 
\end{align*}

\item Next, we consider elements of the form $  \chi_2 \ot (\chi_{2} \ot \chi_{1} + \chi_{2i-1} \ot \chi_{2i})$ and of the form  
$\chi_{2i-1} \ot (\chi_{2} \ot \chi_{1} + \chi_{2i-1} \ot \chi_{2i})$. 
We have 
\begin{align*}
& (\chi_2 + \chi_{2i-1}) \ot (\chi_2 + \chi_{2i-1}) \ot (\chi_1 + \chi_{2i}) \\ 
= & (\chi_{2,2,1} + \chi_{2,2i-1,2i})   + (\chi_{2i-1,2,1} + \chi_{2i-1,2i-1,2i}) \\
& ~ + \chi_{2,2,2i} + \chi_{2,2i-1,1} + \chi_{2i-1,2,2i} + \chi_{2i-1,2i-1,1}.  
\end{align*}
We have already shown that $\Psi_3$ applied to the last four summands is a coboundary. 
Thus, it suffices to show that $\Psi_3(\chi_2 + \chi_{2i-1})^{\ot 2} \ot (\chi_1 + \chi_{2i})$ is a coboundary as well.  
Since 
\begin{align*}
[B_{0,0,1},B_{1,1,0}]\cdot [B_{1,1,0},B_{0,0,1}] = I_4 ~ \text{in} ~ U_4(\F_p),
\end{align*}
we can define a continuous group homomorphism $\rr^{\DDDwl}_{2,2i-1} \colon G \to U_4(\F_p)$ 
by setting 
\begin{align*}
 & \rr^{\DDDwl}_{2,2i-1}(x_{1}) = B_{0,0,1}, \rr^{\DDDwl}_{2,2i-1}(x_{2}) = B_{1,1,0}, \\ 
\text{and} &~ \rr^{\DDDwl}_{2,2i-1}(x_{2i-1}) = B_{1,1,0}, \rr^{\DDDwl}_{2,2i-1}(x_{2i}) = B_{0,0,1}, 
\end{align*}
and $\rr^{\DDDwl}_{2,2i-1}(x_j) = I_4$ for $j \ne 1,2,2i-1,2i$. 
We define the continuous map $\vartheta^{\DDDwl}_{2,2i-1} \colon G \to \F_p$ by $g \mapsto - e_{14}(\rr^{\DDDwl}_{2,2i-1}(g))$, 
and we get 
\begin{align*}
\delta \vartheta^{\DDDwl}_{2,2i-1} = \Psi_3 ((\chi_2 + \chi_{2i-1}) \ot (\chi_2 + \chi_{2i-1}) \ot (\chi_1 + \chi_{2i})). 
\end{align*}

\item Finally, we consider $\chi_{1} \ot (\chi_{1} \ot \chi_{2} + \chi_{2i} \ot \chi_{2i-1}) + \chi_{2} \ot (\chi_{2} \ot \chi_{1} + \chi_{2i-1} \ot \chi_{2i})$. 
We note that we have already shown that $\kappa_3$ vanishes on 
\begin{align*}
& ~ (\chi_{1,1,2} + \chi_{1,4,3}) + (\chi_{2,1,1} + \chi_{1,2,1} + \chi_{1,2i-1,2i}) \\
= & ~ (\chi_{1,1,2} + \chi_{1,2,1} + \chi_{2,1,1}) + (\chi_{1,2i-1,2i} + \chi_{1,2i,2i-1}), 
\end{align*}
as the latter is a sum of an element in $\TTT$ and an element in $\DDD$. 
Thus, if we can show that  
\begin{align*}
\kappa_3((\chi_{2,2,1} + \chi_{2,2i-1,2i}) - (\chi_{2,1,1} + \chi_{1,2,1} + \chi_{1,2i-1,2i})) = 0, 
\end{align*}
then 
\begin{align*}
\kappa_3((\chi_{1,1,2} + \chi_{1,2i,2i-1}) + (\chi_{2,2,1} + \chi_{2,2i-1,2i})) = 0. 
\end{align*}
We have 
\begin{align*}
& (\chi_1 - \chi_{2}) \ot (\chi_1 - \chi_{2} + \chi_{2i-1}) \ot (\chi_1 - \chi_{2i}) \\ 
= & ~ (\chi_{2,2,1} + \chi_{2,2i-1,2i}) - (\chi_{2,1,1} + \chi_{1,2,1} + \chi_{1,2i-1,2i}) \\
&  ~ ~ + (\chi_{1,2,2i} + \chi_{2,1,2i}) + \chi_{1,1,1} - \chi_{1,1,2i} \\
& ~ ~ + \chi_{1,3,1} - \chi_{2,2,2i} - \chi_{2,2i-1,2i}. 
\end{align*}
We have already shown that $\Psi_3$ on $\chi_{1,2,2i} + \chi_{2,1,2i}$ which is an element in $\DDD$  is a coboundary, 
and that $\Psi_3$ on the last four summands, which are elements in $\SSS$, is a coboundary.  
Thus, it suffices to show that $\Psi_3(\chi_1 - \chi_{2}) \ot (\chi_1 - \chi_{2} + \chi_{2i-1}) \ot (\chi_1 - \chi_{2i})$ is coboundary. 
Since 
\begin{align*}
[B_{1,1,1},B_{-1,-1,0}]\cdot [B_{0,1,0},B_{0,0,-1}] = I_4 ~ \text{in} ~ U_4(\F_p),
\end{align*}
we can define a continuous group homomorphism $\rr^{\DDDwl}_{1,2,i} \colon G \to U_4(\F_p)$ 
by setting 
\begin{align*}
\rr^{\DDDwl}_{1,2,i}(x_{1}) = B_{1,1,1}, 
\rr^{\DDDwl}_{1,2,i}(x_{2}) = B_{-1,-1,0}, \\
\text{and} ~ \rr^{\DDDwl}_{1,2,i}(x_{2i-1}) = B_{0,1,0}, 
\rr^{\DDDwl}_{1,2,i}(x_{2i}) = B_{0,0,-1}.
\end{align*} 
and $\rr^{\DDDwl}_{1,2,i}(x_j) = I_4$ for $j \ne 1,2,2i-1,2i$. 
We define the continuous map $\vartheta^{\DDDwl}_{1,2,i} \colon G \to \F_p$ by $g \mapsto - e_{14}(\rr^{\DDDwl}_{1,2,i}(g))$, 
and we get 
\begin{align*}
\delta \vartheta^{\DDDwl}_{1,2,i} = \Psi_3 ((\chi_1 - \chi_{2}) \ot (\chi_1 - \chi_{2} + \chi_{2i-1}) \ot (\chi_1 - \chi_{2i})). 
\end{align*}
\end{itemize}
This proves the final relation and finishes the proof of the lemma. 
\end{proof}


\subsection{The canonical class is trivial}\label{sec:kappa_3_is_hit}

While $\kappa_3$ is non-trivial as a map for $d\ge 4$, 
its class in $\HH^{3,-1}(\Hb)$ vanishes as the following lemma shows:  

\begin{lemma}\label{lemma:kappa_3_is_hit}
The map $\kappa_3$ is a coboundary in 
$(\uHom_{\F_p}(K_{\bbb}^{\bbb}(\Hb), \Hb[-1]), \dee)$.   
\end{lemma}
\begin{proof}
We need to show that we can find an $\F_p$-linear map $\pam \colon R \to H^1$ such that $\partial \pam = \kappa_3$. 
%
By Lemma \ref{lemma:kappa_3_does_not_vanish} 
and since $\kappa_3$ is $\F_p$-linear, 
we may assume that 
$\kappa_3(\chi_{1,1,2} + \chi_{1,2i,2i-1}) = \chi_1 \cup \chi_2$ in $H^2$ for all $2 \le i \le d/2$. 
To keep the notation simple, we first consider $i=2$ and $d=4$, and will then explain how to get the remaining coefficients when $d \ge 6$. 
Let $c_j^{k,n} \in \F_p$ denote the coefficients of $\chi_j$ such that 
\begin{align*}
\pam(\chi_k \ot \chi_n) =\sum_{j=1}^{4} c_j^{k,n} \chi_j
\end{align*} 
in $H^1$. 
Let $c_j^{1,2+4,3} \in \F_p$ denote the coefficients of $\chi_j$ such that 
\begin{align*}
\pam(\chi_1 \ot \chi_2 + \chi_{4} \ot \chi_{3}) = \sum_{j=1}^{4} c_j^{1,2+4,3} \chi_j
\end{align*} 
in $H^1$, 
and we use similar notation for $c_j^{2,1+3,4}$, $c_j^{1,2+2,1}$ and $c_j^{3,4+4,3}$. 
The value of $\partial \pam$ on, for example, $\chi_1 \ot \chi_1 \ot \chi_2 + \chi_1 \ot \chi_{4} \ot \chi_{3} \in \DDDw$ is then given by 
\begin{align*}
\partial \pam (\chi_1 \ot \chi_1 \ot \chi_2 + \chi_1 \ot \chi_{4} \ot \chi_{3}) & = -c_2^{1,2+4,3} (\chi_1 \cup \chi_2) - c_1^{1,1} (\chi_1 \cup \chi_2) - c_4^{1,4}(\chi_4 \cup \chi_3) \\
& = (-c_2^{1,2+4,3} - c_1^{1,1} + c_4^{1,4})(\chi_1 \cup \chi_2), 
\end{align*}
where we use the relations in $H^2$. 
By computing the effect of $\partial \pam$ on all basis elements of $K^3_3(\Hb)$, we then get that 
$\partial \pam = \kappa_3$ is satisfied if and only if the coefficients of $\pam$ satisfy the following system of linear equations: 
\begin{center}
\begin{tabular}{ r r r }
$- c_1^{1,2+4,3} - c_2^{2,2} + c_{3}^{3,2} = 0$ & $c_1^{2,1+3,4} + c_2^{2,2} - c_{3}^{2,3} = 1$  & $- c_{3}^{1,2+2,1} + c_2^{2,4} - c_1^{1,4} = 0$ \\
$- c_2^{1,2+4,3} - c_1^{1,1} + c_{4}^{1,4} = 1$ & $c_2^{2,1+3,4} + c_1^{1,1} - c_{4}^{4,1} = 0$  & $- c_{4}^{1,2+2,1} + c_2^{2,3} - c_1^{1,3} = 0$ \\
$- c_{3}^{1,2+4,3} + c_1^{4,1} - c_{4}^{4,4} = 1$ & $-c_{3}^{2,1+3,4} + c_1^{1,4} - c_{4}^{4,4} = 0$  & $- c_1^{34+43} + c_{4}^{4,2} - c_{3}^{3,2} = 0$ \\
$c_{3}^{1,2+4,3} - c_2^{2,3} + c_{3}^{3,3} = 0$ & $c_{4}^{2,1+3,4} - c_2^{3,2} + c_{3}^{3,3} = 1$  & $- c_2^{34+43} + c_{4}^{4,1} - c_{3}^{3,1} = 0$ \\
%
%
$c_1^{1,3}+ c_2^{3,2} = 0$ & $c_{3}^{3,1} + c_{4}^{1,4} = 0$ & \\ 
$c_1^{1,3}+ c_2^{2,3} = 0$ & $c_{3}^{1,3} + c_{4}^{4,1} = 0$  & \\ 
$c_1^{1,4} + c_2^{4,2} = 0$ & $c_{3}^{2,3} + c_{4}^{4,2} = 0$  & \\ 
$c_1^{4,1} + c_2^{2,4} = 0$ & $c_{3}^{3,2} + c_{4}^{2,4} = 0$  & \\ 
\end{tabular}
\end{center}
and the equation $c_j^{1,2+4,3} + c_j^{2,1+3,4} - (c_j^{1,2+2,1} + c_j^{3,4+4,3}) = 0$ for each $j=1,2,3,4$. 
A solution of the above linear system is given by 
\begin{center}
\begin{tabular}{ r r r r }
$c_1^{1,2+4,3} = 1$ &  $c_1^{2,1+3,4} = 1$ & $c_1^{1,2+2,1} = 1$ & $c_1^{3,4+4,3}  = 1$ \\
$c_2^{1,2+4,3} = 1$ & $c_2^{2,1+3,4} = 1$ & $c_2^{1,2+2,1} = 1$ & $c_2^{3,4+4,3}  = 1$ \\
$c_{3}^{1,2+4,3} = -1$ & $c_{3}^{2,1+3,4} = -1$ & $c_{3}^{1,2+2,1} = -1$ & $c_{3}^{3,4+4,3}  = -1$ \\
$c_{4}^{1,2+4,3} = -1$ & $c_{4}^{2,1+3,4} = -1$ & $c_{4}^{1,2+2,1} = -1$ & $c_{4}^{3,4+4,3} = -1$  \\
$c_1^{1,3} = 1$ & $c_2^{2,3} = 0$ & $c_{3}^{1,3} = 0$ & $c_{4}^{1,4} = 1$ \\
$c_1^{1,4} = 0$ & $c_2^{3,2} = -1$ & $c_{3}^{3,1} = -1$ & $c_{4}^{4,1} = 0$ \\ 
$c_1^{3,1} = 0$ & $c_2^{2,4} = -1$ & $c_{3}^{2,3} = -1$ & $c_{4}^{2,4} = 0$ \\
$c_1^{4,1} = 1$ & $c_2^{4,2} = 0$ & $c_{3}^{3,2} = 0$ & $c_{4}^{4,2} = 1$ \\
$c_1^{1,1} = -1$ & $c_2^{2,2} = -1$ & $c_{3}^{3,3} = 1$ & $c_{4}^{4,4} = 1$ 
\end{tabular}
\end{center}
and we set the other coefficients to be zero. 
When $d\ge 6$, for each $i \ge 2$, we get a similar system of equations with coefficients for terms only involving indices $1,2,2i-1,2i$. 
The linear systems for two different values of $i$ are independent of each other except for coefficients which are independent of $i$, 
i.e., whose indices only involve $1$ and $2$. 
However, by relation \eqref{eq:kappa3_relation_i_and_j} together with the other identities of Lemma \ref{lemma:kappa_3_does_not_vanish}, 
after replacing $3$ by $2i-1$ and $4$ by $2i$ in each occurrence, 
the same values of the coefficients solve the corresponding systems of equations. 
In particular, the values of coefficients only involving indices $1$ and $2$ remain the same for each system corresponding to $i$, 
i..e., for each system, our solution satisfies $c_1^{1,1} = -1$, $c_2^{2,2} = -1$, $c_1^{1,2+2,1} = 1$, and $c_2^{1,2+2,1} = 1$. 
%
This provides the desired linear map $\pam \colon R \to H^1$ such that $\partial \pam = \kappa_3$. 
\end{proof} 

By Proposition \ref{prop:canonical_class_of_Demushkin_group_construction}  and Lemma \ref{lemma:kappa_3_is_hit}, 
this shows that the canonical class of $G$ in $\HH^{3,-1}(\Hb)$ vanishes. 
By Theorem \ref{thm:existence_of_obstruction_class} and Proposition \ref{prop:canonical_class_via_Koszul_complex}, this proves that 
$G$ is $A_3$-formal. 
This concludes the proof of Theorem \ref{thm:Demushkin_groups_are_formal}. \qed

%


\end{document}